\documentclass[11pt,A4paper]{amsart}
\usepackage{amsfonts, amsmath, amssymb, amscd, amsthm, graphicx, mathrsfs, wasysym, setspace, mdwlist, color, mathtools}
\usepackage[all]{xy}
\usepackage{hyperref}
\hypersetup{
  pdfborder={0 0 0},
  colorlinks   = true, 
  urlcolor     = blue, 
  linkcolor    = blue, 
  citecolor   = red 
}

\newtheorem{thm}{Theorem}[section]
\newtheorem{prop}[thm]{Proposition}
\newtheorem{lem}[thm]{Lemma}
\newtheorem{cor}[thm]{Corollary}

\theoremstyle{definition}
\newtheorem{nt}[thm]{Notation}

\newtheorem{defn}[thm]{Definition}

\theoremstyle{remark}

\newtheorem{rem}[thm]{Remark}

\newcommand{\gen}[1]{\left\langle#1\right\rangle}

\newcommand{\prs}[2]{\gen{#1\parallel #2}}

\newcommand{\pc}{{\rm Pc}}

\renewcommand{\d}{{\rm d}}

\newcommand{\ga}{\Gamma}

\newcommand{\e}{\varepsilon }

\renewcommand{\kappa }{\varkappa}

\newcommand{\G}{\ga(G, X\sqcup \mathcal H)}
\newcommand{\Hl}{\{H_\lambda\}_{\lambda \in \Lambda}}

\renewcommand{\le}{\leqslant}
\renewcommand{\ge}{\geqslant}
\renewcommand{\phi}{\varphi}
\newcommand{\vr}{\mathcal{VR}}
\newcommand{\avr}{\mathcal{AVR}}

\newcommand{\cH}{\mathcal{H}}
\newcommand{\cQ}{\mathcal{Q}}

\newcommand{\cS}{\mathcal{S}}
\newcommand{\cT}{\mathcal{T}}
\newcommand{\cW}{\mathcal{W}}
\newcommand{\cM}{\mathcal{M}}

\newcommand{\N }{\mathbb N}
\newcommand{\Z }{\mathbb Z}

\newcommand{\LWPD}{\mathcal{L}_{\text{WPD}}(G,\cS)}

\newcommand{\LWPDHP}{\mathcal{L}^{+}_{\text{WPD}}(H,\cS)}
\newcommand{\LWPDP}{\mathcal{L}^{+}_{\text{WPD}}(G,\cS)}

\newcommand{\apg}{\stackrel{G}{\approx}}
\newcommand{\napg}{\stackrel{G}{\not\approx}}

\newcommand{\h}{\hookrightarrow_{h}}

\newcommand{\Lab}{{\bf Lab}}
\newcommand{\diam}{{\rm diam}}

\newcommand{\nei}[2]{\mathcal{N}_{#2}(#1)}
\newcommand{\B}[2]{\mathbb{B}(#1 , #2)}
\newcommand{\llangle}{\langle\hspace{-2.5pt}\langle}
\newcommand{\rrangle}{\rangle\hspace{-2.5pt}\rangle}

\newcommand{\tame}{loxodromic WPD }
\newcommand{\Tame}{Loxodromic WPD }

\def\coloneq{\mathrel{\mathop\mathchar"303A}\mkern-1.2mu=}

\renewcommand{\le }{\leq}
\renewcommand{\ge }{\geq}

\begin{document}

\title[Commensurating endomorphisms of acylindrically hyperbolic groups]{Commensurating endomorphisms of acylindrically hyperbolic groups and applications}

\author{Yago Antol\'{i}n}
\address[Yago Antol\'{i}n]{1326 Stevenson Center,
Vanderbilt University,
Nashville, TN 37240, USA.}
\email{yago.anpi@gmail.com}

\author{Ashot Minasyan}
\address[Ashot Minasyan]{Mathematical Sciences,
University of Southampton, Highfield, Southampton, SO17 1BJ, United
Kingdom.}
\email{aminasyan@gmail.com}

\author{Alessandro Sisto}
\address[Alessandro Sisto]{Department of Mathematics, ETH Zurich, 8092 Zurich, Switzerland.}
\email{sisto@math.ethz.ch}

\thanks{The work of the first two authors was partially supported by the EPSRC grant EP/H032428/1. The work of the first author was partially supported by  the MCI (Spain) through project MTM2011-25955.
The work of the third author was partially funded by the EPSRC grant ``Geometric and analytic aspects of infinite groups''.}

\begin{abstract} We prove that the outer automorphism group $Out(G)$ is residually finite when the group $G$ is
virtually compact special (in the sense of Haglund and Wise) or when $G$ is  isomorphic to the fundamental group of some compact $3$-manifold.

To prove these results we characterize commensurating endomorphisms of acylindrically hyperbolic groups.
An endomorphism $\phi$ of a group $G$ is said to be commensurating, if for every $g \in G$ some non-zero power of $\phi(g)$ is conjugate to a non-zero power of $g$.
Given an acylindrically hyperbolic group $G$, we show that any commensurating endomorphism of $G$ is inner modulo a small perturbation.
This generalizes a theorem of Minasyan and Osin, which provided a similar statement in the case when $G$ is relatively hyperbolic. We then use this result
to study  pointwise inner and normal endomorphisms of acylindrically hyperbolic groups.
\end{abstract}


\setcounter{tocdepth}{1}
\maketitle
\vspace{-0.5cm}
\tableofcontents

\section{Introduction}
A group $G$ is said to be \emph{residually finite} if for any distinct $x,y \in G$ there is a finite group $Q$ and a homomorphism $\psi:G \to Q$ such that $\psi(x) \neq \psi(y)$
in $Q$. Informally speaking this means that $G$ can be approximated by its finite quotients, in which case these quotients can be used to study the group $G$.
For example, two classical theorems of Mal'cev state that finitely presented residually finite groups have solvable word problem \cite{Malc1},
and finitely generated residually finite groups are Hopfian \cite{Malc2}.
Thus residual finiteness is a very basic property, so given any infinite group $G$, one of the first questions we could ask about $G$ is whether it is residually finite.

The goal of this work is to prove residual finiteness of the outer automorphism group $Out(G)=Aut(G)/Inn(G)$, where $G$ belongs to one of the
following large classes of groups:
\begin{itemize}
  \item the class of virtually compact special groups (in the sense of F. Haglund and D. Wise \cite{Hag-Wise}),
  \item the class of fundamental groups of  compact $3$-manifolds.
\end{itemize}

Before formulating the main results, let us recall some background of the problem.
A  well-known theorem of G. Baumslag \cite{Baumslag} asserts that for a finitely generated residually finite group $G$ its group of automorphisms,
$Aut(G)$, is also residually finite. In \cite{MO} the second author and Osin showed that in the case when $G$ has infinitely many ends, the same assumptions on $G$ (finite generation and residual finiteness) also imply that the group of outer automorphisms $Out(G)$ is residually finite. However, this `outer' version of Baumslag's
result may not hold if $G$ is $1$-ended. Indeed, Bumagin and Wise \cite{Bum-Wise} proved that for any finitely presented group $S$ there is a finitely generated
residually finite ($1$-ended) group $G$ such that  $Out(G) \cong S$.

A classical criterion for establishing residual finiteness of outer automorphism groups was discovered by Grossman \cite{Grossman}.
This criterion imposes stronger assumptions on the group: $G$ must be conjugacy separable and any pointwise inner automorphism of $G$ must be inner.
Recall that a group is {\it conjugacy separable} if for any pair of non-conjugate elements $x,y \in G$ there exists a finite quotient $Q$, of $G$, in which the images of $x,y$ are non-conjugate. An automorphism $\alpha \in Aut(G)$ is said to be {\it pointwise inner} if $\alpha(g)$ is conjugate to $g$ for each $g \in G$.
The set of all pointwise inner automorphism, $Aut_{pi}(G)$, forms a normal subgroup of $Aut(G)$.

In \cite{Grossman} Grossman proved the following theorem: if $G$ is a finitely generated conjugacy separable group such that
$Aut_{pi}(G)=Inn(G)$ then $Out(G)$ is residually finite. Unfortunately, it is usually hard to show that a given group is conjugacy separable, as it is  a much more  delicate
condition than residual finiteness (for example conjugacy separability may not pass to subgroups or overgroups of finite index -- see \cite{MM,Goryaga}).

One class of groups for which conjugacy separability is known is the class $\mathcal{VR}$, consisting of virtual retracts of finitely generated right angled Artin
groups -- see \cite{Minasyan_hcs}.
Let $\mathcal{AVR}$ denote the class of groups which contain a finite index subgroup from the class $\mathcal{VR}$. The significance of the class
$\mathcal{AVR}$ can be seen from the work of Haglund and Wise \cite{Hag-Wise}, who proved that every virtually compact special group belongs to this class
(recall that a group $G$ is said to be \emph{virtually compact special} if $G$ contains a finite index subgroup which is the fundamental group of a
compact special cube complex in the sense of \cite{Hag-Wise}). The list of virtually compact special groups is quite large and includes most Coxeter groups,
$1$-relator groups with torsion and finitely generated fully residually free (limit) groups -- see \cite{Hag-Wise-cox,Wise-QH}.
Therefore the following theorem covers a wide range of groups:

\begin{thm}\label{thm:avr} For any group $G \in \mathcal{AVR}$ the  group $Out(G)$ is residually finite.
\end{thm}

As discussed above, Theorem \ref{thm:avr} together with results from \cite{Hag-Wise} yields
\begin{cor} If $G$ is  virtually compact special then $Out(G)$ is residually finite.
\end{cor}

It is worth mentioning that residual finiteness of $Out(G)$, when $G$ is itself a finitely generated right angled Artin group, was proved by the second author in \cite{Minasyan_hcs} and, independently,
by Charney and Vogtmann in \cite{Char-Vogt}.
On the other hand, there exist finitely generated groups $H$ such that $H$ is a subgroup of some finitely generated right angled Artin group and
$Out(H)$ is not residually finite. Such examples can be easily found using the modification of the Rips's construction proposed by Haglund and Wise in \cite{Hag-Wise}.

It is easy to produce examples of groups from $\mathcal{AVR}$ which possess pointwise inner
automorphisms that are not inner (one can simply take the  direct product of the free group of rank $2$ with a finite group $M$ for which $Aut_{pi}(M)\neq Inn(M)$ --
see \cite{Burnside} for a construction of such finite groups). Also, it is unknown whether all groups from $\avr$ are conjugacy separable, thus to prove Theorem \ref{thm:avr}
one cannot simply apply Grossman's criterion, and a more elaborate approach is required.

The original application of Grossman's criterion was the proof that $Out(\pi_1(\Sigma))$ is residually finite for any compact orientable surface $\Sigma$ (see \cite{Grossman}).
Our second theorem extends this result to $Out(\pi_1(\cM))$, where $\cM$ is a compact $3$-manifold:

\begin{thm} \label{thm:3-man} Let $G$ be a group containing a finite index subgroup that is isomorphic to the fundamental group of
some compact $3$-manifold $\cM$. Then  $Out(G)$ is residually finite.
\end{thm}

For fundamental groups of Seifert fibered $3$-manifolds (with two exceptions), the residual finiteness of outer automorphism groups was proved by Allenby, Kim and Tang \cite{AKT-1,AKT-2}. Evidently, Theorem \ref{thm:3-man} cannot be further generalized to $4$-dimensional manifolds, as it is well-known that for any finitely presented group $G$
there is a closed $4$-manifold $\cM$ such that $G \cong \pi_1(\cM)$.

It is well-known that, for a manifold $\cM$, the group $Out(\pi_1(\cM))$ is closely related to the \emph{mapping class group} (i.e., the group of isotopy classes of self-homeomorphisms)
$\mathcal{H}(\cM)$ of $\cM$. For example,
Waldhausen \cite{Wald} proved that if $\cM$ is an irreducible orientable Haken $3$-manifold with incompressible boundary such that $\cM$ is not a line bundle, then
$\mathcal{H}(\cM)$ embeds into $Out(\pi_1(\cM))$. A similar statement when $\cM$ is non-orientable (but still Haken and $\mathbb{P}^2$-irreducible) was proved in \cite{Heil}. If $\cM$ is not irreducible then
the natural homomorphism $\mathcal{H}(\cM) \to Out(\pi_1(\cM))$ will not, in general, be injective -- see \cite{McCullough}.

Thus Theorem \ref{thm:3-man} yields

\begin{cor} Suppose that $\cM$ is a compact irreducible orientable Haken $3$-manifold with incompressible boundary that is not a line bundle. Then the mapping class group $\mathcal{H}(\cM)$ is residually finite.
\end{cor}

Thanks to a recent result of Hamilton, Wilton and Zalesskii \cite{HWZ}, stating that the fundamental group of any
orientable compact $3$-manifold is conjugacy separable, we can prove Theorem~\ref{thm:3-man} using an approach
which is similar to the one employed in Theorem \ref{thm:avr}.
Namely, in both theorems we use the techniques from geometric group theory, of groups acting on hyperbolic spaces, to prove a strong version of the fact
that pointwise inner automorphisms are inner, which constitutes the second ingredient of Grossman's criterion.

\subsection{Details of the proof}
Let us now discuss how the two theorems above are proved in more detail.

A group $G$ is called \emph{acylindrically hyperbolic} if it admits a non-elementary acylindrical action on a hyperbolic metric space -- see Subsection \ref{subsec:acyl_hyp}.
This definition was originally proposed by D. Osin in \cite{Osin-acyl},
where he proved that the class of such groups coincides with other large classes, previously studied  by Bestvina and Fujiwara \cite{BestvinaFujiwara},
Dahmani, Guirardel and Osin \cite{DGO}, Hamenst\" adt \cite{Ham} and the third author \cite{Sisto-contr}. The class of acylindrically hyperbolic groups is rather extensive:
it includes all non-elementary relatively hyperbolic groups, non-(virtually cyclic) groups acting properly on proper CAT($0$)-spaces with at least one rank $1$-element (see \cite{Osin-acyl}),
mapping class groups of compact surfaces of genus at least $1$, outer automorphism groups of free groups of rank at least $2$ (see \cite{DGO}),
many groups acting on simplicial trees (see \cite{MO2}), etc. In \cite{MO2} the second author and Osin proved that for any compact irreducible
$3$-manifold $\cM$, $\pi_1(\cM)$ is either acylindrically hyperbolic or virtually polycyclic, or $\cM$ is Seifert fibered.

Two elements $g,h$ of a group $G$ are said to be {\it commensurable} if there are $z\in G$ and $n,m\in \mathbb{Z}\setminus \{0\}$ such that $g^n=zh^mz^{-1}$ in $G$. In this we case we will write
$g\stackrel{G}{\approx} h$. Otherwise, if $g$ and $h$ are non-commensurable, we will write
$g\stackrel{G}{\not\approx} h$.
Note that commensurability is an equivalence relation on the set of elements of $G$.
Given a subgroup $H$ of a group $G$ and a homomorphism $\phi \colon H \to G$, we will say that $\phi$ is {\it commensurating  }
if $h\apg \phi(h)$ for all $h\in H$.

Commensurating homomorphisms were introduced and studied by the second author and Osin in the context of relatively hyperbolic groups in \cite{MO}.
To prove Theorems \ref{thm:avr} and \ref{thm:3-man} we study such homomorphisms for
an acylindrically hyperbolic group $G$. Our main technical result (Theorem \ref{thm:comm} in Section~\ref{sec:comm_hom}), generalizing the work from  \cite{MO},
states that if $H$ is a
sufficiently large subgroup of $G$, then every commensurating homomorphism $H \to G$ is induced by an inner
automorphism of $G$ modulo a small perturbation (which  disappears when one restricts to some finite index subgroup of $H$).
In Theorem \ref{thm:comm_end} we apply this result to the case when $H=G$ to obtain a characterization of commensurating endomorphisms of an acylindrically hyperbolic group. In particular, we get the following

\begin{cor}\label{cor:conj_aut} For any acylindrically hyperbolic group $G$, $Inn(G)$ has finite index in $Aut_{pi}(G)$.
Moreover, if $G$ has no non-trivial finite normal subgroups then $Aut_{pi}(G)=Inn(G)$.
\end{cor}

Sections \ref{sec:adding} -- \ref{sec:techn} of the paper develop the theory of acylindrically hyperbolic groups, which is necessary to prove the main technical theorem.
In particular, in Section  \ref{sec:adding} we investigate the necessary and sufficient conditions
for adding a subgroup to an existing family of hyperbolically embedded subgroups,
generalizing Osin's work from \cite{O06} (this has also been independently done by M. Hull \cite{Hull}).

The proof of Theorem \ref{thm:avr} uses the following quite general statement:

\begin{thm}\label{thm:comm_end-raag}
Let $H$ be a non-abelian subgroup of a finitely generated right angled Artin group. Then every commensurating  endomorphism $\phi\colon H \to H$ is an inner automorphism of $H$.
\end{thm}

While Theorem \ref{thm:comm_end-raag} employs the characterization of commensurating endomorphisms of acylindrically hyperbolic groups from Theorem \ref{thm:comm_end},
it  is not a straightforward consequence of it, as not all subgroups of right angled Artin groups are acylindrically hyperbolic. The proof of Theorem \ref{thm:comm_end-raag}
occupies Section \ref{sec:RAAG-end}.

The final ingredient in the proofs of Theorems \ref{thm:avr} and \ref{thm:3-man} is a new criterion
for residual finiteness of outer automorphism groups -- see Proposition \ref{prop:new_crit} in Section \ref{sec:crit}. This criterion could be of independent interest:
for example, it gives a short proof of the fact that $Out(\pi_1(\mathcal{M}))$ is residually finite for any Seifert fibered space $\mathcal M$, which was conjectured
by Allenby, Kim and Tang in \cite{AKT-2} -- see Lemma~\ref{lem:Seifert}.

The last application of Theorem \ref{thm:comm_end} that we discuss here concerns normal endomorphisms. We will say that an endomorphism
$\phi \colon G \to G$ is  \emph{normal} if $\phi(N) \subseteq N$ for every normal subgroup $N \lhd G$. Normal automorphisms (with a slightly more restrictive  definition
requiring that $\phi(N)=N$ for all $N \lhd G$) have been studied by several authors before. For instance, Lubotzky \cite{Lub} showed that all normal automorphisms of free groups are inner.
A similar statement was proved for non-trivial free products \cite{Obr} and non-elementary relatively hyperbolic groups with trivial finite radical \cite{MO}; see \cite{MO} for more
results and references.

It is known that every acylindrically hyperbolic group $G$ contains a unique maximal finite normal subgroup (see \cite[Thm. 2.24]{DGO} or Lemma \ref{lem:E_G(H)} below).
This subgroup, sometimes called the \emph{finite radical} of $G$, will be denoted by $E_G(G)$ ($K(G)$ is the notation used in \cite{DGO}), in line with Lemma \ref{lem:E_G(H)} below.
Combining Theorem \ref{thm:comm_end} with the theory of Dehn fillings for hyperbolically embedded subgroups, developed by Dahmani, Guirardel and Osin in \cite{DGO}
we show that almost all normal endomorphisms of acylindrically hyperbolic groups are commensurating, and so their structure is described by Theorem \ref{thm:comm_end}.

\begin{thm}\label{thm:norm_end} Let $G$ be an acylindrically hyperbolic group and let $\phi\colon G \to G$ be a normal endomorphism.
Then either $\phi(G) \subseteq E_G(G)$ or $\phi$ is commensurating. In particular, if $E_G(G)=\{1\}$ and $\phi(G) \neq \{1\}$ then
$\phi$ is an inner automorphism of $G$.
\end{thm}

\smallskip
\noindent {\bf Acknowledgements.} The authors would like to thank Denis Osin for helpful discussions. We are also grateful to Henry Wilton for suggesting and discussing the applications
of the main result to $3$-manifold groups.

\section{Preliminaries}
\subsection{Notation}\label{subsec:not}
In this subsection we fix the notation and recall some basic concepts that will be used throughout the paper.

Let $(\cS,\d)$ be a metric space. Given a subset  $A\subseteq \cS$ and $\e>0$, we denote by $\nei{A}{\e}$ the closed
$\e$-neighborhood of $A$, i.e., $$\nei{A}{\e}=\{x\in\cS \mid \d(x,A) \le \e\} .$$ Similarly, we denote by $\B{x}{\e}=\{s\in \cS \mid \d(x,s) \le \e\}$, the closed ball of center $x\in \cS$ and radius $\e$.

Recall that for $A,B\subseteq \cS$, the Hausdorff distance is given by $$\d_{Hau}(A,B)\coloneq \max \left\{ \adjustlimits\sup_{a\in A} \inf_{b\in B} \d(a,b),  \adjustlimits\sup_{b\in B} \inf_{a\in A} \d(a,b)\right\}.$$

An isometric action of a group $G$ on $(\cS,\d)$ is {\it metrically proper } if for any bounded subset $B\subseteq \cS$,  the set
$\{g\in G \mid B \cap g\circ B \neq \emptyset\}$
is finite.

Recall that a {\it path} in $\cS$ is a continuous function $p\colon  [0,1]\to \cS$, and  the {\it length } of  $p$ is
$$\ell( p )=\sup_{0=t_0\leq t_1\leq \dots\leq t_n=1} \sum_{i=0}^{n-1} \d(p(t_i),p(t_{i+1})). $$
The path $p$ is rectifiable if $\ell( p )$ is finite. We denote by $p_{-}$ and $p_{+}$ the initial and the final points of $p$.

The metric $\d$ is a {\it length metric} if for every $x,y\in{\cS}$, $$\d(x,y)=\inf \{ \ell ( p ) \mid  \text{$p$ a rectifiable path from $x$ to $y$} \}.$$
If the metric $\d$ is a length metric, $(\cS,\d)$ is called a {\it length space}. If the infinum above is always realized (i.e., for any $x, y \in \cS$ there is a rectifiable
path $p$ with $\ell(p)=\d(x,y)$), then $(\cS,\d)$ is said to be a \emph{geodesic metric space}.

Let $G$ be a group. Suppose that $X$ is a set equipped with a map $\pi\colon X\to G$. We will say that $G$ generated by $X$ if
$G=\langle \pi(X) \rangle$. The set $X$ will be called \emph{symmetric} if  $\pi(X)=\pi(X)^{-1}$ in $G$.
In this case one can define the {\it Cayley graph} $\Gamma(G,X,\pi)$, of $G$ with respect to $X$ and $\pi$, as the graph with vertex set $G$ and
edge set $G\times X$, where the initial vertex of $(g,x)$ is $g$ and the final vertex is $g\pi(x)$.
Note that this definition allows the Cayley graph to have multiple edges joining two vertices.
When the map $\pi$ is clear we will abuse the notation and simply write $\ga(G,X)$ instead of  $\Gamma(G,X,\pi)$.
Given a word $U$ over $X$, $\|U\|$ will denote the length of $U$. For any other word $V$ over $X$, we will write $U \equiv V$ to denote the graphical (letter-by-letter) equality between words $U$ and $V$.

If $X$ generates $G$ and $g \in G$ then $|g|_X$ will denote the length of a shortest word over $X$ representing $g$ in $G$.
We  will denote by $\d_X$ the {\it graph metric} on $\ga(G,X)$, that is $\d_X$ is the metric of the geometric realization of the graph
where all the edges are isometric to the unit interval. Thus if $g,h \in G$ then $\d_X(g,h)=|g^{-1}h|_X$.

In the context of graphs, we will consider combinatorial paths.
A {\it combinatorial path} in $\ga(G,X)$ is  a formal sequence $p=e_1,\dots, e_n$ where $e_1,\dots, e_n$ are edges  and the
initial vertex of $e_i$ is  the terminal vertex of $e_{i-1}$, $i=2,\dots,n$. In this case, the \emph{length} $\ell( p )$ of $p$ is $n$; $p^{-1}$ will be the path, inverse to $p$
(i.e., $p^{-1}=e_n^{-1},\dots,e_1^{-1}$, where $e_j^{-1}$ is the the edge inverse to $e_j$). Furthermore, $p_-$ and $p_+$ will denote the initial and the terminal vertices of $p$ respectively.
If $p$ is a combinatorial path in a labelled directed graph (e.g., a Cayley graph), we will use $\Lab ( p )$ to denote its label.

Given a subgroup $H$ of a group $G$ and a subset $E \subseteq G$, $C_H(E):=\{h \in H \mid he=eh, \,\,\forall\,e \in E\}$ will denote the centralizer of $E$ in $H$, and
$N_G(H):=\{g \in G \mid gHg^{-1}=H\}$ will denote the normalizer of $H$ in $G$. We will also use $\llangle E \rrangle^G\lhd G$ to denote the normal closure of $E$ in $G$.

\subsection{Hyperbolic spaces}
A geodesic metric space $({\cS},\d)$ is called {\it $\delta$-hyperbolic} if for any geodesic triangle, every side of the triangle is contained in the $\delta$-neighborhood of the union of the other two sides.
A metric space is said to be {\it hyperbolic} if it is  geodesic and $\delta$-hyperbolic  for some $\delta\geq 0$.

A subset $A$ of $\cS$ is {\it $\sigma$-quasi-convex}, for some $\sigma \geq 0$, if for every geodesic path $p$ in $\cS$ with $p_-,p_+\in A$, one has $p\subset \nei{A}{\sigma}$.
A set is {\it quasi-convex} if it is $\sigma$-quasi-convex for some $\sigma \geq 0$.

The following observation is an easy exercise on the definitions:

\begin{rem} \label{rem:qc_orbits} Suppose that $Q$ is a subgroup of a group $G$ acting by isometries on some $\delta$-hyperbolic space $(\cS,\d)$. If the orbit $Q \circ s$ is $\sigma$-quasi-convex for some
$ s \in \cS$ and $\sigma \ge 0$ then for any $s' \in \cS$ the orbit $Q \circ s'$ is $\sigma'$-quasi-convex, where $\sigma':=2\delta+ 2\d(s,s')+\sigma$.
\end{rem}

If $(\mathcal{T},{\rm e})$ is another metric space, then a map $f \colon \mathcal{T} \to \cS$ is a \emph{quasi-isometric embedding} if there exist $\lambda \ge 1$ and $c \ge 0$ such that
$$\frac{1}{\lambda} {\rm e}(x,y) -c \le \d(f(x),f(y)) \le \lambda {\rm e}(x,y)+c \mbox{ for all } x,y \in \mathcal{T}.$$
If the quasi-isometric embedding $f$ is \emph{quasi-surjective}, i.e., $\cS= \nei{f(\mathcal{T})}{\e}$ for some $\e \ge 0$, then $f$ is said to be a quasi-isometry.
The spaces $(\mathcal{T},{\rm e})$ and $(\cS,\d)$ are quasi-isometric if there exists a quasi-isometry $f: \mathcal{T}\to \cS$.

We will say that a path $p$ in $({\cS},\d)$ is a $(\lambda, c)$-{\it quasi-geodesic} for some $\lambda \geq 1, c\geq 0$ if for any subpath $q$ of $p$ we have
$$\ell ( q )\leq \lambda \d(q_{-},q_{+})+c,$$ where $\ell(q)$ is the length of $q$ and $q_-$, $q_+$ are the initial and terminal points of $q$ respectively.

We now collect a series of well known facts about quasi-geodesic paths in hyperbolic spaces.

\begin{lem}[{\cite[III.H.1.7]{BH}}]\label{qg}
For any $\delta \geq 0$, $\lambda \geq 1$, $c\geq 0$, there exists a constant $\kappa =\kappa(\delta,\lambda,c)\geq 0$ such that in a $\delta$-hyperbolic space
any two $(\lambda, c)$-quasi-geodesics with the same endpoints belong to the closed $\kappa$-neighborhoods of each other.
\end{lem}

Two paths $p, q$ in a metric space ${(\cS,\d)}$ are called $k$-{\it connected}, if
$$\max\{\d(p_{-}, q_{-}), \d(p_{+}, q_{+})\} \leq k.$$ The paths $p$ and $q$ are {\it $k$-close} for some $k >0$ if $p$ is $k$-connected with either $q$ or $q^{-1}$.

The next lemma is a simplification of Lemma 25 from \cite{Ols92}. Basically it says that if some sides of a geodesic polygon are much longer than the rest, then
there is a pair of the long sides having sufficiently long subsegments which travel close to each other.

\begin{lem}\label{Ols} Let $P$ be a geodesic $n$-gon in a $\delta$-hyperbolic space whose sides $p_1,\dots, p_n$ are divided into two subsets $S$, $T$.
Denote the total length of all sides from $S$ by $\sigma$ and the total length of all sides from $T$ by $\rho$, and assume that $\sigma \ge \max\{10^3 a n, 10^3\rho\}$ for some $a\ge 30\delta $.
Then there are two distinct sides $p_i, p_j \in S$,  and $13\delta $-close
subsegments $u$ and $v$ of $p_i$ and $p_j$, respectively, such that $\min \{ \ell(u), \ell(v)\} > a$.
\end{lem}

For our purposes we need the following version of the {\v S}varc-Milnor Lemma.
\begin{lem} [The {\v S}varc-Milnor Lemma] \label{lem:svarc-milnor} Let $({\cS},\d)$ be a length space. Suppose that a group $G$ acts by isometries on ${\cS}$ and the action is cobounded.
Then there exists a symmetric generating set $X$ of $G$ such that for any $s\in {\cS}$, the map $g \mapsto g\circ s$ is a quasi-isometry from $(G,\d _X)$ to $({\cS},\d)$.

Moreover if the action is metrically proper then $X$ can be chosen to be finite.
\end{lem}
\begin{proof}
This is proved in \cite[I.8.19]{BH} with the assumption that the action is metrically proper, which is only used to conclude that $X$ is finite.
\end{proof}

\begin{lem}\label{lem:proper}
If $G$ acts by isometries on a hyperbolic space $(\cS,\d)$, $s \in \cS$ and $Q \leqslant G$ then the following are equivalent:
\begin{enumerate}
\item[{\rm (1)}] The orbit $Q \circ s$ is  quasi-convex and the induced action of $Q$ on $\cS$ is metrically proper;
\item[{\rm (2)}] $Q$ is generated by a finite set of elements $Y$ and  there exist $\mu\geq 1, c \geq 0$ s.t. $|g|_Y \leq \mu \d(s,g\circ s)+c$ for all $g \in Q.$
\end{enumerate}
\end{lem}
\begin{proof}
Assume (1). Let $s\in {\cS}$ be such that $Q\circ s$ is $\sigma$-quasi-convex for some $\sigma\geq 0$.  Let $\d_Q$ be the induced length metric on $\nei{Q\circ s}{\sigma}$, i.e.,  $\d_Q(x,y)$ is the infimum of  the
lengths of all the paths from $x$ to $y$ contained in $\nei{Q\circ s}{\sigma}$.
Since $Q\circ s$ is $\sigma$-quasi-convex in ${\cS}$, the inclusion map $(\nei{Q\circ s}{\sigma}, \d_Q)\to ({\cS},\d)$  is a quasi-isometric embedding.

Note that the action of $Q$ on $\nei{Q\circ s}{\sigma}$  is metrically proper, by isometries and cobounded. Hence {\v S}varc-Milnor lemma (Lemma \ref{lem:svarc-milnor})
 implies the existence of some finite generating set $Y$ of $Q$ such that $(\nei{Q\circ s}{\sigma},\d_{Q})$ and $(Q, \d_{Y})$ are quasi-isometric. Since the natural inclusion of
 $(Q\circ s,\d_{Q})$ into $({\cS},\d)$ is a quasi-isometric embedding,  there exist $\mu\geq 1$ and $c\geq 0$ such that  $|g|_{Y}\leq \mu \d(s,g\circ s)+c$ for all $g\in Q$, implying that (2) holds.

Now, assume (2).
For every $R>0$ we have $|(Q\circ s) \cap \B{s}{R}|<\infty$, so that the induced action of $Q$ on $\cS$ is metrically proper.

To prove  that the orbit $Q \circ s$ is $\sigma$-quasi-convex, for some $\sigma \ge 0$, take any geodesic path $p$ in ${\cS}$ with endpoints in $Q\circ s$. We are going to show that $p \subseteq \nei{Q \circ s}{\sigma}$, where $\sigma$ will be determined later.
Since $Q$ is a group acting by isometries on $\cS$, without loss of generality we can assume that $p_-=s$.

Define $m:=\max \{\d(s,y\circ s) \mid y\in Y\}$ and choose $g\in Q$ with $g\circ s=p_+$. Suppose that $y_1y_2\dots y_n$ is a shortest word in $Y^{\pm 1}$ representing $g$.
Let $q$ be the path obtained by  concatenating  the geodesic segments $[(y_1\cdots y_i) \circ s, (y_1 \cdots y_{i+1})\circ s]$  of length at most
$m$, for $i=0,\dots, n-1$. Then $q_-=s=p_-$ and $q_+=g \circ s=p_+$.

We are now going to show that $q$ is a quasi-geodesic.

Consider an arbitrary subpath $r$ of $q$. By the construction of $q$,  there is a subpath $r'$ of $q$ such that
$r'_-=(y_1 \cdots y_i) \circ s$, $r'_+ (y_1 \cdots y_j) \circ s$, for some $0 \le i\le j \le n$,
$\d_\cS(r_-,r'_-) \le m/2$, $\d_\cS(r_+,r'_+) \le m/2$ and $\ell(r) \le \ell(r')+m$.
Then we have
$$\ell(r) \le \ell(r')+m \le m(j-i)+m = m |y_{i+1} \cdots y_j|_Y+m.$$
On the other hand, recalling (2) we get
$$|y_{i+1} \cdots y_j|_Y\le \mu \d_\cS(s, (y_{i+1} \cdots y_j)\circ s)+c=\mu \d_\cS((y_{1} \cdots y_i)\circ s, (y_{1} \cdots y_j)\circ s)+c=\mu\d_\cS(r'_-, r'_+)+c.$$

Combining the two inequalities above with the fact that $\d_\cS(r'_-, r'_+) \le \d_\cS(r_-, r_+)+m$, we obtain
$$\ell(r) \le m \mu \d_\cS(r'_-, r'_+)+ m(c+1)\le  m \mu \d_\cS(r_-, r_+)+m(m \mu+c+1).$$
Thus the path $q$ is $(m\mu,m(m \mu+c+1))$-quasi-geodesic in $\cS$. Let $\kappa=\kappa(\delta,m\mu,m(m \mu+c+1))$ be the constant provided by Lemma~\ref{qg},
so that $p$ is lies in the $\kappa$-neighborhood of $q$. Since  $q \subseteq \nei{Q \circ s}{m/2}$, we see that $p \subseteq \nei{Q \circ s}{\sigma}$, where $\sigma:=\kappa+m/2$.
\end{proof}


\subsection{\Tame elements}
Let $({\cS},{\rm d})$ be a hyperbolic metric space and let $G$ be a group acting on ${\cS}$ by isometries.

\begin{defn}An element $h\in  G$ will be called {\it loxodromic} (with respect to the action on $\cS$), if for
some $s\in \cS,$ the map $\Z \to \cS$, $n \mapsto h^n \circ s$ is a quasi-isometric embedding. By Lemma~\ref{lem:proper}, this is equivalent to the requirements that
the orbit $\gen{h}\circ s$ is quasi-convex and the induced action of $\gen{g}$ on $\cS$ is metrically proper.

An element $h \in G$ enjoys the  \emph{weak proper discontinuity} condition (or $h$ is a \emph{WPD element}) if for every $\varepsilon >0$ and any $x\in {\cS}$, there exists $N=N(\varepsilon,x)$ such that
$$| \{g\in G \mid {\rm d}(x,g\circ x)< \varepsilon, {\rm d}(h^N\circ  x, g h^N\circ x)<\varepsilon \}|<\infty.$$
\end{defn}

 \begin{figure}[ht]
  \begin{center}
 \input{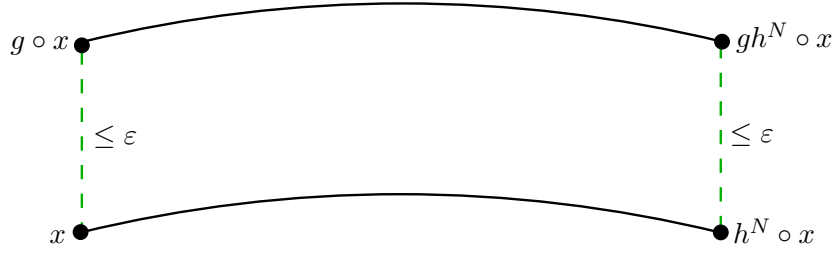}
  \end{center}
  \caption{The WPD property requires the existence of finitely many $g$'s as in the picture.}
  \label{fig:WPD}
\end{figure}

WPD elements originally were introduced by Bestvina and Fujiwara in \cite{BestvinaFujiwara}.
Further in the text we will use $\LWPD$ to denote the set of all elements $g \in G$ that are \tame with respect to the action of $G$  on $\cS$.

\begin{rem}\label{rem:powers}
An element $h\in \LWPD$ if and only if $h^n \in \LWPD$ for any $n\in \mathbb N$.
\end{rem}

To see this, fix some $n\in \mathbb{N}$. Since $\d_{Hau}(\gen{h^n}\circ y,\gen{h}\circ y)$ is finite for all $y\in {\cS}$,  $h$ is loxodromic if and only if $h^n$ is loxodromic.
On the other hand, assuming that the element $h$ is loxodromic,  \cite[Lemma~6.4]{DGO} shows that $h$ is WPD if and only if $h^n$ is WPD.

It is an easy exercise to prove the following
\begin{rem}\label{rem:tame-conj} Suppose that $g,h$ are conjugate elements of $G$. If $g$ is \tame then so is $h$.
\end{rem}

Remark \ref{rem:powers} and \ref{rem:tame-conj} together imply that if $g \in \LWPD$ and $h \apg g$ then $h\in \LWPD$.

Recall that a group is said to be \emph{elementary} if it contains a cyclic subgroup of finite index.

\begin{lem}{\rm \cite[Lemma 6.5, Corollary 6.6]{DGO}}\label{lem:elemrem}
Suppose that $\cS$ is a hyperbolic space, $G$ is a group acting on ${\cS}$ by isometries and $h\in G$ a \tame  element. Then there is a unique maximal elementary subgroup
$E_G(h) \leqslant G$ that contains $h$.
Moreover, for every $x\in G$ the following are equivalent:
\begin{itemize}
\item[{\rm (a)}] $x\in E_G(h)$;
\item[{\rm (b)}] $xh^nx^{-1}=h^{\pm n}$ for some $n\in \N$;
\item[{\rm (c)}] $xh^kx^{-1}=h^l$ for some $k,l\in \Z\setminus \{ 0\}$.
\end{itemize}
Furthermore, set $E_G^+(h) \coloneq \{ x\in G \mid \exists \, k,l\in \mathbb N \mbox{ such that } xh^kx^{-1}=h^l\}$. Then $E_G^+(h)$ is a subgroup of index at most $2$ in $E_G(h)$, and
there exists $n \in \N$ such that $ E_G^+(h)=C_G (h^n)$.
\end{lem}

\begin{rem}\label{rem:E_G(g)} If $G$ is an arbitrary group and $h \in G$ is any element, it is easy to check that the subset $E_G(h) \subseteq G$, defined by
$$E_G(h)\coloneq\left\{x \in G \mid x h^k x^{-1}=h^l \text{ for some }k,l \in \Z\setminus\{0\}\right\},$$
is a subgroup of $G$  containing the centralizer $C_G(h)$. Lemma \ref{lem:elemrem} above describes the structure of this subgroup in the case when
$G$ acts on a hyperbolic space $\cS$ and $h$ is a \tame element.
\end{rem}

\begin{rem}\label{rem:elem_inter} Suppose that $g,h \in G$ are \tame elements for an action of $G$ on some hyperbolic space $\cS$.
\begin{itemize}
  \item If $|E_G(g) \cap E_G(h)|=\infty$ then $g^m=h^n$ for some $m,n \in \Z\setminus\{0\}$.
  \item If $g^m=h^n$ for some $m,n \in \Z\setminus\{0\}$, then $E_G(g)=E_G(h)$.
\end{itemize}
\end{rem}

Indeed, the first claim immediately follows from Lemma \ref{lem:elemrem}, stating that $|E_G(g):\gen{g}|<\infty$ and $|E_G(h):\gen{h}|<\infty$.
The second claim can be quickly derived from part (b) of that lemma.

\subsection{Hyperbolically embedded subgroups}
In this subsection we recall some basic concepts which were originally introduced by Dahmani, Guirardel and Osin in \cite{DGO}.

Let $G$ be a group and let $\Hl $ be a family of subgroups of $G$. Suppose that $X$ is a \emph{relative generating set} of $G$ with respect to $\Hl$
(i.e., $G=\gen{X\cup \bigcup_{\lambda\in \Lambda} H_\lambda}$). Note that $X$ could be infinite; we also assume that it is symmetric, i.e., $X=X^{-1}$ in $G$.
Denote
\begin{equation}
\mathcal{H}= \bigsqcup_{\lambda \in \Lambda} (H_\lambda\setminus \{1\}).
\end{equation}

As discussed in Subsection \ref{subsec:not}, the disjoint union $X\sqcup \mathcal{H}$ can be considered as a `generating alphabet' for $G$, even though
some letters from $X\sqcup\mathcal{H}$ may represent the same element in $G$. Let  $\G$ be the corresponding Cayley graph of $G$.
We also let $\Gamma _\lambda$ denote the Cayley graphs $\Gamma (H_\lambda, H_\lambda\setminus \{1\})$,
which we think of as complete subgraphs of $\G$. By $E\ga_\lambda$ we denote the set of edges of $\ga_\lambda$.

\begin{defn}[{\cite[Def. 4.2]{DGO}}] \label{def:maindef}
For every $\lambda \in \Lambda $, the \textit{relative metric} $\widehat{\rm d}_{\lambda} \colon H_\lambda \times H_\lambda \to [0, \infty)\sqcup\{\infty\}$ is defined as follows.
For any $g,h\in H_\lambda $, $\widehat{\rm d}_{\lambda} (g,h)$ is the length of a shortest path in $\G \setminus E\ga_\lambda$ that joins $g$ to $h$.
If there is no such path, one sets $\widehat{\rm d}_{\lambda} (g,h) \coloneq \infty $.
\end{defn}

It is easy to see that  $\widehat{\rm d}_{\lambda}$ is an extended metric on $H_\lambda$.

\begin{defn}[{\cite[Def. 4.25]{DGO}}]\label{def:he}
The family $\{H_\lambda\}_{\lambda \in \Lambda}$ is \emph{hyperbolically embedded} in $G$ with respect to $X$ (notation: $\{H_\lambda\}_{\lambda \in \Lambda} \h (G,X)$),
if the Cayley graph $\G$ is hyperbolic and the metric space $(H_\lambda, \widehat{\rm d}_{\lambda})$ is locally finite for every $\lambda \in \Lambda$.

We will say that $\{H_\lambda\}_{\lambda \in \Lambda}$ is {\it hyperbolically embedded} in $G$ (notation: $\{H_\lambda\}_{\lambda \in \Lambda} \h G$)
if there exists a (possibly infinite) relative generating set $X$, of $G$ with respect to $\{H_\lambda\}_{\lambda \in \Lambda}$, such that  $\{H_\lambda\}_{\lambda \in \Lambda} \h (G,X)$.
\end{defn}

The concept of hyperbolically embedded subgroups has been introduced by Dahmani, Guirardel and Osin in \cite{DGO}, where they also give an equivalent
definition in terms of relative isoperimetric functions \cite[Theorem 4.24]{DGO} (see also \cite[Thm. 6.4]{Sisto} or Theorem \ref{thm:addQ} below for more equivalent conditions).

Definition \ref{def:he} immediately implies the following observation (cf. \cite[Rem. 4.26]{DGO}):
\begin{rem} \label{rem:subset-hyp_emb} Consider any subset  $\Lambda_1 \subseteq \Lambda$ and set $\Lambda_2:=\Lambda \setminus \Lambda_1$. If  $\Hl \h (G,X)$ then
$\{H_\lambda\}_{\lambda \in \Lambda_1} \h (G,X_1)$, where $X_1:=X \cup \bigcup_{\mu \in \Lambda_2} H_\mu$.
\end{rem}

The following lemma will be useful:

\begin{lem}[{\cite[Cor. 4.27]{DGO}}] \label{lem:changegenset} Suppose that $G$ is a group, $\{H_\lambda\}_{\lambda \in \Lambda}$ is a family of subgroups of $G$ and $X_1,X_2 \subseteq G$ are
relative generating sets of $G$, with respect to $\{H_\lambda\}_{\lambda \in \Lambda}$, such that $|X_1\bigtriangleup X_2|<\infty$. Then
$\{H_\lambda\}_{\lambda \in \Lambda} \h (G,X_1)$ if and only if  $\{H_\lambda\}_{\lambda \in \Lambda} \h (G,X_2).$
\end{lem}

\begin{defn}[{\cite[Def. 4.5]{DGO}}]
Let $q$ be a path in the Cayley graph $\G$ and let $\lambda \in \Lambda$. A non-trivial subpath $p$ of $q$ is called an $H_\lambda$-\emph{component}, if all the edges of $p$ are labelled by letters from
$H_\lambda \setminus\{ 1\}$ and $p$ is a maximal subpath of $Q$ with this property. A \emph{component} of $q$ is an $H_\lambda$-component
for some $\lambda \in \Lambda$.

Two $H_\lambda$-components $p_1$, $p_2$ of  paths $q_1$, $q_2$, respectively, in $\G$ are said to be \emph{connected} if all vertices of $p_1$ and $p_2$
lie in the same left coset of $H_\lambda$ in $G$ (this is equivalent to the existence of an edge $e$ between any two distinct vertices of $p_1$ and $p_2$ with $\Lab(e) \in H_\lambda \setminus\{1\}$).
A component $p$ of a path $q$ is \emph{isolated} if it is not connected to any other component of $q$.

A path $q$ in $\G$ is said to be {\it without backtracking} if all of its components are isolated.
\end{defn}

Below we formulate one of the main technical tools for working with hyperbolically embedded subgroups.
This statement is proved in \cite{DGO}  and is analogous to the relatively hyperbolic case (cf. \cite[Lemma 2.7]{Osin-periph_fil}).

\begin{lem}[{\cite[Lemma 4.11 and Thm. 4.24]{DGO}}]\label{lem:omega}
Suppose that $\Hl$ is hyperbolically embedded in  $(G,X)$. Then for each $\lambda \in \Lambda$ there exists a finite subset
$\Omega_\lambda \subseteq H_\lambda$ and a constant $K \in \mathbb N$ such that the following
holds. Let $q$ be a cycle in $\G$, let $p_i$ be isolated $H_{\lambda_i}$-components of $q$ for $i=1,\dots,k$, and let $h_1, \ldots , h_k$ be the elements of $G$
represented by $\Lab(p_1), \ldots, \Lab(p_k)$
respectively. Then  $h_i$ belongs to the subgroup $\langle \Omega_{\lambda_i}\rangle \le G$ for every $i=1,\dots,k$,  and the word lengths of $h_i$'s
with respect to $\Omega_{\lambda_i}$ satisfy
$$ \sum\limits_{i=1}^k |h_i|_{\Omega_{\lambda_i}} \le K\ell(q).$$
\end{lem}

\subsection{Acylindrically hyperbolic groups}\label{subsec:acyl_hyp}
Suppose that a group $G$ acts by isometries on a metric space $(\cS,\d)$. Following Bowditch \cite{Bow} we will say that this action is \emph{acylindrical} if for every $\varepsilon > 0$ there exist $R,N > 0$
such that for any pair of points $x, y \in \cS$ with $\d(x, y) \ge R$ one has
$$|\{g \in G \mid d(x, g\circ x) \le \varepsilon \mbox{ and } d(y, g\circ y)\le \varepsilon \}|\le N. $$

Comparing this with the definition of a \tame element above, we immediately obtain
\begin{rem}\label{rem:acyl->tame}  If a group $G$ acts acylindrically on a hyperbolic space $\cS$ then every loxodromic element of $G$ satisfies the WPD condition.
\end{rem}

The action of $G$ on $\cS$ is \emph{non-elementary} if for some (equivalently, for any) $s \in \cS$,  the set of limit points $\Lambda(G \circ s)$ of the orbit $G \circ s$ in the Gromov boundary $\partial S$
has at least $3$ points.

In \cite{Osin-acyl} Osin proved the following theorem:

\begin{thm}[{\cite[Thm. 1.2]{Osin-acyl}}]\label{thm:acyl-equiv_def} For any group $G$ the following are equivalent:
\begin{itemize}
  \item[(i)] $G$ admits a non-elementary acylindrical action on some hyperbolic space;
  \item[(ii)] there is a symmetric generating subset $X$ of $G$ such that the Cayley graph $\ga(G,X)$ is hyperbolic, the natural action of $G$ on $\ga(G,X)$ is acylindrical and
                  the Gromov boundary $\partial \ga(G,X)$ has more than two points;
  \item[(iii)] $G$ is non-elementary and there exists a hyperbolic space $\cS$ such that $G$ acts on $\cS$ coboundedly and by isometries and $\LWPD \neq \emptyset$;
  \item[(iv)] $G$ contains a proper infinite hyperbolically embedded subgroup.
\end{itemize}
\end{thm}

Remark that in \cite[Thm. 1.2]{Osin-acyl} the statement (iii) of Theorem \ref{thm:acyl-equiv_def} is formulated in a weaker form, without the requirement for the action to be cobounded.
However, (ii) clearly implies (iii) with this additional condition: assuming (ii), one can simply take $\cS$ to be the Cayley graph $\ga(G,X)$  on which $G$ acts acylindrically (the hypothesis that $\partial \ga(G,X)\neq \emptyset$
implies that the unique $G$-orbit of vertices in $\cS=\ga(G,X)$ is unbounded, hence $\LWPD\neq\emptyset$ by Remark~\ref{rem:acyl->tame} and
the classification of acylindrical actions of groups on hyperbolic spaces obtained by Osin in \cite[Thm. 1.1]{Osin-acyl}).

Theorem \ref{thm:acyl-equiv_def} allows one to say that a group $G$ is \emph{acylindrically hyperbolic} if it satisfies one of the equivalent conditions (i)--(iv) from its claim.


\section{Adding  subgroups to a family of hyperbolically embedded subgroups}\label{sec:adding}
In this section we give necessary and sufficient conditions that allow to add a finite family of subgroups to the existing family of hyperbolically embedded subgroups. This is analogous to
Osin's theorem \cite{O06}, where a similar criterion was developed for relatively hyperbolic groups.

\subsection{Necessary conditions}
In this subsection we suppose that $G$ is a group, $X_1$ is a generating set of $G$ and $Q_1,\dots,Q_n$ is a collection of subgroups of $G$
such that $\{Q_i\}_{i=1}^n \h (G,X_1)$.

\begin{lem}\label{lem:he->fgqi}
For each $i \in \{1,\dots,n\}$ there exists a finite generating set $Y_i$ of $Q_i$ and constants
$\mu_i \ge 1$, $c_i \ge 0$ such that 
$$|h|_{Y_i} \leq  \mu_i |h|_{X_1}+c_i$$ for all $h\in Q_i$.
\end{lem}

\begin{proof} Obviously, it is enough to prove the statement for $i=1$.
Let $X_2:=X_1 \cup \bigcup_{j=2}^n Q_j$. Then $X_2$ generates $G$ and $Q_1 \h (G,X_2)$ by Remark \ref{rem:subset-hyp_emb}.
Let $\Omega_1 \subseteq Q_1$ and $K>0$ be the finite subset and the constant provided by Lemma \ref{lem:omega}.

Consider any element $h \in Q_1\setminus \{1\}$. Since $X_2$ generates $G$, we can let $W$ to be a shortest word over $X_2$ such that $h=W$ in $G$.
Therefore, in the Cayley graph $\ga(G,X_2\sqcup {Q_1\setminus\{1\}})$, there is a cycle $q$ with $\Lab(q) \equiv Wh^{-1}$. Evidently, $q$ has exactly one $Q_1$-component
labelled by $h^{-1}$, hence it must be isolated in it.  Consequently, by Lemma  \ref{lem:omega}, $h \in \gen{\Omega_1}$ and
$$|h|_{\Omega_1} \le K \ell(q)=K\|W\|+K=K|h|_{X_2}+K\le K|h|_{X_1}+K.$$
Thus $Q_1$ is generated by the finite set $\Omega_1$ and the required inequality for the word lengths is satisfied.
\end{proof}

\begin{lem}\label{lem:he->geomsep}
Let $i,j \in\{1,\dots,n\}$ and $g\in G$.
For  every $\e>0$, there exists $R=R(\e)>0$ such that diameter
 $ \diam_{X_1} (Q_i \cap \nei{gQ_j}{\e})<R$ 
whenever $i\neq j$ or $i=j$ and $g\notin Q_i$ (here $\nei{gQ_j}{\e}:= \{z\in G \mid  \d_{X_1}(z,gQ_j) \le \e\}$).

\end{lem}

\begin{proof} This is a straightforward consequence of Definition \ref{def:he}. Indeed, this definition implies that for any $\e>0$  the set $I_l:=\{h \in Q_i \mid \widehat {\rm d}_l (1,h) \leq 1+2\e\}$ is finite,
where $\widehat {\rm d}_l$ is the relative metric on $Q_l$ induced from the Cayley graph $\ga(G,X_1 \sqcup \mathcal{Q})$ with $\mathcal{Q}:=\bigsqcup_{k=1}^n Q_k$.
Hence we can let $R:=\max \{|h|_{X_1} \mid h \in I_l, l\in \{1,\dots,n\}\}+1$.

Now, for any distinct $h_1,h_2\in Q_i \cap \nei{g Q_j}{\e}$ there are $f_1,f_2 \in G$ such that $|f_k|_{{X_1}} \leq \e$ and $h_k \in gQ_j f_k$ for $k=1,2$.
Therefore $h_1^{-1}h_2 = Q_i \cap f_1^{-1} t f_2$ for some $t \in Q_j$.
For $k=1,2$, let $U_k$ be a word over ${X_1}$ of length at most $\e$ representing $f_k$ in $G$, and let $T \in \cH$ be the letter representing $t$.
Consider the path $p$ in $\ga(G,X_1 \sqcup \mathcal{Q})$ starting at $1$ and labelled by the word $U_1^{-1}TU_2$. If $i \neq j$ or $g \notin Q_i$ then the path $p$
has no edges from $ E\ga_i$ (indeed, if $i=j$ but $g \notin Q_i$ then $f_1^{-1} \notin Q_i$) and $p_+=h_1^{-1}h_2$. Since
$\ell(p)\leq 2\e+1$ we see that $h_1^{-1}h_2 \in I_i$, which implies that $\d_{X_1}(h_1,h_2)=|h_1^{-1}h_2|_{X_1}<R$ as required.
\end{proof}


\subsection{Sufficiency}

\begin{nt} \label{hyp:add} Throughout this section  $G$ is a group,  $X_1$ is a generating set of $G$  such that  $\ga=\ga(G,X_1)$ is $\delta$-hyperbolic, for some $\delta \ge 0$,
and $Q_1,\dots, Q_n$ is a finite collection  of subgroups of $G$. We use $\d$ to denote the graph metric on $\ga$.

We will consider the following properties for the family $Q_1,\dots, Q_n$.

\begin{enumerate}

\item[(Q1)] ({\it $\{Q_i,\}_{i=1}^n$ is geometrically separated}) For every $\e>0$ there exists $R=R(\e)$ such that for $g\in G$ if $$\diam (Q_i \cap \nei{gQ_j}{\e}) \ge R$$
then $i=j$ and $g\in  Q_i$ (here the distances are measured with respect to the graph metric $\d$ on $\ga$).

\item[{\rm (Q2)}] ({\it Finite generation}) For each $i$,  there exists a finite subset $Y_i\subset G$   generating  $Q_i$.
\item[{\rm (Q3)}] ({\it Quasi-isometrically embedded}) There exist $\mu\geq 1$ and $c\geq 0$ such that for any $i \in \{1,\dots,n\}$ and all $h\in Q_i$ one has $|h|_{Y_i}\leq \mu |h|_{X_1}+c$.

\end{enumerate}
\end{nt}

\begin{rem}\label{rem:he->Q1-3}
Under the previous notation, suppose that  $\{Q_1,\dots, Q_n\}$ is hyperbolically embedded in $(G, X_1)$, then
by Lemma \ref{lem:he->geomsep} the family $\{Q_1,\dots, Q_n\}$ satisfies (Q1), and by Lemma~\ref{lem:he->fgqi} $\{Q_1,\dots, Q_n\}$ satisfies (Q2) and (Q3) with $\mu:=\max\{\mu_i \mid i=1,\dots,n\}$
and $c:=\max\{c_i \mid i=1,\dots,n\}$.
\end{rem}
The goal of this section is prove the converse result. Namely, if $Q_1,\dots, Q_n$ satisfy (Q1)--(Q3) then $\{Q_1,\dots, Q_n\}\h (G,X_1)$.

The next lemma says that if a pair of geodesics, labelled by elements of some $Q_i$'s, have sufficiently long $k$-connected subpaths, then the endpoints of these geodesics belong to the same coset of
$Q_i$.

\begin{lem}\label{lem:Q2}
In the Notation \ref{hyp:add}, suppose that $Q_1,\dots, Q_n$ satisfy {\rm (Q1)--(Q3).}

For every $k > 0$ there exists $A=A(k )>0$ such that the
following holds. Suppose that $p$, $q$ are geodesic paths in $\ga$ such that $\Lab( p )$ (resp. $\Lab (q )$)
represents an element of $Q_i$ (resp. $Q_j$), and that there exist two $k$-close subpaths  $u$ and $v$ of $p$ and $q$.
If $\max\{ \ell( u ), \ell( v )\}\geq A$, then $i=j$ and  
the label of an arbitrary path connecting   any endpoint of $p$ with any endpoint of $q$ represents an element of $Q_i$.
\end{lem}

\begin{proof}
By Lemma \ref{lem:proper} and Remark \ref{rem:qc_orbits}, there exists $\sigma \ge 0$ such that $Q_i$ and $Q_j$ (considered as subsets of $\ga$) are $\sigma$-quasi-convex.
Let $\e \coloneq k+2\sigma$ and $R=R(\e)$ be given by (Q1); set $A:=R+2\sigma$.

Without loss of generality we can assume that $u$ and $v$ are $k$-connected, $p_-=1$ and $\ell(u)\ge A$. Denote $g \coloneq q_-$.
Then there are vertices $a_-,a_+ \in Q_i$ and $b_-,b_+ \in g Q_j$ such that $\d(a_-,u_-)\le \sigma$, $\d(a_+,u_+)\le \sigma$,
$\d(b_-,v_-)\le \sigma$ and $\d(b_+,v_+)\le \sigma$. Consequently, $\d(a_-,b_-),\d(a_+,b_+) \le k+2\sigma$, thus
$a_-,a_+ \in Q_i \cap \nei{gQ_j}{\e}$. One can also note that $$\d(a_-,a_+) \ge \ell(u)-\d(u_-,a_-)-\d(u_+,a_+)\ge A-2\sigma \ge R,$$
hence $i=j$ and $g \in Q_i$ by the assumption (Q1), finishing the proof of the lemma.
\end{proof}

\begin{nt}\label{not:phi}
In the Notation \ref{hyp:add}, suppose that $Q_1,\dots, Q_n$ satisfy (Q1)--(Q3) and $\bigcup_{i=1}^n Y_i \subseteq X_1$.
Let $\mathcal{Q}= \bigcup_{i=1}^n (Q_i\setminus \{1\})$ and $\ga' = \Gamma(G,X_1\sqcup \mathcal{Q})$.
We will denote by  $\d'$ the graph metric on $\ga'$.

For every $i=1,\dots, n$ and every $h\in Q_i$,  fix a shortest word $V(h)$ over $Y_i^{\pm 1}$
representing $h$. Since $\ga$ and $\ga'$ have the same vertex set $G$, we can define a map $$\phi\colon \{\text{paths in }\Gamma' \}\to \{\text{paths in }\Gamma\} $$ just by replacing each edge $e$,
labelled by some $h\in Q_i$  in $\ga'$, with the (unique) path $\phi(e)$, labelled by $V(h)$ and having the same initial and terminal vertices as $e$ in $\ga$. In particular, $\phi(p)_-=p_-$ and $\phi(p)_+=p_+$
for any path $p$ in $\ga'$.
\end{nt}

Our goal now is to show that if $p$ is a geodesic in $\Gamma'$ then $\phi(p)$ is a quasi-geodesic in $\Gamma$. In order to do so we will use the following lemma that deals with the situation when
for some path $p$ the path $\phi(p)$ is ``far'' from being a geodesic. The conclusion is that in this case $p$ backtracks, i.e., it goes through a coset of some $Q_i$ twice.

\begin{lem}\label{lem:paths}
In the Notation \ref{not:phi}, there exists $D\ge 1$ such that for all $r \ge 1$, $k\geq 0$ and  every path $p$ in $\ga'$ satisfying $\ell( p )\leq r$ and $\d(p_-,p_+)\leq k$, if
$\ell(\phi( p )) \geq D(r+k)$  then there exist $l\in \{1,\dots, n\}$ and two distinct
edges $e_1$ and $e_2$ of $p$ that are labelled by letters from $Q_l\setminus\{1\}$, so that
all the endpoints of $e_1$ and $e_2$ belong to the same left coset of $Q_l$ in $G$.
\end{lem}

\begin{proof}
Let $A=A(13\delta)$ be the constant provided by Lemma \ref{lem:Q2}, where $\delta$ is the hyperbolicity constant of $\ga$, and set $a:=A+30\delta$.

We now fix $r\geq 0$, $k\geq 0$ and a path  $p$ in $\ga'$ such that  $\ell( p )\leq r$ and $\d(p_-,p_+)\leq k$.

Suppose that $\Lab (p )\equiv W_0 h_1 W_1 h_2\dots W_{m-1} h_mW_m$ where each $h_i\in \cQ$ and each $W_i$ is a (possibly empty) word in $X_1$, in particular $m \le r$.
We have that $$\Lab (\phi ( p )) \equiv W_0 V(h_1) W_1 V(h_2)\dots W_{m-1} V(h_m)W_m.$$ Let $U_i$ be a shortest word over $X_1$ representing the same
element of $G$ as $W_i$, $i=0,\dots,m$, and let $V_j$ be the shortest word over $X_1$ representing the element $h_j$, $j=1,\dots,m$.
Consider the path $q$ in $\ga$ with the same endpoints as $\phi(p)$ and with
$\Lab(q)\equiv U_0 V_1 U_1 V_2\dots U_{m-1} V_mU_m$.
Recall that $\sum_{i=0}^m \|W_i\| \le \ell(p)\le r$, hence,  in view of (Q3), we have
\begin{equation}\label{eq:sum-V_j}
\sum_{j=1}^m \|V_j\| \ge \frac{1}{\mu} \sum_{j=1}^m\|V(h_j)\|-mc \ge   \frac{1}{\mu} \ell(\phi(p))-(c+1)r .
\end{equation}

Observe that $ q $ can be written as the concatenation of geodesic paths  $t_0, s_1,\dots, t_{m-1}, s_m,t_{m} $ in $\ga$, where $\Lab (t_i)\equiv  U_i$ and $\Lab (s_j) \equiv V_j$.
Let  $t_{m+1}$ be a geodesic path in $\Gamma$ from $q_+$ to $q_{-}$; then $\ell(t_{m+1})=\d(q_-,q_+)=\d(p_-,p_+)\le k$.
The polygon $\mathcal{P}:=t_0s_1 \dots s_m t_m t_{m+1}$ is a geodesic $(2m+2)$-gon in $\ga$
and we partition its sides into two subsets $S:=\{s_1,\dots,s_m\}$ and $T:=\{t_0,\dots,t_{m+1}\}$.

By the assumptions we have that $\rho:=\sum_{i=0}^{m+1} \ell (t_i) \leq \ell(p)+d(p_-,p_+) \leq r+k$, and
$$\sigma:=\sum_{j=1}^m \ell(s_j)=\sum_{j=1}^m \|V_j\| \ge \frac{1}{\mu} \ell(\phi(p))-(c+1)r$$
by \eqref{eq:sum-V_j}.
Choose a constant $D\ge 1$ (independent of $r$ and $k$) so that $$\frac{D}{\mu}(r+k)-(c+1)r \ge \max\{ 10^3 a (2r+2),10^3 (r+k)\},$$ and suppose that $\ell(\phi( p )) \geq D(r+k)$.
Since $2m+2 \le 2r+2$, all the conditions of Lemma \ref{Ols} will then be satisfied, hence there will be $i,j \in \{1,\dots,m\}$, $i \neq j$, and two
$13\delta$-{close} subsegments $u$ of $s_i$ and $v$ of $s_j$ such that $\min \{ \ell (u ) ,\ell (v )\}>a\geq A$.
It remains to apply Lemma \ref{lem:Q2}, claiming that there is $l\in \{1,\dots, n\}$ such that $h_i,h_j \in Q_l$ and all the endpoints of the corresponding edges
of $\Gamma'$ belong to the same left coset of $Q_l$.
\end{proof}

We are now ready to show that $\phi(p)$ is a geodesic when $p$ is a geodesic. The key observation, which allows us to use the previous lemma, is that a geodesic does not backtrack. (We also apply this to subpaths of $p$.)

\begin{lem}\label{lem:geod-qgeod}
In the Notation \ref{not:phi}, let  $D\ge 1$ be the constant provided by Lemma \ref{lem:paths}.
Then for any geodesic path $p$ in $\ga'$, the path $\phi(p)$ is $(2D,5D)$-quasi-geodesic in $\ga$.
\end{lem}

\begin{proof} As before, suppose that  $\Lab (p )\equiv W_0 h_1 W_1 h_2\dots W_{m-1} h_mW_m$ where each $h_i\in \cQ$ and each $W_i$ is a (possibly empty) word in $X_1$.
Consider any (combinatorial) subpath $p'$ of $\phi(p)$ in $\ga$. Let us assume that $\Lab(p')$ starts with a suffix $V'(h_\alpha)$ of $V(h_\alpha)$ and ends with a prefix $W'_\beta$  of $W_\beta$
for some $\alpha, \beta \in \{1,\dots,m\}$, $\alpha \le \beta$, as the other cases can be treated similarly. Thus
$\Lab (p') \equiv V'(h_\alpha) W_\alpha V(h_{\alpha+1})\dots V(h_\beta)W'_\beta$. Since $V'(h_\alpha)$ is a geodesic word over $Y_i^{\pm 1}$ for some $i \in \{1,\dots,n\}$,
it represents an element $h'\in Q_i$ and  $\|V'(h_\alpha)\|=\|V(h')\|$. Let $q$ be the path in $\ga'$ with $q_-=p'_-$ and
$\Lab(q) \equiv h' W_\alpha h_{\alpha+1}\dots h_\beta W'_\beta$. Then $\Lab(\phi(q))\equiv V(h') W_\alpha V(h_{\alpha+1})\dots V(h_\beta)W'_\beta$, which implies that
$q_+=\phi(q)_+=p'_+$ and $\ell(\phi(q))=\ell(p')$.

Let $s$ be the subpath of $q$ with $\Lab(s) \equiv W_\alpha h_{\alpha+1}\dots h_\beta W'_\beta$. Then $s$ is geodesic in $\ga'$,
as it is also a subpath of $p$, $\ell(s) \ge \ell(q)-2$, and the endpoints of $s$ lie at distance at most $1$ from the corresponding endpoints of $q$ in $\ga'$.

Set $r:=\ell(q)+1$ and $k:=\d(q_-,q_+)$. Then $r \le k+5 $ because
$$k=\d(q_-,q_+)\ge \d'(q_-,q_+) \ge \d'(s_-,s_+)-2=\ell(s)-2 \ge \ell(q)-4= r-5.$$

Since $p$ is geodesic in $\ga'$, all $\cQ$-components of $p$ consist of single edges and no two components of $p$ are connected. The latter
also holds for $q$ since any component of $q$ is connected to a component of $p$. Therefore Lemma \ref{lem:paths} implies that $\ell(\phi(q))<D(r+k)$.
Consequently,
$$\ell(p')=\ell(\phi(q)) <  D(2k+5) =2D \d(q_-,q_+)+5D=2D\d(p'_-,p'_+)+5D,$$
which shows that $\phi(p)$ is $(2D,5D)$-quasi-geodesic in $\ga$.
\end{proof}

The following is the main result of this section. It generalizes  \cite[Theorem 1.5]{O06}.
\begin{thm}\label{thm:addQ}
Suppose that $G$ is a group, $\Hl$  is a collection of subgroups of $G$ and $X$ is a relative generating set of $G$ with respect to $\Hl$, such that  $\Hl \h (G,X)$.
Set $X_1:=X \sqcup \cH$.

A family $\{Q_i\}_{i=1}^n$  of subgroups of $G$ satisfies {\rm (Q1)--(Q3)} if and only if the family $\Hl \sqcup \{ Q_i\}_{i=1}^n$ is hyperbolically embedded in $(G,X)$.
  \end{thm}

\begin{proof}
The necessity is given by Remark \ref{rem:he->Q1-3}, so we only have to show that if  $\{Q_i\}_{i=1}^n$   satisfies (Q1)--(Q3)  then $\Hl \cup \{ Q_i\}_{i=1}^n \h(G,X)$.

Since the set $\bigcup_{i=1}^n Y_i$ is finite, without loss of generality we can suppose that $\bigcup_{i=1}^n Y_i \subset X\subset X_1$ (see Lemma \ref{lem:changegenset}).
Using Notation \ref{not:phi}, let $D\ge 1$ be the constant provided by Lemma~\ref{lem:paths}.

Take any $i\in \{1,\dots, n\}$ and $\lambda \in  \Lambda$. We will denote by $\ga_i$ the Cayley graph $\ga(Q_i, Q_i\setminus\{1\})$ and by $\ga_\lambda$ the Cayley graph $\ga(H_\lambda, H_\lambda\setminus\{1\})$.
The set of edges of $\ga_i$ and $\ga_\lambda$ will be denoted $E\ga_i$ and $E\ga_\lambda$ respectively.
By $\widehat{\d}_\lambda$ and $\widehat{\d'}_\lambda$ we denote the metrics on $H_\lambda$ induced by graph metric on $\ga \setminus E\ga_\lambda$ and $\ga' \setminus E\ga_\lambda$, respectively.
The metric $\widehat{\d'}_i$ on $Q_i$ is defined similarly.

We now break the proof in three claims.

{\bf Claim 1:} For every $i=1,\dots, n$  the metric space $(Q_i,\widehat{\d'}_i)$  is locally finite.

Let $a\in Q_i\setminus\{1\}$ and let $p_1$ be a shortest path from $1$ to $a$ in $\Gamma' \setminus E\Gamma_i$.
Let $e$ be the edge of $\ga_i$ from $(p_1)_-=1$ to $(p_1)_+=a$. Define $p$ to be the cycle in $\ga'$ obtained by concatenating $p_1$ with $e$.
Suppose that $\ell( \phi ( p ) ) \geq  D \ell( p  )=D(\ell(p_1)+1)$.
Then, by Lemma \ref{lem:paths}, there are $l \in \{1,\dots,n\}$ and two distinct edges $e_1$ and $e_2$ of $p$, labelled
by some letters from $Q_l\setminus\{1\}$, such that all endpoints of these edges belong to the same left coset $gQ_l$.

Note that if $l=i$ then $g \notin Q_i$, as otherwise both $e_1$ and $e_2$ would have belonged to  $E\ga_i$, but the only edge of $p$ from $E\ga_i$ is $e$.
In particular,  $e_1 \neq e$ and $e_2 \neq e$.
It follows that the subsegment of $p_1$ starting with $e_1$ and ending with $e_2$ can be substituted by a single edge $e'$, labelled by a letter from $Q_l$, so that the resulting path $p_1'$
still lies in $\ga' \setminus E\ga_i$, connects $1$ with $a$ and $\ell(p_1')<\ell(p_1)$, which contradicts the choice of $p_1$. Therefore
$$\d(1,a) \le\ell( \phi ( p_1 ) ) \le\ell( \phi ( p ) ) <  D\ell( p  )=D(\ell(p_1)+1)=D \widehat{\d'}_i(1,a)+D.$$

By (Q2) and (Q3), for each $R$ there are only finitely many elements in $Q_i$ of $X_1$-length at most $DR+D$. This completes the proof of Claim 1.

{\bf  Claim 2:} For each $\lambda \in \Lambda$ the metric space $(H_\lambda, \widehat{\d'}_\lambda)$ is locally finite.

Recall that by hypothesis $(H_\lambda, \widehat{\d}_\lambda)$ is locally finite.
Arguing by contradiction, suppose that for some $r \ge 1$, there exist infinitely many $h\in H_\lambda$ such that $\widehat{\d'}_\lambda (1,h)\le r$.

Since $(H_\lambda, \widehat{\d}_\lambda)$ is locally finite, there exists $h_0\in H_\lambda$ such that $\widehat{\d'}_\lambda (1,h_0) \le r$ and
$\widehat{\d}_\lambda (1,h_0)> D(r +1)$. Let $p$ be a shortest path  in $\ga' \setminus E\ga_{\lambda}$ from $1$ to $h_0$, with $\ell( p ) \le r$.
Notice that, by construction, $\phi( p )$ is a path in $\ga \setminus E\ga_\lambda$.
Since ${\d}(1,h_0)={\d}(p_-,p_+) =1$, the inequality $\widehat{\d}_\lambda (1,h_0)> D(r +1)$ implies that $\ell (\phi ( p ) )> D(r+\d(p_-,p_+))$.
Hence, we can use Lemma \ref{lem:paths} to argue as above that the path $p$ can be shortened, yielding the required contradiction.

{\bf Claim 3:} The graph $\Gamma'$ is $\delta'$-hyperbolic.

Consider any geodesic triangle $\Delta=p_1 p_2 p_3$ in $\Gamma'$ and vertex $v \in p_1$.
By Lemma \ref{lem:geod-qgeod}, the triangle $\phi(\Delta):=\phi(p_1) \phi(p_2) \phi(p_3)$ is $(2D,5D)$-quasi-geodesic in $\ga$.
Let $\varkappa=\varkappa(\delta,2D,5D)$ be the constant from Lemma \ref{qg}.

Note that $v$ is also a vertex of $\phi(p_1)$, and any vertex $u \in \phi(p_i)$, regarded as an element of $G$ (and thus as a vertex of $\ga'$), lies within $\d'$-distance $1$ of a vertex of $p_i$ in $\ga'$,
$i=1,2,3$. Now, since the graph $\ga$ is $\delta$-hyperbolic, there is a vertex $u \in \phi(p_2) \cup \phi(p_3)$ such that $\d(v,u) \le \delta+2\varkappa$.
Observe that $\d'(v,u) \le \d(v,u)$ by definition, hence there is a vertex $w \in p_2 \cup p_3$ such that
$$\d'(v,w) \le \d'(v,u)+1 \le \delta+2\varkappa+1.$$ Thus the graph $\ga'$ is $\delta'$-hyperbolic, for $\delta':=\delta+2\varkappa+1$.

Claims 1--3 imply that the family of subgroups $\Hl \sqcup \{ Q_i\}_{i=1}^n$ is hyperbolically embedded in $G$, and so the theorem is proved.
\end{proof}

The following corollary gives and alternative proof of \cite[Theorem 4.42]{DGO} when the action of $G$ on $\cS$ is cobounded. During the work on this paper the authors learned that this corollary was
independently proved by Hull in \cite[Thm. 3.16]{Hull}.
See also \cite[Thm. 6.4]{Sisto} for other equivalent conditions.

\begin{cor}\label{cor:add_Q} Let $G$ be a group acting by isometries  on a hyperbolic space $({\cS},\d)$. Suppose that this action is cobounded and $\{Q_i\}_{i=1}^n$ is a finite family of subgroups of $G$.
Fix any $s \in \cS$. Then the following are equivalent.
\begin{enumerate}
\item[\rm{(a)}] The family $\{Q_i\}_{i=1}^n$ satisfies the conditions:
\begin{enumerate}
\item[\rm {(i)}] $Q_i\circ s$ is quasi-convex and the induced action of $Q_i$ on $\cS$ is metrically proper, $i=1,\dots,n$;
\item[\rm{(ii)}]  for every $\e > 0$ there exists $R$ such that for $g \in G$ if $\diam (Q_i\circ s\cap \nei{gQ_j\circ s}{\e}) >R$ then $i=j$ and $g\in Q_i$.
\end{enumerate}
\item[\rm{(b)}]  The family $\{Q_i\}_{i=1}^n$ is hyperbolically embedded in $(G,X_1)$, where $X_1$ is a generating set of $G$ provided by Lemma \ref{lem:svarc-milnor}.
\end{enumerate}
\end{cor}
\begin{proof}
By the {\v S}varc-Milnor lemma (Lemma \ref{lem:svarc-milnor}),  the map $g\mapsto g\circ s$ is a $G$-equivariant quasi-isometry between
$G$, endowed with the metric from $\Gamma(G,X_1)$, and $({\cS},\d)$. In particular, $\Gamma(G,X_1)$ is hyperbolic and $\emptyset \h (G,X_1)$.

If we show that (i)-(ii) are equivalent to (Q1)--(Q3), the result will follow from Theorem~\ref{thm:addQ}.
Indeed, by Lemma \ref{lem:proper} and as $\Gamma(G,X_1)$ is quasi-isometric to $({\cS},\d)$, the family $\{Q_i\}_{i=1}^n$ satisfies (i)
if and only if it satisfies (Q2) and (Q3). On the other hand, (ii) is a restatement of (Q1).
\end{proof}

As a corollary we obtain the following statement (cf. \cite[Cor. 4.14]{Hull}):

\begin{cor}\label{cor:wpd-he}
Let $G$ be a group acting coboundedly on a hyperbolic space $({\cS},d)$ and let $X_1$ be a generating set of $G$ given by Lemma \ref{lem:svarc-milnor}.
If $h_1, \dots , h_k $ is a collection of pairwise non-commensurable \tame elements with respect to the action of $G$ on $\cS$ then $\{ E_G(h_1), \ldots , E_G(h_k)\} \hookrightarrow_{h} (G,X_1)$.
\end{cor}

\begin{proof}
Fix $i\in\{1,\dots n\}$. Since $h_i$ is loxodromic, there is $s\in {\cS}$ such that the orbit $\gen{h_i}\circ s$ is  quasi-convex and the action of $\gen{h_i}$ on $\cS$ is metrically proper.
Thus the condition (a).(i)  from Corollary \ref{cor:add_Q} is satisfied.

The geometric separability condition (a).(ii) from Corollary \ref{cor:add_Q}  for the family $\{ E(h_i)\}_{i=1}^n$ is proved in \cite[Thm.~6.8]{DGO}.
Hence,  $\{E(h_i)\}_{i=1}^n \h (G,X_1)$ by Corollary  \ref{cor:add_Q}.
\end{proof}

One can note that Corollary \ref{cor:wpd-he} resembles  \cite[Theorem 6.8]{DGO}. The main difference is that we require the action to be cobounded, but because of this we are able to
specify that the relative generating set $X_1$ comes naturally from the action of $G$ on $\cS$ (this will be important for the rest of the paper).

Similarly, Theorem \ref{thm:addQ} can also be used to obtain the following strengthening of Corollary~\ref{cor:wpd-he}:

\begin{cor} Let $G$ be a group with a family of subgroups $\Hl$ and  a relative generating set $X$ (with respect to $\Hl$), such that  $\Hl \h (G,X)$.
Set $\mathcal{H}:=\bigsqcup_{\lambda \in \Lambda} (H_\lambda \setminus \{1\})$.
If $h_1, \dots , h_k $ is a collection of pairwise non-commensurable \tame elements with respect to the action of $G$ on $\ga(G,X \sqcup \mathcal{H})$ then
 the family $\{H_\lambda\}_{\lambda \in \Lambda} \sqcup \{ E_G(h_i)\}_{i=1}^k$ is hyperbolically embedded in $(G,X)$.
\end{cor}


\section{Combinatorics of paths}
This section provides some technical geometric tools which will later be used to develop the theory of acylindrically hyperbolic groups similarly to the theory of relatively hyperbolic groups.
Let $G$ be a group, let $\Hl$ be a family of subgroups of $G$ and let $X$ be a symmetric relative generating set of $G$ with respect to $\Hl$. As usual, we set
$\mathcal{H}:=\sqcup_{\lambda \in \Lambda} (H_\lambda \setminus \{1\} )$.

\begin{defn} \label{defn:w} 
Suppose that  $m \in \mathbb N$  and
$\mathcal{O}=\{\Omega_\lambda\}_{\lambda \in \Lambda}$ is a collection of finite subsets of $G$. Define $\cW (\mathcal{O},m,X,\cH)$ to be the set of all
words $W$ over the alphabet $X \cup \mathcal H$ that have the following form:
$$W \equiv x_0h_1x_1h_2 \dots x_{l-1} h_l x_{l},$$ where $ l \in \mathbb Z$, $l \ge -1$ (if $l=-1$ then $W$ is the empty word;
if $l=0$ then $W \equiv x_0$),
$h_i$ and $x_i$ are considered as single letters and
\begin{itemize}
\item[(1)] for every $i=0,1,\dots,l$ either $x_i \in X$ or $x_i$ is the empty word,
and for each $i=1,2,\dots,l$, there exists $\lambda(i) \in \Lambda$ such that $h_i \in H_{\lambda(i)}$;
\item[(2)] if $\lambda(i)=\lambda(i+1)$ then $x_{i+1} \notin H_{\lambda(i)}$ for each $i=1,\dots,l-1$;
\item[(3)] $h_i \notin \{h \in \langle \Omega_{\lambda(i)} \rangle \mid |h|_{\Omega_{\lambda(i)}} \le m \}$, $i=1,\dots,l$.
\end{itemize}
Finally, let $\cW_0(\mathcal{O},m,X,\cH)$ be defined as the subset of all words $W \equiv h_1x_1h_2 \dots x_{l-1} h_l x_{l} \in \cW (\mathcal{O},m,X,\cH)$ such that
$l \ge 1$ and if $\lambda(l)=\lambda(1)$ then $x_l \notin H_{\lambda(1)}$.
Thus $\cW_0(\mathcal{O},m,X,\cH)$ can be thought of as the set of cyclically reduced words from $\cW(\mathcal{O},m,X,\cH)$.
\end{defn}

For the remainder of this section assume that $\Hl$ is hyperbolically embedded in  $(G,X)$.
Choose the collection of finite subsets $\mathcal{O}=\{\Omega_\lambda\}_{\lambda \in \Lambda}$ of $G$ (so that $\Omega_\lambda \subseteq H_\lambda$ for all $\lambda \in \Lambda$)
and the constant $K>0$ according to the claim of Lemma \ref{lem:omega}.

The following lemmas are taken from \cite[Section 6]{Minasyan_classes}, where they were established in the case when $G$ is hyperbolic relative to the family $\Hl$.
Their proofs  only use the combinatorial properties of the paths with labels from $\cW(\mathcal{O}, m, X, \cH)$, together
with the claim of \cite[Lemma~6.1]{Minasyan_classes}.  Using
Lemma \ref{lem:omega} instead of the latter, the proofs transfer verbatim to the more general settings of the present paper.

\begin{lem} \label{lem:no_back} Let $q$ be a path in the Cayley graph $\G$ with $\Lab(q) \in \cW (\mathcal{O},m,X,\cH)$ and $m \ge 5K$.
Then $q$ is without backtracking.
\end{lem}
\begin{proof}
See the proof of \cite[Lemma 6.2]{Minasyan_classes}.
\end{proof}

\begin{lem}\label{lem:regul}
Let $o=rqr'q'$ be a cycle in the Cayley graph $\G$, such that $\Lab(q),\Lab(q')\in \cW(\mathcal{O}, m, X,\cH)$.
Suppose that $m\ge 7K$ and denote $C=\max\{ \ell ( r ),\ell( r' )\}$. Then
\begin{itemize}
\item[\rm (a)] if $C\le 1$ then no component of $q$ or $q'$ is isolated in $o$;
\item[\rm (b)] if $C\ge 2$ then each of $q$ and $q'$ can have at most $4C$ isolated components;
\item[\rm (c)] if $l$ is the number of components of $q$, then at least $(l-6C)$ of
components of $q$ are connected to components of $q'$ and are not connected to components of $r$ or $r'$; two distinct components of $q$ cannot be connected to the same
component of $q'$. Similarly for $q'$.
\end{itemize}
\end{lem}
\begin{proof}
See the proof of \cite[Lemma 6.3]{Minasyan_classes}.
\end{proof}

\begin{lem} \label{lem:conseq-reg} In the notations of Lemma \ref{lem:regul}, let $m \ge 7K$ and $C=\max\{ \ell(r),\ell(r')\}$.
For any positive integer $d$ there exists a constant $L=L(C,d) \in \N$ such that if $\ell(q)\ge L$ then
there are $d$ consecutive components $p_s,\dots,p_{s+d-1}$ of $q$ and
$p'_{s'},\dots,p'_{s'+d-1}$ of $q'^{-1}$, so that $p_{s+i}$ is connected to $p'_{s'+i}$ for each $i=0,\dots,d-1$.
\end{lem}
\begin{proof}
See the proof of \cite[Lemma 6.5]{Minasyan_classes}.
\end{proof}

\begin{cor}\label{cor:wgeod}
If $m \ge 12K$ then every path $p$ in $\G$, with $\Lab(p)\in \cW(\mathcal{O}, m, X,\cH)$,  is $(4,1)$-quasi-geodesic.
\end{cor}

\noindent\emph{Proof.} Let $p$ be a path in $\G$ such that $\Lab(p)\in \cW(\mathcal{O},m, X,\cH)$. Then $p=r_0 p_1 r_1 \cdots p_l r_l$ where $p_1,\dots, p_l$ are the edges labelled by
elements of $\mathcal H$, and $r_0,\dots, r_l$ are either  trivial paths or edges labelled by elements of $X$.
Let $\lambda(1),\dots,\lambda(l) \in \Lambda$ be such that $h_i=\Lab(p_i) \in H_{\lambda(i)}$ for $i=1,\dots,l$.
Since any combinatorial subpath $p'$ of $p$ still satisfies $\Lab(p')\in \cW(\mathcal{O},m, X,\cH)$,
to prove the lemma it is enough to show that $\ell(p) \le 4 \ell(q)+1$, where $q$ is a geodesic path from $p_+$ to $p_-$ in $\G$. Note that $\ell(q) \le \ell(p) \le 2l+1$.

If $\ell(p) \leq 1$ the claim is obvious, so we assume that $\ell(p)\geq 2$, hence $l\geq 1$.
Note that by the definition of $p$, each $p_i$ is a component of $p$. Let $I \subseteq \{1,\dots,l\}$ be the set of all indices $i$ such that $p_i$ is not connected to a component of $q$ in $\G$.
Lemma \ref{lem:no_back} implies that for each $i \in I$ such $p_i$ is
an isolated component of the  cycle $pq$. Therefore, by Lemma \ref{lem:omega}, we have
$h_i \in \langle \Omega_{\lambda(i)} \rangle$ and $$ \sum\limits_{i\in I} |h_i|_{\Omega_{\lambda(i)}} \leq K\ell(pq)\leq K(4l +2).$$
However, since for $i\in I$, $|h_i|_{\Omega_{\lambda(i)}}> 12K$, we achieve  $|I|\leq K(4l+2)/(12 K)\leq  6Kl/(12K)=l/2$.

Let $I^c:=\{1,\dots,l\} \setminus I$. Then $|I^c| \ge l/2$, and for every $i \in I^c$ the component $p_i$ of $p$ is connected to a component of $q$,
and no two such components of $p$ can be connected to the same component of $q$ (as $p$ is without backtracking by Lemma \ref{lem:no_back}).
Therefore, $q$ has at least $|I^c|$ distinct components and hence
\[
 \pushQED{\qed}
\ell (q ) \ge |I^c| \ge  l/2 \ge \frac{1}{2} (\ell(p)/2- 1/2) = \frac{\ell(p)-1}{4}.\qedhere
 \popQED
\]

The main result of this section is the following.
\begin{thm}\label{thm:combtame}
Suppose that $\{H_\lambda \}_{\lambda \in \Lambda }\h (G,X)$.
Take  $\mathcal{O}=\{\Omega_\lambda\}_{\lambda \in \Lambda}$ and $K>0$ according to the claim of Lemma \ref{lem:omega}.
Let $W$ be any word from $ \cW_0(\mathcal{O}, 12K, X, \cH)$ and let $g\in G$ be the element represented by the word $W$.
Then $g$ is \tame with respect to the action of $G$ on $\G$.
\end{thm}
\begin{proof} Suppose that $W\equiv  h_1x_1h_2 \dots x_{l-1} h_l x_{l}$, where $x_i \in X$ (or $x_i$ is the empty word) and $h_i \in H_{\lambda(i)}$
for some $\lambda(i) \in \Lambda$, $i=1,\dots,l$.
Observe that according to the definition of $ \cW_0(\mathcal{O}, 12K, X, \cH)$, for any $n \in \Z$, $W^n \in \cW(\mathcal{O},12K,X,\cH)$, hence any path labelled by $W^n$ in $\G$ is
$(4,1)$-quasi-geodesic by Corollary \ref{cor:wgeod}.  It follows that the map $n \mapsto g^n$ is a quasi-isometric embedding from $\Z$ to $\G$.
Hence $g$  is loxodromic with respect to the action of $G$ on $\G$.

Let us prove the WPD property.
Fix any $\e >0$ and  $x\in G$, and choose $N \in \N$ so that $lN >6\e_1+1$, where $\e_1:=2|x|_{X\cup\cH}+\e$. Suppose that $f \in G$ satisfies
\begin{equation}\label{eq:f}
\d_{X\cup\cH}(x,fx)< \e \mbox{ and } \d_{X\cup\cH}(g^Nx,fg^Nx)< \e.
\end{equation}
Then $|f|_{X\cup\cH}\le \d_{X\cup\cH}(1,x)+\d_{X\cup\cH}(x,fx)+\d_{X\cup\cH}(fx,f) < 2|x|_{X\cup\cH}+\e=\e_1$. Similarly, $|g^{-N}f^{-1}g^N|_{X\cup\cH}< \e_1$.

Choose words $R$ and $R'$ over $X \cup \cH$ representing the elements $f$ and $g^{-N}f^{-1}g^N$ in $G$ with $\|R\|, \|R'\| < \e_1$.
 Let $o=r q r' q'$ be the cycle in $\G$ starting at $1$ such that $\Lab(r) \equiv R$,  $\Lab( q )\equiv W^N$,  $\Lab(r') \equiv R'$ and $\Lab( q' )\equiv W^{-N}$.

Let $p_1,\dots,p_{Nl}$ and $p'_1,\dots,p'_{Nl}$ be the lists of components of $q$ and $q'^{-1}$ in the order of their occurrence.
By Lemma \ref{lem:regul}.(c) there exists $k\in \N$, $k \le 6 \e_1+1$, such that $p_k$ is connected to a component of $p'_{k'}$ of $q'^{-1}$
and $p_k$ is not connected to any component of $r$. Thus there is a path $s$ from $(p_k)_-$ to $(p'_{k'})_-$ in $\G$, such that $s$ is
labelled either by the empty word (if $(p_k)_-=(p'_{k'})_-$) or by a letter from $H_{\lambda(j)}$, for some $j\in \{1,\dots,l\}$ (see Figure~\ref{fig:3}).
Then $\ell(s) \le 1$ and one can consider the cycle
$o_1=rq_1 s q_1'$ in $\G$, where $q_1$ is the initial segment of $q$ from $q_-=f$ to $(p_k)_-$ and $q_1'$ is the terminal segment of $q'$ from $(p'_{k'})_-$ to $q'_+=1$.

\begin{figure}[!ht]
  \begin{center}
   \input{fig3.pstex_t}
  \end{center}
  \caption{}\label{fig:3}
\end{figure}

Then $p_1',\dots,p_{k'-1}'$ is the list of components of $q_1^{-1}$ and if $k'-1 >6\e_1+1$, one can apply Lemma \ref{lem:regul}.(c) again to the cycle $o_1$ to
find $k_1' \le 6\e_1+1$ such that $p_{k_1'}'$ is connected to a component $p_{k_1}$ of $q_1$ and is not connected to a component of $r$.
In this case we replace $k$ with $k_1$ and $k'$ with $k_1'$. Thus, without loss of generality, we can further assume that $\max\{k,k'\}\le 6\e_1+2$. It follows that
$\ell(q_1)\le 2(k-1)+1\le 12 \e_1+3$; similarly, $\ell(q_1') \le 12\e_1+3$.

Let $y,z$ and $h$ be the elements of $G$ represented by the words $\Lab(q_1^{-1})$, $\Lab(q_1'^{-1})$ and $\Lab(s^{-1})$ respectively. Then $f=zhy$ in $G$,
where $h \in H_{\lambda(j)}$ for some $j \in \{1,\dots,l\}$.

By construction, $y,z$ belong to the subgroup of $G$ generated by the finite set of elements $A:=\{ x_1,\dots,x_l,h_1,\dots, h_l\}$ and $|y|_A,|z|_A\le 12\e_1 +3$.
On the other hand, note that if $h \neq 1$ in $G$ then $s$ must be an isolated $H_{\lambda(j)}$-component of the cycle $o_1$ (because $q$ and $q'$ are without backtracking
by Lemma \ref{lem:no_back} and $p_k$ is not connected to a component of $r$).
Hence we can use Lemma \ref{lem:omega} to conclude that $h \in \langle \Omega_{\lambda(j)}\rangle$ and $|h|_{\Omega_{\lambda(j)}} \le K \ell(o_1) \le K(25\e_1+7)$.

Let $B \subseteq G$ be the finite subset defined by $B=\{z \in \langle A \rangle \mid |z|_A\le 12\e_1+3\}$.
We have shown that any element $f$ satisfying \eqref{eq:f} belongs to the subset
$$\bigcup_{i=1}^l \left(B \cdot
\left\{\left. h \in\gen{\Omega_{\lambda(i)}} \,\right| \, |h|_{\Omega_{\lambda(i)}}\le K(25 \e_1+7)\right\}
\cdot B\right),$$
which is finite as a finite union of products of finite subsets. Thus we have shown that  the element $g$ is WPD.
\end{proof}


\section{Special elements in acylindrically hyperbolic groups}

In this section we fix a group $G$ and a hyperbolic space $(\cS,\d)$ where $G$ acts by isometries and coboundedly.
By Lemma \ref{lem:svarc-milnor}, there is a generating set $X$ of $G$ such that $(G,\d_X)$ is equivariantly quasi-isometric to $\cS$.
It follows that $g\in G$ is a \tame element with respect to the action of $G$ on $\cS$ if and only if $g$ is a \tame element with respect to the action of $G$ on
$\Gamma(G,X)$. Thus, without loss of generality, we can work with either $\cS$ or $\Gamma(G,X)$.

The following observation will be useful.
\begin{lem}\label{lem:tameinS}
Suppose that $X_1$ is a subset of $G$ containing $X$.
If $g$ is a \tame element with respect to the action of $G$ on $\Gamma(G,X_1)$ then $g\in \LWPD$.
\end{lem}
\begin{proof}
It is enough to show that $g$ is \tame with respect to the $G$-action on $\Gamma(G,X)$. Since the action of $g$ is loxodromic on $\Gamma(G,X_1)$
there exist $\mu\geq 1$ and $c\geq 0$ such that $|n|\leq \mu |g^n|_{X_1}+c$ for all $n\in \mathbb{Z}$.
Since   $|h|_{X_1}\leq |h|_X$ for all $h\in G$, we get  $|n|\leq \mu |g^n|_{X}+c$ for all $n\in \mathbb{Z}$, which shows
that $g$ acts as a  loxodromic element on $\Gamma(G,X)$.

Similarly, since $\d_X(x,y) \ge \d_{X_1}(x,y)$ for any $x,y  \in G $, it easily follows that any WPD element with respect to the action of $G$ on $\Gamma(G,X_1)$ is also a WPD element with respect to the $G$-action on
$\Gamma(G,X)$.
\end{proof}

\subsection{Creating new \tame elements }
The purpose of this section is to develop basic tools for working with \tame elements and producing new \tame elements from a number of old ones.

\begin{lem}\label{lem:tame1}
Let $\{H_\lambda\}_{\lambda \in \Lambda}$ be a family of subgroups of $G$ that is hyperbolically embedded in $(G,X)$.
Set $\mathcal H =\sqcup_{\lambda \in \Lambda} (H_\lambda \setminus \{1\})$ and take an arbitrary finite subset $\{\lambda_1,\dots,\lambda_l \}\subseteq \Lambda$, $l \ge 1$.
Consider any subset $F$ of $G$ such that $|F\setminus X|< \infty$ and if $l=1$ then $F\cap H_{\lambda_1}=\emptyset$.
Then there exists a finite subset $\Phi  \subseteq G$ such that  for any $f_i\in F$ and $g_i\in H_{\lambda_i}\setminus \Phi,$  $i=1,\dots,l$, the element $g:=g_1f_1g_2f_2\dots g_lf_l$ has the
following properties:
\begin{enumerate}
\item[{\rm (a)}] $g$ is a \tame element with respect to the action on $\G$; in particular, $g\in \LWPD$; 
\item[{\rm (b)}] $g$ is not commensurable with any element $h \in \bigcup_{\lambda \in \Lambda} H_\lambda$ in $G$.
\end{enumerate}
\end{lem}

\begin{proof}
By Lemmas \ref{lem:changegenset} and \ref{lem:tameinS}, we can replace $X$ with $X \cup F$ to assume that $F \subseteq X$.
Let $\mathcal{O}=\{\Omega_{\lambda}\}_{\lambda\in\Lambda}$ and $K \in \N$ be the collection of finite subsets and the constant from the claim of Lemma \ref{lem:omega}.
We can then define the finite subset $\Phi \subseteq G$ by setting $\Phi\coloneq\bigcup_{j=1}^l\{h \in \gen{\Omega_{\lambda_j}} \mid |h|_{\Omega_{\lambda_j}} \le 12K\}$.
Now part (a) follows from
the assumptions together with the claims of Theorem \ref{thm:combtame} and Lemma \ref{lem:tameinS}.

To prove part (b) notice that for every $h \in \bigcup_{\lambda \in \Lambda} H_\lambda$, the cyclic subgroup
$\gen{h}$ acts with bounded orbits on the Cayley graph $\G$. On the other hand, all the orbits of $\gen{g}$ are unbounded because $g$ is loxodromic by part (a).
Thus a non-zero power of $g$ cannot be conjugate to a power of $h$ in $G$, i.e.,  (b) holds.
\end{proof}

Applying Lemma \ref{lem:tame1} in the special case when $l=1$ we obtain the following statement, generalizing \cite[Corollary 6.12]{DGO}:
\begin{cor}\label{cor:elemhe}
Suppose that $\{H_\lambda\}_{\lambda \in \Lambda}\h (G,X)$. Then for any $\lambda \in \Lambda$ and $f\in G\setminus H_\lambda$, there exists a finite subset $\Phi\subset G$  such that
for all $g\in H_\lambda\setminus\Phi$ the element $gf$ is  \tame with respect to the action of $G$ on $\Gamma (G, X\sqcup \cH)$; in particular, $gf \in \LWPD$.
\end{cor}

Recall that by Lemma \ref{lem:elemrem}, every $g\in \LWPD$ belongs to the virtually cyclic subgroup
$$E_G^+(g)=\{ f\in G \mid fg^nf^{-1}=g^n \text{ for some } n\in \N \}\leqslant E_G(g),$$ and  $|E_G(g):E_G^+(g)|\le 2$.
This lemma also implies that $E_G(g)=E^+_G(g)$ if and only if $E_G(g)$ has infinite center.

\begin{lem}\label{lem:newtame1}
Let $\{g_1,\dots, g_l\}$ be a non-empty family of pairwise non-commensurable \tame elements with respect to the action of $G$ on $\cS$.
Consider any subset $F\subseteq G$ such that  $|F \setminus X|< \infty$ and if $l=1$ then $F\cap E_G(g_1)=\emptyset$.

Then there exists  $N_1=N_1(F) \in \mathbb N$ such that  for arbitrary $f_i\in F$ and $m_i\in  \N$ with $|m_i|\geq N_1$, $i=1,\dots,l$,
the element $g:=g_1^{m_1}f_1g_2^{m_2}f_2\dots g_l^{m_l}f_l$ belongs to $\LWPD$ and is not commensurable with any $g_i$, $i=1,\dots, l$. Moreover,
\begin{enumerate}
\item[{\rm (i)}] if $l=1$ then for every $y\in E_G(g)$ there exist $\xi, \zeta \in \Z$ such that  $g^\xi y g^{\zeta}\in E_G(g)\cap E_G(g_1)$;
\item[{\rm (ii)}] if $l \ge 3$ and $f_l=1$ then $E_G(g)=E_G^+(g)$ and for every $y\in E_G(g)$ there exist $\xi, \zeta \in \Z$ satisfying  $g^\xi y g^{\zeta}\in E_G(g_l)\cap E_G(g_1)$.
\end{enumerate}
\end{lem}
\begin{proof} Recall that by Corollary \ref{cor:wpd-he} the family $\{E_G(g_i)\}_{i=1}^l$ is hyperbolically embedded in $(G,X)$. As before, in view of Lemmas \ref{lem:changegenset} and \ref{lem:tameinS},
we can assume that $F \subseteq X$.
Set $\mathcal H :=\sqcup_{i=1}^l (E_G(g_i) \setminus \{1\})$, and let the finite subsets
$\Omega_i \subset E_G(g_i)$, $i=1,\dots,l$, and $K \in \N$ be chosen according to Lemma \ref{lem:omega}.
Take $N_1 \in \N$ so that $g_i^m \notin \Phi\coloneq \bigcup_{j=1}^l\{h \in \gen{\Omega_j} \mid |h|_{\Omega_j} \le 12K\}$
for any $i =1,\dots,l$, whenever $|m|\ge N_1$. Consider any $g=g_1^{m_1}f_1g_2^{m_2}f_2\dots g_l^{m_l}f_l$ with $f_i \in F$ and $|m_i| \ge N_1$, $i=1,\dots,l$.
By Lemma~\ref{lem:tame1}, $g \in \LWPD$ and it is not commensurable with with any $g_i$, $i=1,\dots, l$. So, it remains to prove claims (i) and (ii).

Consider any $y\in E_G(g)$. By Lemma \ref{lem:elemrem}, there exist $m\in \mathbb N$ and $\epsilon \in  \{-1,1\}$ such that
\begin{equation}
\label{eq:m}
yg^my^{-1}=g^{\epsilon m}.
\end{equation}
Let $L=L(C, 2l)$ be the constant provided by Lemma \ref{lem:conseq-reg}, where $C:=\d_{X\sqcup \mathcal H}(1,y)$. Evidently we can take $m$ in \eqref{eq:m} to be large enough so that $ml \ge L$.

Let $U$ be a word over $X\sqcup \mathcal H$ representing $y$, with $\|U\|=C$, and let $W \equiv h_1f_1h_2 f_2\dots h_lf_l$
be the word from $\cW (\mathcal{O},12K,X,\cH)$ representing $g$, where $h_i:=g_i^{m_i} \in E_G(g_i) \setminus\{1\}$ and $\mathcal{O}=\{\Omega_j\}_{j=1}^l$.
Consider a cycle $o=rqr'q'$ in $\G$, where $\Lab( r )\equiv U$, $\Lab( q )\equiv W^m$, $\Lab(r')\equiv U^{-1}$ and
$\Lab(q')\equiv W^{-\epsilon m}$. 
Then $\ell(q)\ge ml \ge L$, hence by
Lemma \ref{lem:conseq-reg} there are $2l$ consecutive components of $q$ connected to $2l$ consecutive components of $q'^{-1}$.

Suppose, first, that $l=1$. Then there is an $E_G(g_1)$-component $p$ of $q$ connected to  $p'$, an $E_G(g_1)$-component of $q'^{-1}$. That is, there is a path $s$ in $\G$ with $s_{-}=p_{-}$, $s_{+}=p'_{-}$ such that
$\Lab(s)$ represents an element $z\in E_G(g_1)$. Note that $\Lab( p ) \equiv h_1$ and $\Lab(p')\equiv h_1^{\epsilon }$.

Let $q_1$ be the subpath of $q$ starting at $r_+=q_-$ and ending at
$p_-=s_-$; let $q_1'$ be the subpath of $q'$ starting at $s_+=p'_-$
and ending at $q'_+=r_-$.
Consider the cycle $o_1=rq_1sq_1'$ in $\G$. If $\epsilon=1$ we see that $\Lab(q'_1) \equiv W^\xi$ for some integer $\xi\le 0$ and
$\Lab(q_1) \equiv W^\zeta$ for an integer $\zeta \ge 0$. Therefore $g^\xi y g^\zeta= z^{-1}$ in $G$.
Recall that $z^{-1} \in E_G(g_1)$ and the left hand side of the latter equality belongs to $E_G(g)$,
hence $g^\xi y g^\zeta\in E_G(g) \cap E_G(g_1)$.
Similarly, in the case when $\epsilon=-1$ we see that
$g^\xi y g^\zeta= g_1^{m_1} z^{-1} \in E_G(g) \cap E_G(g_1)$ for some $\xi,\zeta \in \Z$. Thus part (i) is proved.

To prove part (ii), assume that $l \ge 3$. Then three consecutive components  $p_1$, $p_2$, $p_3$ of $q$,
with $\Lab(p_i) \equiv h_i$, $i=1,2,3$, are connected to three consecutive components $p_1'$, $p_2'$, $p_3'$, of $q'^{-1}$.
Since an $E_G(g_i)$-component cannot be connected to an $E_G(g_j)$-component if $i \neq j$, we see that $p_i'$ must be labelled by $h_i^\epsilon$, $i=1,2,3$.
However, if $\epsilon=-1$, any triple of consecutive components of $q'^{-1}$ would be labelled by a cyclic permutation of the sequence $h_3^{-1}$, $h_2^{-1}$, $h_1^{-1}$,
which cannot give the sequence $h_1^{-1}$, $h_2^{-1}$, $h_3^{-1}$. Thus $\epsilon=1$, implying that $y \in E^+_G(g)$. Since the latter is true for any $y \in E_G(g)$
we can conclude that $E_G(g)=E^+_G(g)$.

For the last claim of part (ii), suppose that $f_l=1$ and choose consecutive components $p_l$ and $p_1$ of $q$ that are connected to consecutive components $p_l'$ and $p_1'$ of $q'^{-1}$,
so that $p_i$ and $p_i'$ are $E_G(g_i)$-components of the corresponding paths for $i=1,l$.
It follows that for any path $s$ in $\G$ joining $(p_l)_+=(p_1)_-$ with $(p_l')_+=(p'_1)_-$, $\Lab(s)$ represents an element $z \in E_G(g_l) \cap E_G(g_1)$.
Since $f_l=1$ and $\epsilon=1$ the label of the subpath of $q'$ from $(p_l')_+=s_+$ to $q'_+=r_-$ represents a negative power of $g$, and the label of the subpath of $q$ from $r_+=q_-$ to $(p_l)_+=s_-$
represents a positive power of $g$. Thus there are integers $\xi<0 $ and $\zeta>0$ such that $g^\xi y g^{\zeta}=z^{-1}\in E_G(g_l)\cap E_G(g_1)$.
This completes the proof  of the lemma.
\end{proof}

\begin{lem} \label{lem:newtamenc} Let $g\in \LWPD$ and $f\in G\setminus E_G(g)$. For any finite subset $Y$ of $G$, there exists $N_2\in \N$ such that $g^n f \in \LWPD$ and is not commensurable
with any $y\in Y$ whenever $|n|\geq N_2$.
\end{lem}
\begin{proof}
By Corollary \ref{cor:wpd-he} $E_G(g)\h(G,X)$. Let $Y_1\subseteq Y$ be a maximal subset of pairwise non-commensurable elements  such that each $y\in Y_1$ is \tame with respect to the action of $G$ on $\cS$
and is not commensurable with $g$. By Corollary \ref{cor:wpd-he}, $\{ E_G(g)\}\sqcup \{E_G(y) \mid  y\in Y_1\}\h (G,X)$, hence we can apply Lemma \ref{lem:tame1} to find $N_2 \in \N$ such that
the element $g^n f$ belongs to $\LWPD$ and is not commensurable with any element from the subset $\{g\} \cup Y_1$ whenever $|n|\geq N_2$.

Suppose that there is an integer $n$ such that $|n| \ge N_2$ and
$g^nf$ is commensurable with some $z\in Y$. Then $z \in Y \setminus Y_1$,
$z$ is not commensurable with any element of $\{g\} \cup Y_1$ and $z \in \LWPD$ by Remarks \ref{rem:powers} and \ref{rem:tame-conj}.
This contradicts the maximality of $Y_1$. Thus the lemma is proved.
\end{proof}

\subsection{Special elements}
Let $H$ be a subgroup of $G$. In this subsection we develop the theory of $H$-special elements. Many ideas and statements in this subsection are similar to those of \cite[Section 3]{MO}
(see also \cite[Subsection 6.2]{DGO} for the case $H=G$).

\begin{lem}\label{lem:E_G(H)}
Let $H$ be a non-elementary subgroup of $G$ such that $H \cap \LWPD\neq \emptyset$. Then the subgroup $\displaystyle E_G(H):=\bigcap_{h \in H \cap \LWPD} E_G(h)$ is the unique maximal finite subgroup of $G$ normalized by $H$.
\end{lem}

\begin{proof}
If a finite subgroup $F\leqslant G$ is normalized by $H$, then $|H:C_H(F)|<\infty$, where $C_H(F)$ denotes the centralizer of $F$ in $H$.
Therefore for every $h\in H$ and $f\in F$, there is $n\in \mathbb N$ such that $fh^nf^{-1}=h^n$. Hence, by Lemma \ref{lem:elemrem}(b), $F\leqslant E_G(h)$ for all $h\in H \cap \LWPD$, thus $F\leqslant E_{G}(H)$.

Let $g\in H \cap \LWPD$. Since $H$ is non-elementary, there exists $a\in H \setminus E_G(g)$. Then $aga^{-1}\in H \cap \LWPD$ by Remark \ref{rem:tame-conj}.
If the intersection $E_G(g)\cap E_G(aga^{-1})$ is infinite then, according to Remark \ref{rem:elem_inter}, there exist $m,n \in \mathbb Z\setminus \{0\}$ such that $ag^na^{-1}=g^m$, which implies that $a\in E_G(g)$
(by Lemma \ref{lem:elemrem}.(c)).
This contradiction shows that $E_G(H)\leqslant E_G(aga^{-1})\cap E_G(g)$ is finite. The fact that $E_G(H)$ is normalized by $H$ follows from its definition together with Remark \ref{rem:tame-conj} and
Lemma  \ref{lem:elemrem}: the latter two statements imply that for any $h \in H \cap \LWPD$ and any $f \in H$, $fhf^{-1} \in H \cap \LWPD$ and $fE_G(h)f^{-1}=E_G(fhf^{-1})$.
\end{proof}

\begin{rem} In the case when $H=G$, the statement of Lemma \ref{lem:E_G(H)} is proved in \cite[Lemma~6.15]{DGO},
where $K(G)$ is used to denote the largest finite normal subgroup of $G$, which is $E_G(G)$ in our notation.
\end{rem}

Set $\LWPDP :=\{g \in \LWPD \mid E_G(g)=E_G^+(g)\}$.

\begin{lem}\label{lem:lwpd+}
Let $H\leqslant G$ be a non-elementary subgroup such that $H \cap \LWPD \neq \emptyset$. For every finite subset $Y \subset G$ there exists $h \in H \cap \LWPDP$ that is not commensurable in $G$ with any element of $Y$.
In particular, $H \cap \LWPDP$ contains infinitely many pairwise non-commensurable (in $G$) elements.
\end{lem}

\begin{proof}Let $Y_1=\{g_1,\dots,g_l\} \subset Y$ be a maximal subset consisting of pairwise non-commensurable \tame elements (thus any element from $Y \cap \LWPD$ is commensurable to some
element from $Y_1$). If $l=0$ we understand that $Y_1$ is empty.

Take any element $g  \in H \cap \LWPD$. Since $H$ is non-elementary, there exists $f \in H \setminus E_G(g)$ and we can apply Lemma~\ref{lem:newtamenc}, to find $n \in \N$ such that
$g_{l+1}:=g^n f \in H \cap  \LWPD$ and $g_{l+1}$ is not commensurable with any element of $Y_1$. Applying this lemma two more times, we get elements $g_{l+2}, g_{l+3} \in H \cap \LWPD$
such that $g_i$ is not commensurable to $g_j$ whenever $1 \le i<j \le l+3$.

Now, by Lemma \ref{lem:newtame1}, there is $m\in \N$ such that the element $h:=g_1^{m}g_2^{m}\dots g_{l+3}^{ m} \in H$
belongs to $\LWPDP$ and is not commensurable with  any element from $\{g_1,\dots ,g_{l+3}\}$. Finally, if $h$ was commensurable to some $z \in Y$ then
$z \in \LWPD$ (by Remarks \ref{rem:powers} and \ref{rem:tame-conj}) and $z$ would be non-commensurable with any $y \in Y_1$, contradicting the choice of $Y_1$. Thus the lemma is proved.
\end{proof}

\begin{lem}\label{lem:intersec_EG}
Given two non-commensurable elements $g_1,g_2\in \LWPDP$, there exists $h\in \gen{g_1,g_2}\cap \LWPDP$ with the properties that $h$ is not commensurable with $g_i$, $i=1,2$,
$E_G(h)= \gen{h}\cdot (E_G(g_1)\cap E_G(g_2))$ and $h\in C_G(E_G(h))$, that is $E_G(h)\cong \gen{h}\times (E_G(g_1)\cap E_G(g_2))$.
\end{lem}
\begin{proof}
By Lemma \ref{lem:elemrem}, Remarks \ref{rem:powers} and \ref{rem:elem_inter}, we can replace $g_i$ with its power to assume that $g_i$ is central in $E_G(g_i)$, $i=1,2$.
The subgroup $\gen{g_1,g_2}\leqslant G$ is non-elementary because $g_1$ and $g_2$ are non-commensurable, hence, according to Lemma \ref{lem:lwpd+}, there is $g_3 \in \gen{g_1,g_2} \cap \LWPDP$
that is not commensurable with $g_1$ and $g_2$.

Now, by Lemma \ref{lem:newtame1}, we can choose $m \in \N$ so that the element $h:=g_1^m g_3^m g_2^m$ belongs to $\gen{g_1,g_2} \cap \LWPDP$, is not commensurable with $g_1$ and $g_2$, and
satisfies $E_G(h) \subseteq \gen{h} (E_G(g_1) \cap E_G(g_2)) \gen{h}$. Thus $E_G(h)\leqslant \gen{ h, E_G(g_1)\cap E_G(g_2)}$.
But each of $g_1$ and $g_2$ commutes with $E_G(g_1) \cap E_G(g_2)$, hence so does $h$, and so Lemma \ref{lem:elemrem} yields that  $E_G(g_1)\cap E_G(g_2) \leqslant E_G(h)$. Thus
$E_G(h)= \gen{ h, E_G(g_1)\cap E_G(g_2)}$. Finally, note that $h$ has infinite order and
$|E_G(g_1) \cap E_G(g_2)|<\infty$ by Remark \ref{rem:elem_inter},
which implies that $\gen{h} \cap E_G(g_1) \cap E_G(g_2)=\{1\}$. Therefore $E_G(h) \cong \gen{h} \times \left(  E_G(g_1) \cap E_G(g_2)\right)$, as claimed.
\end{proof}

\begin{lem}\label{lem:E_G(H)+}
Let $H\leqslant G$ be a non-elementary subgroup such that $H\cap \LWPD\neq \emptyset$.  Then $E_G(H)=\bigcap_{g\in H \cap \LWPDP} E_G(g). $

\end{lem}
\begin{proof}
By Lemma \ref{lem:lwpd+}, there exist two non-commensurable elements  $g_1,g_2\in H \cap \LWPDP$.
Then $E_G(g_1)\cap E_G(g_2)$ is finite (Remark \ref{rem:elem_inter}), and therefore $\bigcap_{g\in H \cap \LWPDP} E_G(g)$ is finite.
Notice that the set  $H \cap \LWPDP$ is closed under $H$-conjugation and $E_G(hgh^{-1})=hE_G(g)h^{-1}$ for any $g \in \LWPD$ and any $h \in H$.
Hence $H$ normalizes the finite subgroup $\bigcap_{g\in \LWPDHP} E_G(g) \leqslant G$.
Clearly $E_G(H)=\bigcap_{g\in H \cap \LWPD} E_G(g)\leqslant \bigcap_{g\in H \cap \LWPDP} E_G(g)$. To obtain the desired equality,
it remains to recall that $E_G(H)$ is the unique maximal finite subgroup of $G$ normalized by $H$ by Lemma \ref{lem:E_G(H)}.
\end{proof}

\begin{defn} Let $H$ be a non-elementary subgroup of $G$. An element $g\in H$ will be called {\it $H$-special} if $g\in \LWPD$, $E_G(g)=\gen{g}\cdot E_G(H)$ and
$g\in C_G(E_G(H))$ (i.e., $E_G(g)\cong\gen{g}\times E_G(H)$).
The set of all $H$-special elements will be denoted by $S_G(H,\cS)$.
\end{defn}

The next statement is an analogue of \cite[Lemma 3.8.(ii)]{AMO}.
\begin{lem}\label{lem:special}
Let $H\leqslant G$ be a non-elementary subgroup such that $H \cap \LWPD \neq \emptyset$. Then $S_G(H,\cS)$ is non-empty.
\end{lem}
\begin{proof}
Let $B$ be the set of all elements $h\in H \cap \LWPDP$ such that $E_G(h)$ is the direct product of $\gen{h}$ with some finite subgroup $K_h$ of $G$.
By Lemma \ref{lem:lwpd+} there exists two non-commensurable elements in $H \cap \LWPDP$, and so, by Lemma \ref{lem:intersec_EG} and Remark \ref{rem:elem_inter}, the set
$B$ is non-empty. Let $h\in B$ be such that $|K_{h}|$ is minimal. We will show that $K_h=E_G(H)$ and thus $h\in S_G(H,\cS)$.

Notice that $E_G(H)\leqslant K_{h}$, as $E_G(H)\leqslant E_G(h)$ and $K_h$ is the unique maximal finite subgroup of $E_G(h)$ by definition.
Arguing by contradiction, assume that there exists a finite order element $x\in K_{h}\setminus E_G(H)$.
Then, according to Lemma \ref{lem:E_G(H)+}, there is $g\in H \cap \LWPDP$ such that $x\notin E_G(g)$. If $g$ and $h$ are non-commensurable,
using Lemma \ref{lem:intersec_EG} we can find $f\in H \cap \LWPDP$ such that  $E_G(f)=\gen{f} \cdot (E_G(h)\cap E_G(g))$ and $f\in C_G(E_G(h)\cap E_G(g))$.
Moreover, Remark \ref{rem:elem_inter} shows that $E_G(h)\cap E_G(g)$ is finite, and so it is contained in $K_h$.
Thus $f\in B$ and, as $x\notin E_G(h)\cap E_G(g)$, we have that $|K_f|=|E_G(h)\cap E_G(g) |< |K_h|$, contradicting the minimality of $|K_h|$.

It remains to consider the case when $g$ is commensurable with $h$. By Lemma \ref{lem:lwpd+}, there exists $g'\in H \cap \LWPDP$
non-commensurable with $g$. Then, by Lemma \ref{lem:intersec_EG}, we can find $f\in H \cap \LWPDP$ such that
$E_G(f)=\gen{f} \cdot (E_G(g')\cap E_G(g))$, $f\in C_G(E_G(g')\cap E_G(g))$ and $f$ is not commensurable with $g$, and hence $f \napg h$.
Moreover, since $x\notin E_G(g)$, we have that $x\notin E_G(f)$ as the torsion of
$E_G(f)$ is exactly $E_G(k)\cap E_G(g)$. Then $f$ has the same properties as $g$ in the previous paragraph, which leads to a contradiction with the minimality of $|K_h|$.
Therefore $K_h=E_G(H)$ and so $h \in S_G(H,\cS)\neq \emptyset$.
\end{proof}

The following lemma is similar to \cite[Lemma 3.6]{MO}:

\begin{lem}\label{lem:SG}
Suppose that $H \leqslant G$, $g\in S_G(H,\cS)$ and $x\in C_H(E_G(H))\setminus E_G(g)$. Then there exists $N_3\in \mathbb N$ such that $g^nx\in S_G(H,\cS)$ for any $n\in \mathbb Z$ with $|n|\geq N_3$.
\end{lem}

\begin{proof}
By Lemma  \ref{lem:newtame1} there exists $N_3\in \N$ such that for all $n \in \Z$ with $|n|\geq N_3$,
$h:=g^n x \in H \cap \LWPD$ and this element is not commensurable with $g$. Part (i) of this lemma also shows that $E_G(h) \subseteq \gen{h} (E_G(g) \cap E_G(h)) \gen{h}$.
Since $g$ is $H$-special and the subgroup $E_G(g) \cap E_G(h)$ is finite (by Remark \ref{rem:elem_inter}), we see that $E_G(g) \cap E_G(h) \leqslant E_G(H)$.
Recalling Lemma \ref{lem:E_G(H)}, we obtain
$$E_G(h) \leqslant \gen{h, E_G(H)} =\gen{h} E_G(H) \leqslant E_G(h),$$
thus $E_G(h) =\gen{h} E_G(H)$. It remains to observe that $h \in C_H(E_G(H))$ because both $g$ and $x$ belong to this centralizer by the assumptions.
Hence $h \in S_G(H,\cS)$, as claimed.
\end{proof}

\begin{prop}\label{prop:special}
Let $H$ be a non-elementary subgroup of $G$ with $H \cap \LWPD\neq \emptyset$. Then $C_H(E_G(H))$ is generated by the set $S_G(H,\cS)$. In particular $\gen{S_G(H,\cS)}$ has finite index in $H$.
\end{prop}

\begin{proof} The proof is  omitted, as it is identical to the proof of \cite[Proposition 3.3]{MO}, modulo Lemmas~\ref{lem:special} and \ref{lem:SG}.
\end{proof}


\section{Technical lemmas}\label{sec:techn}
The goal of this section is to prove several auxiliary statements that will help in establishing the claim of the main Theorem \ref{thm:comm}. All of these statements are analogous to the ones from \cite[Section 4]{MO}.
Throughout this section $G$ will denote a group acting coboundedly by isometries on a hyperbolic space $(\cS,\d)$.
Let $X$ be the generating set of $G$ given by Lemma  \ref{lem:svarc-milnor}, so that $\Gamma(G,X)$ is equivariantly quasi-isometric to $\cS$.

The main technical tool is the following lemma, which generalizes \cite[Lemma 4.4]{MO}. Roughly speaking,
it says that the products of large powers of WPD loxodromic elements are commensurable only in the ``obvious'' cases.

\begin{lem}\label{lem:cyclicperm}
Let $\{g_1,\dots, g_l\}\subseteq \LWPD$, $l\geq 2$, be a set of pairwise non-commensurable \tame elements. Let $F$ be a subset of $G$ such that $|F \setminus X|<\infty$ (e.g., $F$ could be finite).

There exists  $N_4 \in \mathbb N$ such that for any permutation $\sigma$ of $\{1,\dots, l\}$ and
arbitrary elements $h_i\in E_G(g_{\sigma(i)})$, $i=1,\dots,l$, of infinite order, the following holds. Suppose that
$(g_1^{m_1}g_2^{m_2}\dots g_l^{m_l})^{\zeta}$ is conjugate to $(h_1^{n_1}f_1h_2^{n_2} \cdots h_l^{n_l}f_l)^{\eta}$ in $G$,
for  some $f_i\in F$, $\eta, \zeta\in \mathbb{N}$, and $m_i,n_i \in  \mathbb Z,$ $|m_i|\geq N_4$, $|n_i|\geq N_4$ for all $i=1,\dots,l$.
Then $\eta=\zeta$ and there is $k\in \{0,\dots, l-1\}$ such that $\sigma$ is a cyclic shift by $k$, that is $\sigma(i)\equiv i+k (\text{mod } l)$
for all $i\in \{1,2,\dots, l\}$, and $f_j\in E_G(g_{\sigma(j)})E_G(g_{\sigma(j+1)})$ when $j=1,2,\dots, l-1$, $f_l \in  E_G(g_{\sigma(l)})E_G(g_{\sigma(1)})$.
\end{lem}

\begin{proof} This proof is very similar to the proof of \cite[Lemma 4.4(2)]{MO}, using the appropriate references.

By Corollary \ref{cor:wpd-he} the family $\{E_G(g_i)\}_{i=1}^l$ is hyperbolically embedded in $(G,X)$, and, by Lemma \ref{lem:changegenset}, we can enlarge $X$ to ensure that $F \subseteq X$.
Set $\mathcal H :=\sqcup_{i=1}^l (E_G(g_i) \setminus \{1\})$ and let the finite subsets $\Omega_i \subset E_G(g_i)$, $i=1,\dots,l$ and $K \in \N$
be chosen according to Lemma \ref{lem:omega}. Let $S$ be the finite subset of $G$ given by $S:=\bigcup_{j=1}^l\{h \in \gen{\Omega_j} \mid |h|_{\Omega_j} \le 7K\}$.

First, let us show that for each $i$ there is $K_i \in \N$ such that $g^k \notin S$ whenever $g \in E_G(g_i)$ is an element of infinite order and $|k|\ge K_i$. Indeed, since $|E_G(g_i):\gen{g_i}|<\infty$
we see that every infinite order element $g\in E_G(g_i)$ in fact belongs to the subgroup $E^+_G(g_i)$. Note that the center of $E^+_G(g_i)$ has finite index in it
(e.g., by the last assertion of Lemma \ref{lem:elemrem}). Hence all the elements of finite order form a finite normal subgroup
$T_i \lhd E^+_G(g_i)$, and the quotient $E^+_G(g_i)/T_i$ is an infinite cyclic group, generated by the coset $yT_i$, for some $y \in E^+_G(g_i)$. Since $y$ has infinite order and the set $ST_i$ is finite,
there exists $K_i \in \N$ such that $y^k \notin ST_i$ provided $|k| \ge K_i$. Then for any infinite order element $g \in E_G(g_i)$ there is $m \in \Z \setminus\{0\}$ with $g \in y^{m}T_i$.
Thus for any $k \in \Z$,  $g^k \in y^{km}T_i $. But if  $|k|\ge K_i$ then $|km| \ge K_i$ and hence $y^{km}T_i \cap S = \emptyset$, implying that $g^k \notin S$, as required.

Now,  set $N_4:= \max\{K_i \mid i=1,\dots,l\}$.
Choose arbitrary elements  $f_1,\dots,f_l \in F$ and assume that
$ b\left( g_1^{m_1}g_2^{m_2}\dots g_l^{m_l} \right)^\zeta b^{-1}=
\left( h_1^{n_1}f_1 h_2^{n_2} f_2 \dots h_l^{n_l}f_l\right)^\eta$ in $G$,
for some infinite order elements $h_i \in E_G(g_{\sigma(i)})$, where $\sigma$ is a permutation of $\{1,\dots,l\}$,
and some $b \in G$, $\zeta,\eta \in \mathbb N$, $m_i,n_i\in \mathbb Z$ with $|m_i|,|n_i|\ge N_4$, $i=1,2,\dots,l$.
Then, for every $n \in \mathbb N$ we have
\begin{equation}\label{eq:b-zeta} b\left( g_1^{m_1}g_2^{m_2}\dots g_l^{m_l} \right)^{n \zeta} b^{-1}=
\left( h_1^{n_1}f_1 h_2^{n_2} f_2 \dots h_l^{n_l}f_l\right)^{n \eta} .
\end{equation}
Let $U_i$, $V_i$ and $W_i$ be the letters from $\mathcal H$ and from $X$ representing the elements $h_{i}^{n_i}$,
$g_i^{m_i}$ and $f_i$, $i=1,\dots,l$, respectively. By our choice of $m_i$ and $n_i$, the words $(V_1V_2\dots V_l)^{n\zeta}$,
and
$\left( U_1 W_1 U_2 W_2\dots U_l W_l \right)^{n\eta}$ are in $\cW(\mathcal{O}, 7K,X,\cH)$ for all $n\in \mathbb Z$, where $\mathcal{O}\coloneq \{\Omega_j\}_{j=1}^l$.

Choose a shortest word $B$ over $X\cup \mathcal H$ representing $b$ in $G$.
Set $\e=|B|$ and let $L=L(\e,2l)\in \mathbb N$ be the constant given by Lemma \ref{lem:conseq-reg}. Take $n \in \N$ to be sufficiently large so that
$nl>6 \e$ and $n\zeta l \ge L$.

In the Cayley graph $\G$ equation \eqref{eq:b-zeta} gives rise to a cycle $o=rqr'q'$, in which
$\Lab(r)\equiv B$, $q_-=r_+$, $\Lab(q)\equiv (V_1V_2\dots V_l)^{n\zeta}$,
$r'_-=q_+$, $\Lab(r')\equiv B^{-1}$, $q'_-=r'_+$,
$\Lab(q')\equiv \left( U_1 W_1 U_2 W_2\dots U_l W_l \right)^{-n\eta}$.

By construction, the paths $q$ and $q'$ have exactly $n \zeta l$ and $n\eta l$ components respectively.
Suppose that $\zeta >\eta$. By Lemma \ref{lem:regul}.(c), at least
$n\zeta l-6\e > n l(\zeta-1)\geq n l\eta$ components of
$q$ must be connected to components of $q'$, hence two distinct components of $q$ will have to be connected
to the same component of $q'$, contradicting Lemma \ref{lem:regul}.(c). Hence $\zeta \le \eta$. A symmetric argument
shows that $\eta \le \zeta$. Consequently $\zeta=\eta$.

Since $\ell(q)= n \zeta l \ge L$, we can apply Lemma
\ref{lem:conseq-reg} to find $2l$ consecutive components of $q$
that are connected to $2l$ consecutive components of $q'^{-1}$.
Therefore there are consecutive components $p_1,\dots ,p_{l+1}$ of
$q$ and $p'_1,\dots, p'_{l+1}$ of $q'^{-1}$ such that $p_j$ is
connected to $p'_j$ for each $j$, and $\Lab(p_i)\equiv V_i$ for
$i=1,\dots,l$, $\Lab(p_{l+1})\equiv V_1$
(see Figure \ref{pic:2}). Therefore $\Lab(p_i') \in E_G(g_i)$, $i=1,\dots,l$, $\Lab(p_{l+1}') \in E_G(g_1)$. From the
form of $\Lab(q'^{-1})$ it follows that there is $k \in \{0,1,\dots,l-1\}$ such that $\Lab(p_j')\equiv U_{j+k}$ for
$j=1,\dots,l+1$  (indices are added modulo $l$). Thus
$U_{j+k}= h_{j+k}^{n_{j+k}}\in E_G(g_{j})$. On the other hand,
$h_{j+k}^{n_{j+k}} \in E_G(g_{\sigma(j+k)})$ and it has infinite order by the assumptions, hence $g_{\sigma(j+k)}$ is commensurable with $g_j$ in $G$ by Remark \ref{rem:elem_inter}.
The latter yields
that $\sigma(j+k)=j$ for all $j$. Therefore $\sigma$ is a cyclic shift (by $l-k$) of $\{1,\dots,l\}$.

\begin{figure}[!ht]
  \begin{center}
   \input{pic2.pstex_t}
  \end{center}
  \caption{}\label{pic:2}
\end{figure}

To prove the last claim of the lemma, note that the subpath $w_i$ of $q'^{-1}$ between
$(p'_i)_+$ and $(p'_{i+1})_-$ is labelled by $W_{i+k} \equiv W_{\sigma^{-1}(i)}$. As we showed, the vertex $(p_i)_+=(p_{i+1})_-$
is connected to $(w_i)_-$ by a path $s_i$ with $\Lab(s_i) \in E_G(g_i)$, and  to $(w_i)_+$
by a path $t_i$ with $\Lab(t_i) \in E_G(g_{i+1})$, $i =1,\dots,l$ (here we use the convention that $g_{l+1}=g_1$). Considering the cycle
$t_i^{-1}s_i w_i$ we achieve the desired inclusion: $f_{\sigma^{-1}(i)}=\Lab(w_i) \in E_G(g_i) E_G(g_{i+1})$,
$i=1,\dots,l$.
\end{proof}

 Let $H \leqslant G$ be a non-elementary subgroup such that $H \cap \LWPD \neq\emptyset$.
The following three lemmas are analogues of \cite[Lemmas 4.5, 4.6, 4.7]{MO} respectively.
The proofs are exactly the same as in \cite{MO} once one uses Lemma \ref{lem:newtame1} instead of \cite[Lemma 4.4.(i)]{MO},
 Lemma~\ref{lem:cyclicperm} instead of \cite[Lemma 4.4.(ii)]{MO} and Lemma \ref{lem:elemrem}.(c) instead of \cite[Lemma 2.4.(b)]{MO}.

\begin{lem}\label{lem:prod_of_three}
Suppose that $\varphi \colon H \to G$ is a homomorphism such that
$\varphi(h) \apg h$ for all $h\in H \cap \LWPD$. Then for any $g_1,g_2,g_3\in H \cap \LWPD$,
satisfying $g_i \stackrel{G}{\not\approx} g_j$
for $i\neq j$,
there exists $N_5 \in \N$ such that for arbitrary $n_1,n_2,n_3 \in \Z$, with $|n_i| \ge N_5$, $i=1,2,3$, and for
$g =g_1^{n_1} g_2^{n_2} g_3^{n_3}$, one has $g \in \LWPD$ and
$(\varphi(g))^\zeta= e g^\zeta e^{-1}$, for some $e \in G$ and $\zeta \in \N$.
\end{lem}

\begin{lem}\label{lem:first_step}
 Let $a,b \in \LWPD$ be non-commensurable elements and let $y,z \in G$.
There exists $N_6\in \N$ such that the following holds. Suppose that
$a^{k'} y b^{l'} z\stackrel{G}{\approx} a^{k} b^{l}$ for some integers $k,l,k',l'$ with
$|k|,|l|,|k'|,|l'|\ge N_6$. Then  $y \in E_G(a) E_G(b)$ and $z \in E_G(b) E_G(a)$.
\end{lem}

\begin{lem} \label{lem:spec-image}  Assume that $g \in S_G(H,\cS)$ and $\psi\colon H \to G$ is a homomorphism satisfying
$\psi(g^n)=g^n z$ for some $n \in \N$ and $z \in E_G(H)$. Then there is $f \in E_G(H)$ such that $\psi(g)=gf$.
\end{lem}


\section{Commensurating homomorphisms} \label{sec:comm_hom}
This section is dedicated to proving our main technical theorem:

\begin{thm}\label{thm:comm}
Let $G$ be a group acting coboundedly by isometries on a hyperbolic space $\cS$.
Let $H\leqslant G$ be a non-elementary subgroup of $G$ and let $\phi\colon H\to G$ be a homomorphism.
Suppose that $H \cap \LWPD\neq \emptyset$ and  $\phi(h)\apg h$ for all $h\in H \cap \LWPD$.

Then there exists a set map $\e \colon H\to E_G(H)$, whose restriction to $C_H(E_G(H))$ is a homomorphism,
and an element $w\in G$ such that for every $h\in H$, $\phi(h)=w(h\e(h))w^{-1}$.
Moreover, if $\phi(H)=H$ then $w\in N_G(H  E_G(H))$.
\end{thm}

We need two auxiliary lemmas in order to prove the  theorem. As usual, $G$ is a group acting isometrically and coboundedly on a hyperbolic space $\cS$ and $H \leqslant G$ is a non-elementary subgroup
with $H \cap \LWPD \neq \emptyset$.

\begin{lem}\label{lem:towards_inner}
Let $\psi\colon H\to G$ be a homomorphism such that $\psi(h)\apg h$ for all $h\in H \cap \LWPD$.
Suppose that $g_1,g_2,g_3 \in H \cap \LWPD$ is a triple of pairwise non-commensurable (in $G$) elements with $g_1 \in S_G(H,\cS)$ and $g_2,g_3 \in C_H(E_G(H))$.
Then for any $l,m \in \N$ there are $n_1,n_2,n_3,n_1'\in \N$ such that $a:=g_1^{ln_1}g_2^{mn_2}g_3^{n_3}$ and $b:=g_1^{n_1'}g_2^{mn_2}g_3^{n_3}$ satisfy the following properties:
\begin{itemize}
  \item $a,b \in S_G(H,\cS)$;
  \item the elements $a,b,g_1,g_2,g_3$ are pairwise non-commensurable in $G$;
  \item there exist $\mu,\nu\in \N$, $u,v\in G$ such that $\psi(a^\mu)=u a^\mu u^{-1}$ and $\psi(b^\nu)=v b^{\nu} v^{-1}$.
\end{itemize}
\end{lem}

\begin{proof}
Let $N_1\in \N$ be given by Lemma \ref{lem:newtame1} applied to the set $\{g_1,g_2,g_3\}$ and $F=\emptyset$.
Choose $N_5\in \N$ according to an application of Lemma \ref{lem:prod_of_three} to $\psi, g_1,g_2,g_3$ and let $n_3:=\max\{N_1,N_5\}$.
By Lemma \ref{lem:newtamenc}, there is $n_2\geq \max\{N_1,N_5\}$ such that $g_2^{mn_2}g_3^{n_3}\in H \cap \LWPD$ and this element is not commensurable with $g_1$ in $G$.
It follows that the element $g_2^{mn_2}g_3^{n_3}\in C_H(E_G(H))$ has infinite order, and thus it cannot belong to the virtually cyclic subgroup  $E_G(g_1)$.
Since $g_1$ is $H$-special, we can use  Lemma \ref{lem:SG} to find $N_3\in \mathbb{N}$ such that $g_1^ng_2^{mn_2}g_3^{n_3}\in S_G(H,\cS)$ whenever $n\geq N_3$.
Take $n_1 \in \N$ so that $ln_1 \ge\max\{N_1,N_3, N_5\}$, and apply Lemma \ref{lem:newtamenc} to find $n_1'>ln_1$ such that the elements
$a=g_1^{ln_1}g_2^{mn_2}g_3^{n_3}$ and $b=g_1^{n_1'}g_2^{mn_2}g_3^{n_3}$ are non-commensurable in $G$.

By Lemma  \ref{lem:SG} we have $a,b \in S_G(H,\cS)$, and by Lemma \ref{lem:newtame1} neither of these two elements is commensurable to any $g_i$, $i=1,2,3$.
Finally, using Lemma \ref{lem:prod_of_three}, one can conclude that there exist $u,v\in G$, $\mu, \nu\in \N$ such that $\psi(a^\nu)=ua^\nu u^{-1}$ and $\psi(b^\nu)=vb^\nu v^{-1}$.
\end{proof}

\begin{lem}\label{lem:from2toall}
Let $\psi\colon H\to G$ be a homomorphism such that $\psi(h)\apg h$ for all $h\in H \cap \LWPD$.
Suppose that there are two non-commensurable elements $a,b\in S_G(H,\cS)$ such that $\psi(a^\mu)=a^\mu$ and $\psi(b^\nu)=b^\nu$ for some $\mu,\nu\in \N$.
Then for every $g\in S_G(H,\cS)$ there is $f=f(g)\in E_G(H)$ such that $\psi(g)=gf$.
\end{lem}
\begin{proof}
Consider any $g\in S_G(H,\cS)$. If $g\in E_G(a)$ then there is $n\in \N$ such that $g^n\in \gen{a^\mu}$ because $|E_G(a):\gen{a^\mu}|<\infty$. Hence $\psi(g^n)=g^n$ and
then by Lemma \ref{lem:spec-image}, $\psi(g)=gf$ for some $f\in E_G(H)$.

Suppose now that $g\not\in E_G(a)$. Recall that $g\in C_H(E_G(H))$ because this element is $H$-special.
Now, combining Lemmas  \ref{lem:SG} and \ref{lem:newtamenc},
we can find some $l\in \N$ such that $d:=a^{l\mu}g\in S_G(H,\cS)$ and $d$ is not commensurable with $a$ and $b$ in $G$.

By Lemma \ref{lem:towards_inner}, we can find $n_1,n_2,n_3\in \N$ such that $c:=a^{n_1\mu}b^{n_2\nu}d^{n_3}\in S_G(H,\cS)$,
$c\napg a$, $c\napg b$ and $\psi(c^\eta)=ec^{\eta}e^{-1}$ for some $\eta\in \N$ and $e\in G$.

By Lemma \ref{lem:newtame1}, $a^{k\mu}c^{k\eta}\in H \cap \LWPD$ for every sufficiently large $k\in \N$.
Hence $a^{k\mu}ec^{k\eta}e^{-1}=\psi(a^{k\mu}c^{k\eta})\apg a^{k\mu}c^{k\eta}$ whenever $k$ is sufficiently large.
So, Lemma \ref{lem:first_step} shows that $e\in E_G(a)E_G(c)$. Thus $e=a^pc^sf$ for some $p,s\in \Z$ and $f\in E_G(H)$. This implies that $\psi(c^{\eta})=a^pc^{\eta} a^{-p}$ since $c\in C_H(E_G(H))$.

Similarly one proves that  $e\in E_G(b)E_G(c)$, and thus there is $q\in \Z$ such that $\psi(c^\eta)=b^qc^{\eta}b^{-q}$. Hence $(a^{-p}b^q) c^{\eta}(a^{-p}b^q)^{-1}=c^{\eta}$ and therefore $a^{-p}b^q\in E_G( c )$
by Lemma \ref{lem:elemrem}.(b).

Assume, first, that $p\neq 0$.
If  $a^{-p}b^q\in E_G(c)$ has finite order, then $a^{-p}b^q\in E_G(H)$ because $c\in S_G(H,\cS)$. Hence, $a^{p}\in b^q E_G(H)\subset E_G(b)$ contradicting the assumption that $a\napg b$.
Thus $a^{-p}b^q$ must have infinite order, and so there are $\alpha,\beta\in \Z\setminus\{0\}$  such that $(a^{-p}b^q)^\alpha=c^\beta$.

Since  $\psi(a^\mu)=a^\mu$ and $\psi(b^\nu)=b^\nu$, by Lemma \ref{lem:spec-image} there exist $f_1,f_2\in E_G(H)$ such that $\psi(a)=af_1$ and $\psi(b)=bf_2$.
Since $a,b\in C_H(E_G(H))$ we obtain that
$$\psi(c^\beta)=\psi((a^{-p}b^q)^\alpha)=(a^{-p}b^q)^\alpha f_3= c^\beta f_3\text{  for some } f_3\in E_G(H). $$
Then for $\gamma:=\beta \eta |E_G(H)|$ we get that $c^\gamma=\psi(c^\gamma)=a^p c^{\gamma} a^{-p}$, implying that $a^p\in E_G(c)$, which contradicts $a\napg c$.

Therefore, $p=0$ and, thus, $\psi(c^\eta)=c^\eta$.
By Lemma \ref{lem:spec-image}, there exists $f_4\in E_G(H)$ such that $\psi(c)=cf_4$.
Since $c=a^{n_1\mu}b^{n_2\nu}d^{n_3}$ and $a^\mu$, $b^\nu$ are fixed by $\psi$, we see that $\psi(d^{n_3})=d^{n_3}f_4$.
Applying Lemma \ref{lem:spec-image} again, we find $f_5\in E_G(H)$ such that $\psi(d)=df_5$. Finally, since $d=a^{l\mu}g$, we achieve that $\psi(g)=gf_5$ as needed.
\end{proof}

We are now ready to prove the main result of this section.

\begin{proof}[Proof of Theorem \ref{thm:comm}]
Since $H \cap \LWPD\neq \emptyset$,  by Lemma \ref{lem:special} there is at least one element $g _1 \in S_G(H,\cS)$.
Since $H$ is non-elementary and $C_H(E_G(H))$ has finite index in it, $C_G(E_G(H))$ is non-elementary itself.
On the other hand, $E_G(g_1)$ is elementary by  Lemma \ref{lem:elemrem}, hence there exists $y\in C_H(E_G(H))\setminus E_G(g_1)$.

By Lemma \ref{lem:newtamenc}, there is $k_2\in \N$ such that $g_2:=g_1^{k_2}y\in \LWPD$ and $g_2\napg g_1$.
Using the same lemma again, we can find $k_3\in \N$ such that $g_3:=g_1^{k_3}y\in \LWPD$ and $g_3\napg g_i$ for $i=1,2$. 

Note that, by construction, $g_2,g_3 \in C_H(E_G(H))$, so one can
use Lemma \ref{lem:towards_inner} to find non-commensurable elements $a,b\in S_G(H,\cS)$ such that $\phi(a^\mu)=ua^\mu u^{-1}$ and $\phi(b^\nu)=vb^\nu v^{-1}$ for some
$u,v\in G$ and $\mu, \nu\in \N$.

Let $\chi\colon H\to G$ be the homomorphism defined by $\chi(h)=u^{-1}\phi(h)u$ for all $h\in H$. Then $\chi(a^\mu)=a^\mu$, $\chi(b^{\nu})=(u^{-1}v)b^{\nu}(u^{-1}v)^{-1}.$
Note that $\chi(h)\apg h$ for every $h\in H \cap \LWPD$. By Lemma \ref{lem:newtame1}, $(a^\mu)^k(b^\nu)^k\in H \cap \LWPD$ if $k\in \N$ is large enough. Therefore
$$a^{k\mu}(u^{-1}v)b^{k\nu}(u^{-1}v)^{-1}= \chi (a^{k\mu}b^{k\nu})\apg a^{k\mu}b^{k\nu} \text{  for every sufficiently large }k\in\N. $$
Consequently, by Lemma \ref{lem:first_step}, $u^{-1}v\in E_G(a)E_G(b)$, thus $u^{-1}v=a^sb^tf$ for some $s,t\in \Z$ and $f\in E_G(H)$. Hence $\chi(b^\nu)=(a^sb^tf) b^\nu (a^sb^tf)^{-1}$, and since $b\in C_H(E_G(H))$, $\chi(b^\nu)=a^s b^\nu a^{-s}$. Let $w:=ua^s\in G$ and let the homomorphism $\psi\colon H\to G$ be defined by $\psi(h)=w^{-1}\phi(h)w=a^{-s}\chi(h)a^s$ for all $h\in H$. By construction
\begin{equation*}
\psi(a^\mu)=a^\mu,\, \phi(b^\nu)=b^\nu \text{ and } \psi(h)\apg h \text{ for each } h\in H \cap \LWPD .
\end{equation*}
Now we are under the hypothesis of Lemma \ref{lem:from2toall}, claiming that for every $g\in S_G(H,\cS)$ there exists $f=f(g)\in E_G(H)$ such that $\psi(g)=gf$.

By Proposition \ref{prop:special}, $C_H(E_G(H))$ is generated by $S_G(H,\cS)$, therefore for each $x\in C_H(E_G(H))$ there is $\tilde{\e}(x)\in E_G(H)$ such that $\psi(x)=x\tilde{\e}(x).$ Since the map $\psi$ is a homomorphism,  the map $\tilde{\e}\colon C_H(E_G(H))\to E_G(H)$ will also be a homomorphism. By construction, we have $\phi(x)=w\psi (x) w^{-1}=wx \tilde{\e}(x)w^{-1}$ for all $x \in C_H(E_G(H))$.

Now we need to extend the homomorphism $\tilde{\e}\colon C_H(E_G(H))\to E_G(H)$ to a set map $\e:H \to E_G(H)$. Define
$l:=|H:C_H(E_G(H))|$, $m:=|E_G(H)|$  and $n:=ml\in \N$.

Since $E_G(H)$ is normalized by $H$, the centralizer $C_H(E_G(H))$ is a normal subgroup of $H$. Consequently, for any $z\in H$ we have that $z^l\in C_H(E_G(H))$ and
\begin{equation}\label{eq:fix_zn}
\psi(z^n)=z^n\tilde{\e}(z^l)^m=z^n.
\end{equation}
Take  an arbitrary $h\in H$. For any $g\in H \cap \LWPD$ we have 
$\psi(h)g^n\psi(h)^{-1}=\psi(hg^nh^{-1})=hg^nh^{-1}$, implying that $h^{-1}\psi(h)\in E_G(g)$. Since $g$
was an arbitrary element of $H \cap \LWPD$, we conclude that $h^{-1}\psi(h)\in E_G(H)=\cap_{g\in H \cap \LWPD}E_G(g)$ (see Lemma \ref{lem:E_G(H)}).

After defining the $\e(h):=h^{-1}\psi(h)$ for each $h\in H$, one immediately sees that $\e\colon H \to E_G(H)$ is a map with the required properties.
Evidently the restriction of $\e$ to $C_H(E_G(H))$ is the homomorphism $\tilde{\e}$.

It remains to prove the last claim of the theorem. Assume that $\phi(H)=H$.
Consider any element $f\in E_G(H)$. By the above assumption, for any $g\in H \cap \LWPD$ there is $h\in H$ such that $\phi(h)=g$.
Recalling \eqref{eq:fix_zn} and the definition of $\psi$ we achieve $g^n=\phi(h^n)=w \psi(h^n) w^{-1}=wh^nw^{-1}$.
But $h^n \in C_H(E_G(H))$, therefore
$$wfw^{-1}g^n(wfw^{-1})^{-1}=wfh^nf^{-1}w^{-1}=wh^nw^{-1}=g^n.$$
Hence, $wfw^{-1}\in E_G(g)$ for every $g\in H \cap \LWPD$; consequently $wfw^{-1}\in E_G(H)$. The latter implies that $wE_G(H)w^{-1}\subseteq E_G(H)$ and since $E_G(H)$ is finite, we conclude that $w$ normalizes
$E_G(H)$.

Observe that $\widehat{H}:=HE_G(H)$ is a subgroup of $G$ because $E_G(H)$ is normalized by $H$ (see Lemma \ref{lem:E_G(H)}). For any $h\in H$ we have that $whw^{-1}=wh\e(h)w^{-1}w\e(h)^{-1}w^{-1}\in HE_G(H)$;
thus $wHw^{-1}\leqslant \widehat{H}$. Since $w^{-1}\phi(h)w=h\e(h)\in \widehat{H}$ and $\phi(H)=H$, one gets $w^{-1}Hw\subseteq \widehat{H}$. Therefore $w\widehat{H}w^{-1}\subseteq \widehat{H}wE_G(H)w^{-1}=\widehat{H},$ $w^{-1}\widehat{H}w\subseteq \widehat{H}w^{-1}E_G(H)w=\widehat{H}$, i.e., $w\in N_G(\widehat{H})$. This finishes the proof of the theorem.
\end{proof}

 Theorem \ref{thm:comm} allows us to  generalize  Corollaries 5.3 and 5.4 from \cite{MO}.
\begin{cor}\label{cor:comm_equiv_cond}
Let $G$ be group acting coboundedly and by isometries on a hyperbolic space $\cS$. Suppose that $H\leqslant G$ is a non-elementary subgroup,
with $H \cap \LWPD\neq \emptyset$, and $\phi\colon H\to G$ is a homomorphism.
The following are equivalent:
\begin{enumerate}
\item[{\rm (a)}] $\phi$ is commensurating;
\item[{\rm (b)}] $\phi(g)\apg g$ for every $g\in H \cap \LWPD$;
\item[{\rm (c)}] there is a set map $\e \colon H\to E_G(H)$, whose restriction to $C_H(E_G(H))$ is a homomorphism, and an element $w\in G$ such that for every $g\in G$, $\phi(g)=w(g\e(g))w^{-1}$.
\end{enumerate}
In particular, if  $E_G(H)=\{1\}$ then every commensurating homomorphism from $H$ to $G$ is the restriction to $H$ of an inner automorphism of $G$.
\end{cor}
\begin{proof}
(a) implies (b) by definition, and (b) implies (c) by Theorem \ref{thm:comm}.
It remains to show that (c) implies (a). Indeed, let the homomorphism $\phi $ satisfy (c), and let $g$ be an arbitrary element of $H$.
Thus
$\phi (g)=w(g\e(g))w^{-1}$ for some $w\in G$ and $\e(g)\in E_G(H)$.

Since $E_G(H)$ is a finite subgroup of $G$ normalized by $H$, the subgroup $C_H(E_G(H))$ is normal and of finite index in $H$. Set $m:=|E_G(H)| \in \N$ and $l:=|H:C_H(E_G(H))|\in \N$.
It follows that $g^l\in C_H(E_G(H))$ and $\e(g^{lm})={\e(g^l)}^m=1$ in $G$ by the assumptions of (c).
Therefore
$$\phi(g)^{lm}=\phi(g^{lm})=w g^{lm}\e(g^{lm})w^{-1}=w g^{lm} w^{-1}.$$
Hence $\phi(g) \apg g$ for all $g \in H$, as required.
\end{proof}

An application of the above corollary to the case $H=G$ gives rise to the following characterization of commensurating endomorphisms, which is
very similar to the result for relatively hyperbolic groups from \cite[Cor. 1.4]{MO}:

\begin{thm}\label{thm:comm_end} Let $G$ be an acylindrically hyperbolic group.
An endomorphism $\phi\colon G\to G$ is commensurating if and only if there is a set map $\e \colon G\to E_G(G)$, whose restriction to $C_G(E_G(G))$ is a homomorphism, and an element $w\in G$
such that $\phi(g)=w(g\e(g))w^{-1}$ for every $g\in G$. In particular, if $E_G(G)=\{1\}$ then every commensurating endomorphism is an inner automorphism of $G$.
\end{thm}

\begin{proof} 
By Theorem \ref{thm:acyl-equiv_def}, since $G$ is acylindrically hyperbolic,
it is non-elementary and admits a cobounded action on a hyperbolic space $\cS$ such that $\LWPD \neq \emptyset$.
Now the claim follows from Corollary \ref{cor:comm_equiv_cond} applied to the case when $H=G$.
\end{proof}

\begin{rem}\label{rem:f_many}
If $G$ is a finitely generated acylindrically hyperbolic group then Theorem \ref{thm:comm_end} easily implies that $Inn(G)$ has finite index in the group $Aut_{com}(G)$ of all commensurating
automorphisms of $G$. On the other hand, it is not difficult to show that this is not true for $F_\infty \times \Z_2$, the direct product of the free group of countably infinite rank and the cyclic group of order $2$ (in fact this group
has uncountably many commensurating automorphisms).
\end{rem}

The above remark shows that to establish Corollary  \ref{cor:conj_aut} we need to work a bit more since the group $G$ may not be finitely generated
(however, the proof is very similar to that of \cite[Cor. 5.4]{MO}).

\begin{proof}[Proof of Corollary \ref{cor:conj_aut}]
Again, by Theorem \ref{thm:acyl-equiv_def}, $G$ is non-elementary and admits a cobounded action on a hyperbolic space $\cS$ such that $\LWPD \neq \emptyset$.
Applying Corollary \ref{cor:comm_equiv_cond} to the case when $H=G$, we see that for any automorphism $\varphi \in Aut_{pi}(G)$,
there exist $w \in G$ and a map $\e \colon G \to E_G(G)$ such that $\varphi(h)=wg \e(g) w^{-1}$ for each $g \in G$.
Take any element $h \in S_G(G,\cS)$. 
Then $h$ commutes with $\e(h) \in E_G(G)$, and, consequently,
$(\varphi(h))^m=wh^mw^{-1}$ where $m:=|E_G(G)|\in \N$.

Now, since $\varphi$ is a pointwise inner automorphism of $G$, there
is $x \in G$ such that $\varphi(h)=xhx^{-1}$. Hence $x h^m x^{-1}=\varphi(h^m)=wh^mw^{-1}$, i.e.,
$w^{-1}x \in E_G(h)\cong \langle h \rangle \times E_G(G)$, hence $w^{-1}x \in C_G(h)$.
Consequently, we have $h=w^{-1}x h \left(w^{-1}x\right)^{-1}=h \e(h)$, which implies that $\e(h)=1$.
Since the latter holds for any $h \in S_G(G,\cS)$, it follows from Proposition \ref{prop:special} that $\e(C_G)=\{1\}$, where $C_G:=C_G(E_G(G))$.

Note that $|G:C_G|<\infty$, hence there are $g_1,\dots,g_l \in G$ such that $G=\bigsqcup_{i=1}^l C_G g_i$.
For any $g \in G$ there are $a \in C_G$ and $i \in \{1,\dots, l\}$ such that $g=a g_i$, and one has
\begin{multline*}\varphi(a) \varphi(g_i) = \varphi(g)=w g \e(g) w^{-1}=(waw^{-1}) (w g_i \e(a g_i) w^{-1})\\ =  \varphi(a)
(\varphi(g_i) w (\e(g_i))^{-1}\e(a g_i) w^{-1}). \end{multline*}
Therefore $\e(g)=\e(a g_i)=\e(g_i)$, i.e.,
the map $\e$ is uniquely determined by the images of $g_1,\dots,g_l$.
Since $\varphi(g)=w (g \e(g_i))w^{-1}$, the automorphism $\varphi \in Aut_{pi}(G)$, up to composition with an inner automorphism of $G$,
is completely determined by the finite collection of elements $\e(g_1),\dots,\e(g_l) \in E_G(G)$, and since
$E_G(G)$ is finite, we can conclude that $|Aut_{pi}(G):Inn(G)|<\infty$.

Finally, if $E_G(G)=\{1\}$ we have $\varphi(g)=wgw^{-1}$ for all $g \in G$, that is $\varphi \in Inn(G)$.
\end{proof}

Combining Grossman's criterion with Corollary \ref{cor:conj_aut}, we obtain the following
\begin{cor}
Let $G$ be a finitely generated acylindrically hyperbolic group. If $G$ is conjugacy separable and contains no non-trivial finite normal subgroups then $Out(G)$  is residually finite.
\end{cor}

In \cite[Cor. 1.6]{Cap-Min} Caprace and the second author showed that any pointwise inner automorphism of a finitely generated Coxeter group $W$ is inner.
Theorem \ref{thm:comm_end} can be used to say much more in the case when $W$ is acylindrically hyperbolic.

\begin{lem}\label{lem:Cox-ahyp} Suppose that $W$ is a finitely generated infinite irreducible non-affine Coxeter group. Then $W$ is acylindrically hyperbolic and $E_W(W)=\{1\}$.
\end{lem}

\begin{proof} The assumptions imply that $W$ is not virtually cyclic and $W$ contains a rank $1$ isometry for the natural action on the associated Davis CAT($0$) complex --  see \cite{Cap-Fuj}.
Therefore $W$ is acylindrically hyperbolic by \cite{Sisto-contr}.

It remains to note that $E_W(W)=\{1\}$ because any finite normal subgroup of a Coxeter group is contained in a finite normal
parabolic subgroup, but an infinite irreducible Coxeter group cannot have any proper normal parabolic subgroups (the normalizer of a parabolic subgroup $P \leqslant W$ is itself a parabolic subgroup, which is
isomorphic to the direct product $P \times R$, where $R$ is the orthogonal complement of $P$ in $W$ -- see \cite{Deod,Kram}).
\end{proof}

A combination  of Lemma \ref{lem:Cox-ahyp} with Theorem \ref{thm:comm_end} immediately yields the following:

\begin{cor}\label{cor:Cox} If $W$ is a finitely generated infinite irreducible non-affine Coxeter group then every commensurating endomorphism of $W$ is an inner automorphism.
\end{cor}

\section{Normal endomorphisms of acylindrically hyperbolic groups}
This section is dedicated to  proving Theorem \ref{thm:norm_end}. Our argument uses the powerful machinery of algebraic Dehn fillings,
developed for hyperbolically embedded subgroups by Dahmani, Guirardel and Osin \cite{DGO}:
\begin{thm}[{\cite[Thm. 7.19]{DGO}}] \label{thm:Dehn_fil} Let $G$ be a group, $X$ a subset of $G$, $\Hl$ a collection of subgroups of $G$.
Suppose that $\Hl \h (G,X)$. Then  there exists a family of
finite subsets $F_\lambda \subseteq H_\lambda\setminus\{1\}$,  $\lambda\in \Lambda$, such that for every collection of normal subgroups $\mathfrak{N} = \{N_\lambda\lhd H_\lambda \mid \lambda \in \Lambda\}$, satisfying
$N_\lambda \cap F_\lambda=\emptyset$ for all $\lambda \in \Lambda$, the following hold:

\begin{itemize}
  \item[(a)] $N \cap H_\lambda=N_\lambda$ for all $\lambda \in \Lambda$, where $N:= \llangle N_\lambda \mid \lambda \in \Lambda \rrangle^G \lhd G$;
  \item[(b)] every element of $N$  is either conjugate to an element of $\bigcup_{\lambda \in \Lambda} N_\lambda\subseteq G$ or is loxodromic with respect to the action of $G$ on $\ga(G,X \sqcup \mathcal{H})$,
  where $\mathcal{H}:=\bigsqcup_{\lambda \in \Lambda} \left( H_\lambda \setminus\{1\}\right) $;
  \item[(c)] $N$ is isomorphic to the free product of copies of groups from $\mathfrak{N}$.
\end{itemize}
\end{thm}

Combining the above result with Corollary \ref{cor:wpd-he} one obtains the following statement:

\begin{lem}\label{lem:norm_sbgp_in_acyl_hyp} Assume that $G$ is a group acting isometrically and coboundedly on a hyperbolic space $\cS$.
For any element $g \in \LWPD$  there exists $M \in \N$ such that if $|m| \ge M$ and $\gen{g^m} \lhd E_G(g)$ then
the normal closure $\llangle g^m \rrangle^G\lhd G$ is free and every non-trivial element in it is loxodromic (with respect to the action of $G$ on $\cS$).
\end{lem}

\begin{proof} Let $X$ be a symmetric generating set of $G$ given by Lemma \ref{lem:svarc-milnor}.
By Corollary \ref{cor:wpd-he}, $E_G(g) \h (G,X)$, therefore we can apply Theorem \ref{thm:Dehn_fil}, which claims that
there exists a finite subset $F \subseteq E_G(g)\setminus\{1\}$ such that for every normal subgroup $N_0 \lhd E_G(g)$, with $N_0 \cap F =\emptyset$, the normal closure $N:=\llangle N_0 \rrangle^G$ is isomorphic to
the free product of some copies  of $N_0$, and every element of $N$ is either conjugate to an element of $N_0$ in $G$ or is loxodromic with respect to the action of $G$ on the
Cayley graph $\ga(G, X \sqcup E_G(g)\setminus\{1\})$.

Since the order of $g$ is infinite, there is $M \in \N$ such that $\gen{g^m} \cap F=\emptyset$ whenever $|m| \ge M$. So, if $m$ satisfies this inequality and  $\gen{g^m} \lhd E_G(g)$, by the
previous paragraph we see that $\llangle g^m \rrangle^G\lhd G$ is isomorphic to the free product of infinite cyclic groups (hence, it is free) and every element $h \in \llangle g^m \rrangle^G \setminus \{1\}$
is either conjugate to some non-zero power of $g$ in $G$ or is loxodromic with respect to the action of $G$ on  $\ga(G, X \sqcup E_G(g)\setminus\{1\})$. Therefore such $h$ is loxodromic with respect to the
action of $G$ on $\ga(G,X)$: in the former case this is true because $g \in \LWPD$ and in the latter case this is demonstrated in the first paragraph of the proof of Lemma \ref{lem:tameinS}
(one can take $X_1:=X \cup E_G(g)\setminus\{1\}$).
It follows that $h$ is loxodromic with respect to the action of $G$ on $\cS$.
\end{proof}

\begin{proof}[Proof of Theorem \ref{thm:norm_end}] The argument will be split in two cases.

\emph{Case 1:} $E_G(G)=\{1\}$. We need to show that either $\phi(G)=\{1\}$ or $\phi$ is an inner automorphism of $G$.
Arguing by contradiction suppose that $\phi(G) \neq \{1\}$ and $\phi \notin Inn(G)$. Let us first prove the following claim:
\begin{equation}\label{eq:claim} \text{ there is some $g_1 \in \LWPD$ such that $\phi(g_1) \in \LWPD$ and $\phi(g_1) \napg g_1$}.
\end{equation}
Since $G$ is acylindrically hyperbolic, it has a symmetric generating set $X$  such that $\cS\coloneq \ga(G,X)$ is hyperbolic, $|\partial\cS|>2$, and $G$ acts on $\cS$ acylindrically.
Then $G$ is non-elementary and $\LWPD\neq \emptyset$ (as explained in Theorem \ref{thm:acyl-equiv_def} and in the paragraph after it), hence
$G$ is generated by the $G$-special elements (by Proposition \ref{prop:special}).

Therefore there must exist $g \in S_G(G,\cS)$ such that $\phi(g) \neq 1$. Choose $M \in \N$ according to Lemma \ref{lem:norm_sbgp_in_acyl_hyp}. Then for any
$m \ge M$, $\gen{g^m} \lhd E_G(g)=\gen{g}$ and every non-trivial element of $N:=\llangle g^m \rrangle^G$ is loxodromic. It follows that $N \setminus \{1\} \subseteq \LWPD$ by Remark \ref{rem:acyl->tame}
and the fact that  the action of $G$ on $\cS$ is acylindrical. Since $\phi(g) \neq 1$, there exists $m \ge M$ such that $\phi(g^m) \neq 1$. On the other hand, $\phi(N) \subseteq N$ as $\phi$ is a normal endomorphism, hence
we can conclude that $\phi(g^m) \in \LWPD$. Consequently, $\phi(g) \in \LWPD$ by Remark \ref{rem:powers}.

If $\phi(g) \napg g$, then claim \eqref{eq:claim} is true for $g_1=g$. So, suppose that $\phi(g) \apg g$. Since $E_G(G)=\{1\}$ and $\phi \notin Inn(G)$, $\phi$ is not commensurating by Theorem \ref{thm:comm_end}.
Hence, according to Corollary~\ref{cor:comm_equiv_cond}, there exists $h \in \LWPD$  such that $\phi(h)\napg h$. Recall that $E_G(h)$ is virtually cyclic, hence there is $L \in \N$ such that $\gen{h^l} \lhd E_G(h)$
whenever $l$ is divisible by $L$. Therefore, we can apply Lemma \ref{lem:norm_sbgp_in_acyl_hyp} as before to find $l \in \N$ such that $\llangle h^l \rrangle^G \setminus \{1\} \subseteq \LWPD$. Again, since $\phi$ is normal,
it must map this normal closure into itself. So, if $\phi(h^l) \neq 1$ then $\phi(h^l) \in \LWPD$,  consequently $\phi(h) \in \LWPD$ and $g_1=h$ satisfies claim \eqref{eq:claim}.

Thus it remains to consider the case when $\phi(h^l)=1$. Then $h^l \notin E_G(g)=\gen{g}$, and
by Lemma~\ref{lem:newtamenc}, there exists $n \in \N$ such that the element $g_1\coloneq g^n h^l$ belongs to $\LWPD$ and is not commensurable with $g$ in $G$.
But $\phi(g_1)=\phi(g^n) \apg g$ by the assumption above, therefore $\phi(g_1) \in \LWPD$ (by Remarks \ref{rem:powers} and \ref{rem:tame-conj}) and $\phi(g_1) \napg g_1$ (as $g \napg g_1$).
Thus we have shown the validity of claim \eqref{eq:claim}.

So, let $g_1 \in \LWPD$ be as in claim \eqref{eq:claim}. Then, according to Corollary \ref{cor:wpd-he}, the family $\{E_G(g_1),E_G(\phi(g_1))\}$ is hyperbolically embedded in $G$.
Now we can use the theory of algebraic Dehn fillings: let $F_1 \subset E_G(g_1)\setminus \{1\}$ and $F_2 \subset E_G(\phi(g_1))\setminus\{1\}$ be the finite subsets given
by Theorem \ref{thm:Dehn_fil}.
Evidently, there is $n \in \N$ such that $\gen{g_1^n}\cap F_1=\emptyset$ and $\gen{g_1^n}\lhd E_G(g_1)$. Then we can take $N_1\coloneq \gen{g_1^n} \lhd E_G(g_1)$ and $N_2\coloneq \{1\} \lhd E_G(\phi(g_1))$.
Since $N_i \cap F_i=\emptyset$, Theorem \ref{thm:Dehn_fil} claims that for $N\coloneq \llangle N_1,N_2 \rrangle^G=\llangle g_1^n \rrangle^G \lhd G$ one has
$$N \cap E_G(g_1)=N_1=\gen{g_1^n} \text{ and } N\cap E_G(\phi(g_1))=N_2=\{1\}.$$
Thus the image of $g_1$ in $G/N$ has finite order $n\in \N$ and the image of $\phi(g_1)$ has infinite order in $G/N$. On the other hand, since $\phi\colon G \to G$ is a normal endomorphism, $\phi(N) \subseteq N$, hence
it naturally induces an endomorphism $\overline{\phi} \colon G/N \to G/N$, defined by the formula $\overline{\phi}(fN) \coloneq\phi(f)N$ for all $f \in G$.
This yields a contradiction, as the order of $\overline{\phi}(g_1N)$ does not divide the order of $g_1N$
in $G/N$. Therefore, the proof under the assumption of Case 1 is complete.

\emph{Case 2:} $E_G(G)\neq \{1\}$. In this case $\overline{G}\coloneq G/E_G(G)$ is also acylindrically hyperbolic and
$E_{\overline{G}}(\overline{G})=\{1\}$ (see \cite[Lemma 5.10]{Hull}). Since $\phi:G \to G$ is normal, it naturally induces an endomorphism
$\overline{\phi}:\overline{G} \to \overline{G}$. Clearly $\overline{\phi}$ will be
a normal endomorphism of $\overline{G}$. Therefore we can apply Case 1 to $\overline{G}$ and $\overline{\phi}$, concluding that either $\overline{\phi}(\overline{G})=\{1\}$ or there exists
an element $\overline{w} \in \overline{G}$ such that $\overline{\phi}(\overline{f})=\overline{w}\overline{f}\overline{w}^{-1}$ for all $\overline{f}\in \overline{G}$.

If  $\overline{\phi}(\overline{G})=\{1\}$ then $\phi(G) \subseteq E_G(G)$, as required. In the remaining case, pick some preimage $w \in G$ of $\overline{w}\in \overline{G}$. Then for every $f \in G$ there exists
$\e(f) \in E_G(G)$ such that $\phi(f)=w f \e(f) w^{-1}$. Clearly, since $\phi$ is an endomorphism, the restriction of $\e$ to $C_G(E_G(G))$ is a homomorphism from $G$ to $E_G(G)$, hence,
by Corollary \ref{cor:comm_equiv_cond}, $\phi$ is commensurating.
\end{proof}

\begin{rem} Now that we have proved Theorem \ref{thm:norm_end}, one can show that if $G$ is acylindrically hyperbolic then
$Inn(G)$ has finite index in the group of all normal automorphisms $Aut_n(G)\leqslant Aut(G)$. If $G$ is finitely generated, then this is a consequence of Remark \ref{rem:f_many}.
If $G$ is not finitely generated, then one can use a more involved argument similar to the one from \cite[Thm. 6.4 and Cor. 6.5]{MO}.
\end{rem}

\begin{rem} If the finite radical of an acylindrically hyperbolic group $G$ is non-trivial, then it may possess non-commensurating normal automorphisms with non-trivial finite images.
Indeed, let $F$ be the free group of rank $2$ and let $Q$ be a non-abelian finite simple group.
Let $G:=F \times Q$ be the direct product of $F$ and $Q$, so that $G$ is hyperbolic and $E_G(G)=Q$.
Then $G$ has a natural endomorphism $\phi\colon G \to G$, which is the projection onto $Q$. It is not difficult to check that every normal subgroup $N \lhd G$
either contains $Q$ or is contained in $\ker(\phi)=F$. It follows that $\phi$ is a normal endomorphism of $G$ with $\phi(G)=Q$.
\end{rem}

\section{Commensurating endomorphisms of subgroups of right angled Artin groups}\label{sec:RAAG-end}
The purpose of this section is to prove Theorem \ref{thm:comm_end-raag} from the Introduction.

Let $\ga=(V,E)$ be a simplicial graph with the vertex set $V\ga=V$ and the edge set $E\ga=E$.
The associated  \textit{right angled Artin group} $A=A(\ga)$ is the group given by the presentation $$\prs{V}{ [u,v]=1, \, \forall \{u,v\}\in E }.$$
The cardinality $|V|$ is said to be the \emph{rank} of $A$. Algebraically, the rank of $A$ is exactly the smallest cardinality of  a generating set of $A$ (this can be justified by looking at the abelianization of
$A$, which is isomorphic to $\Z^{|V|}$).

Right angled Artin groups are special cases of \emph{graph products of groups}, when all the vertex groups are infinite cyclic (see \cite[Subsection 2.2]{AntolinMinasyan}
for some background on graph products).

\begin{lem}\label{lem:norm-centr} Suppose that $A$ is a right angled Artin group and $H \leqslant A$ is any subgroup.
\begin{itemize}
  \item[(i)] If $N \lhd H$ is a normal subgroup which does not contain non-abelian free subgroups, then $N$ is central in $H$.
  \item[(ii)] The quotient of $H$ by its center $Z=Z(H)$ is centerless.
\end{itemize}
\end{lem}

\begin{proof} To prove (i), suppose that $N$ is not central in $H$. Then there exist $h \in H\setminus \{1\}$ and $g \in N\setminus \{1\}$ such that $hg \neq gh$.
By a theorem of Baudisch \cite{Baud} (see also \cite[Cor. 1.6]{AntolinMinasyan}), the latter implies that $h$ and $g$ generate
a free subgroup $F$, of rank $2$, in $A$. Since  $g \in F \cap N$, this intersection is a non-trivial normal subgroup of $F$, hence it is a non-abelian free group.
This contradicts the assumption that $N$ has no non-abelian free subgroups. Therefore $N$ must be central in $H$.

To verify (ii), let $N\lhd H$ be the full preimage of the center of $H/Z$ under the homomorphism $H \to H/Z$. Then  $N$ is nilpotent of class at most $2$, hence it satisfies the assumptions of (i), and therefore it must be central in $H$.
Thus $N \leqslant Z$; on the other hand $Z \leqslant N$ by the definition of $N$. It follows that $N=Z$, and so the image of $N$ in $H/Z$ (i.e., the center of $H/Z$) is trivial.
\end{proof}

For any subset $U$ of $V$ the subgroup $A_U:=\gen{U}$ is said to be a {\it full subgroup} of $A$. It is not difficult to show that $A_U$ is naturally isomorphic to the right angled Artin group
$A(\ga_U)$, where $\ga_U$ the full subgraph of $\ga$ spanned on the vertices from $U$ (see, for example, \cite[Section 6]{Minasyan_hcs}).
For every $U \subseteq V$ there is a {\it canonical retraction} $\rho_U\colon A\to A_U$ defined on the
generators of $A$ by $\rho_U(x)=x$, if $x\in U$ and $\rho_U(x)=1$ if $x\notin U$.

A subgroup $H\leqslant A(\ga)$ is called {\it parabolic} if it is conjugate to a full subgroup, i.e.,  there exist $U\subseteq V$ and $a\in A$ such that $H=a^{-1}A_Ua$; we will say that $H$ is a \emph{proper parabolic subgroup} of $A(\ga)$
if $U \neq V$. If the graph $\ga $ is finite then any subgroup $H \leqslant A(\ga)$ is contained in a unique minimal parabolic subgroup $\pc_\ga(H)$, called the \emph{parabolic closure} of $H$ in $A(\ga)$
(see \cite[Prop.~3.10]{AntolinMinasyan}).

Using the terminology from \cite{AntolinMinasyan}, we will say that a graph $\ga$ is \emph{reducible} if there exists a partition $V=A \sqcup B$ into non-empty disjoint subsets $A$ and $B$ such that
every vertex from $A$ is adjacent to every vertex from $B$ in $\ga$. Otherwise, $\ga$ is said to be \emph{irreducible}.
Alternatively, $\ga$ is irreducible if and only if the \emph{complement graph} $\ga^c$ is connected (recall that $\ga^c$ is defined by $V\ga^c:=V$ and $E\ga^c:=(V \times V) \setminus E$).

Every finite graph $\ga$ can be decomposed into irreducible subgraphs; this means that there is a partition $V= U_1\sqcup \dots \sqcup U_k$, where $U_i \neq \emptyset$, $\ga_{U_i}$ is irreducible for $i=1,\dots,k$,
and for any pair of indices $i \neq j$, every vertex of $U_i$ is adjacent with every vertex of $U_j$ in $\ga$ (this corresponds to the decomposition of $\ga^c$ into the union of its connected components).
Using this we obtain the \emph{standard factorization} of the right angled Artin group $A=A(\ga)$:
$$A = A_0 \times A_1 \times \dots \times A_l,$$ where $A_0$ is a free abelian group (i.e., the right angled Artin group corresponding to a complete subgraph of $\ga$)
and each $A_i$, $i=1,\dots,l$, is a right angled Artin group corresponding to a full irreducible subgraph $\ga_i$, of $\ga$, with $|V\ga_i|\ge 2$.
We will say that $A_0$ is the \emph{abelian factor} of $A$ and $A_1,\dots, A_l$ are the \emph{irreducible factors} of $A$. Note that $A_0$ is central in $A$ by definition (in fact $A_0$ coincides with the
center of $A$, which, for example,  follows from Lemma \ref{lem:no_center} below).

The following fact was proved in \cite[Cor. 3.15]{AntolinMinasyan}:
\begin{lem}\label{prop_parab} Let $\ga$ be a finite irreducible graph and let $A=A(\ga)$ be the associated right angled Artin group. Suppose that $H \leqslant A$ and $N \lhd H$ is a non-trivial normal subgroup of $H$.
If $\pc_\ga(H)=A$ then $\pc_\ga(N)=A$.
\end{lem}

We will also need the following statement, which is a special case of \cite[Cor. 6.20]{MO2}.

\begin{lem} \label{lem:heraags}
Let $A=A(\ga)$ be a right angled Artin group corresponding to some finite irreducible graph $\ga$ with $|V\ga| \ge 2$. Then $A$ acts simplicially and coboundedly by isometries on a simplicial tree $\cT$ so that the following holds.
For any subgroup $H\leqslant  A$ with $\pc_\ga (H) = A$ one has $H \cap \mathcal{L}_{WPD}(A,\cT) \neq\emptyset$.
\end{lem}
Note that the geometric realization of a simplicial tree is $0$-hyperbolic. Therefore, Lemma~\ref{lem:heraags}
shows that the theory which we developed in Section \ref{sec:comm_hom} can be applied to any such $H$.
%
%

\begin{lem}\label{lem:no_center} Let $\ga$ be a finite irreducible graph and let $A=A(\ga)$ be the corresponding right angled Artin group. Suppose that $H \leqslant A$ is a non-cyclic subgroup
such that $\pc_\ga(H)=A$. Then $H$ has trivial center and there is $h \in H\setminus\{1\}$ such that $E_A(h)=\gen{h} \subseteq H$,
where the subgroup $E_A(h)\leqslant A$ is defined as in Remark~\ref{rem:E_G(g)}.
\end{lem}

\begin{proof} Since $H$ is not cyclic, $|V\ga| \ge 2$, and so we can apply Lemma \ref{lem:heraags} to find a simplicial tree $\cT$ such that $A$ acts on $\cT$ isometrically and coboundedly,
and $H \cap \mathcal{L}_{WPD}(A,\cT) \neq\emptyset$.
Recall that right angled Artin groups are torsion-free, hence $E_A(H)=\{1\}$ (see Lemma \ref{lem:E_G(H)}) and $H$ is non-elementary (because it is not cyclic, and a torsion-free elementary group is cyclic).
Therefore we can apply Lemma \ref{lem:special} to find an infinite order element $h \in H$ such that $E_A(h)=\gen{h}$. Moreover, by Lemma \ref{lem:lwpd+}, there is an element $g \in H \cap \mathcal{L}_{WPD}(A,\cT)$
such that $g$ is not commensurable with $h$ in $A$. In view of Remark \ref{rem:elem_inter}, the latter implies that $E_A(h) \cap E_A(g)=\{1\}$. Since this intersection contains the center of $H$, $H$ must be centerless.
\end{proof}

The following simple observation will be useful:
\begin{rem}\label{rem:comm_end_of_ab} If $H$ is a free abelian group then the only commensurating endomorphisms of $H$
are endomorphisms of the form $h \mapsto h^s$ for some $s \in \Z \setminus \{0\}$ and for all $h \in H$.
\end{rem}

We can now prove the main result of this section.

\begin{proof}[Proof of Theorem \ref{thm:comm_end-raag}] Choose a  finite graph $\ga$, with the smallest possible $|V\ga|$, so that the corresponding right angled Artin group $A=A(\ga)$ contains (an isomorphic copy of) $H$.
Let $A=A_0\times A_1 \times \dots \times A_l$ be the standard factorization of $A$, where $A_0$ is the abelian factor of $A$ and $A_1,\dots,A_l$ are the irreducible factors of $A$.
Observe that $A_0$ is a finitely generated free abelian group and $l\ge 1$ as $H$ is non-abelian.
Let $\rho_i\colon A \to A_i$ denote the canonical retraction (in other words, $\rho_i$ is the $i$-th coordinate projection), $i=0,1,\dots,l$.

Note that for every $i \in \{1,\dots,l\}$, the image $\rho_i(H)$ cannot be isomorphic to a subgroup of a right angled Artin group $G$ whose rank is strictly smaller than the rank of $A_i$.
Indeed, otherwise $H$ would embed into the direct product $$P:=A_0 \times A_1 \times A_{i-1} \times G \times A_{i+1} \times \dots \times A_l, $$
which would be a right angled Artin group of smaller rank than $A$, contradicting the choice of $\ga$.
It follows that for each $i \in \{1,\dots,l\}$, $\rho_i(H)$ cannot be cyclic (as the rank of $A_i$ is at least $2$ by the definition of irreducible factors) and the parabolic closure of $\rho_i(H)$ in $A_i$ is $A_i$.

One can also deduce that $N_i:=H \cap A_i\lhd H$ is non-trivial whenever $i =1,\dots,l$, because otherwise $H$ would embed into the direct product of $A_0 \times A_1 \times A_{i-1} \times A_{i+1} \times \dots \times A_l$,
which is a right angled Artin group of smaller rank than $A$.
Observe that $N_i=\rho_i(N_i) \lhd \rho_i(H)$, hence $\pc_{\ga_i}(N_i)=A_i$ by Lemma \ref{prop_parab}, where $\ga_i$ is the full irreducible subgraph of $\ga$ corresponding to $A_i$,
$i=1,\dots,l$. Moreover, $N_i$ cannot be cyclic in view of Lemma \ref{lem:norm-centr}.(i) as the center of $\rho_i(H)$ is trivial by Lemma \ref{lem:no_center}. Hence we can apply Lemma  \ref{lem:no_center}
to $A_i$ and $N_i$ in to find an element $h_i \in N_i\setminus\{1\}$ such that $E_{A_i}(h_i)=\gen{h_i} \subseteq N_i$, $i=1,\dots,l$.

Now consider any commensurating endomorphism $\phi\colon H \to H$. For each $i \in \{0,1,\dots,l\}$ let $B_i\lhd A$ denote the product of all $A_j$, $j \neq i$; thus
$A=A_i B_i \cong A_i \times B_i$ and $B_i= \ker \rho_i$. By the hypothesis, for any $g \in H \cap B_i$, $\phi(g) \in H$ and $\phi(g)^m=u g^n u^{-1} \in B_i$ for some $m,n \in \Z\setminus \{0\}$ and $u \in A$.
And since $A/B_i \cong A_i$ is torsion-free, we can conclude that $\phi(g) \in B_i$. The latter shows that $\phi$ preserves the kernel of the restriction of $\rho_i$ to $H$, $i=0,1,\dots,l$. Therefore
$\phi$ naturally induces an endomorphism $\phi_i\colon \rho_i(H) \to \rho_i(H)$ for $i=0,1,\dots,l$, defined by the formula
$\phi_i(\rho_i(g)):=\rho_i(\phi(g))$ for all $g \in H$.

Evidently, $\phi_i$ will be a commensurating endomorphism of $\rho_i(H)$ for each $i=0,1,\dots,l$. Therefore, according to Remark \ref{rem:comm_end_of_ab}, there must exist $s \in \Z \setminus\{0\}$ such that
$\phi_0(a)=a^s$ for all $a \in \rho_0(H)$. On the other hand, if $i \in \{1,\dots,l\}$, we can recall that $A_i$ is an irreducible factor of $A$ and  $\rho_i(H)$ is a non-elementary subgroup of $A_i$ such that the parabolic closure of
$\rho_i(H)$ in $A_i$ is $A_i$. Therefore, in view of Lemma \ref{lem:heraags}, all the assumptions of Theorem \ref{thm:comm} are satisfied, hence there exists $w_i \in A_i$ such that $\phi_i(a)=w_i a w_i^{-1}$ for all
$a \in \rho_i(H)$ (here we used the fact that $E_{A_i}(\rho_i(H))=\{1\}$ as $A_i$ is torsion-free), $i=1,\dots,l$.

Let $\psi \in Inn(A)$ be the inner automorphism defined by $\psi(g):=wgw^{-1}$ for all $g \in A$, where $w:=w_1 \cdots w_l \in A$. Let us show that the endomorphism $\phi$ is actually the restriction of $\psi$ to $H$.
The preceding paragraph implies that this is true if the abelian factor $A_0$ is trivial, because in this case for every $g \in H$ one would have $g=\rho_1(g) \cdots \rho_l(g)$, and so
$$\phi(g)=\rho_1(\phi(g)) \cdots \rho_l(\phi(g))=\phi_1 \left(\rho_1(g)\right) \cdots \phi_l \left( \rho_l(g)\right)=\rho_1(g)^{w_1} \cdots \rho_l(g)^{w_l}=g^w.$$

On the other hand, if $A_0$ is non-trivial, then $N_0:=H \cap A_0$ is also non-trivial (by the minimality of the rank of $A$). So, pick any
$h_0 \in N_0\setminus\{1\}$. Let $h_1 \in N_1= H\cap A_1$ be the element constructed above. Since $\phi$ is commensurating and $h_0h_1 \in H$, there must exist $m,n \in \Z\setminus\{0\}$ and $u \in H$ such that
$$\phi(h_0h_1)^m=u (h_0h_1)^n u^{-1}=h_0^n u_1 h_1^n u_1^{-1}, \text{ where } u_1:=\rho_1(u) \in A_1.$$
But we also have $\phi(h_0 h_1)=\phi_0(h_0) \phi_1(h_1)=h_0^s w_1h_1w_1^{-1}$. Therefore
\begin{equation*}\label{eq:left-right}
h_0^{sm} w_1h_1^mw_1^{-1}=h_0^n u_1 h_1^n u_1^{-1}.
\end{equation*}

Applying $\rho_0$ and $\rho_1$ to the above equation we obtain $h_0^{sm}=h_0^n$ and $u_1^{-1}w_1 h_1^m w_1^{-1} u_1=h^n$.
The former yields that $n=sm$; and the latter shows that $u_1^{-1}w_1 \in E_{A_1}(h_1)=\gen{h_1}$,
in particular this element commutes with $h_1$. Thus $h_1^m=h_1^n$, and so $m=n$. Consequently, $s=1$, which implies that $\phi(g)=w g w^{-1}=\psi(g)$ for all $g \in H$.
If $w \in H$ then the proof would have been finished. However, this may not be the case, so one more step is needed.

Let $h_i \in N_i =H \cap A_i$, $i=1,\dots,l$, be the elements constructed above so that  $E_{A_i}(h_i)=\gen{h_i} \subseteq H$,
and set $h:= h_1 \cdots h_l \in H$. By the  assumption, there exist $m,n \in \Z\setminus\{0\}$ and $u \in H$ such that $\phi(h)^m=u h^n u^{-1}$. On the other hand, we know that $\phi(h)=w h w^{-1}$.
Combining these two equalities one gets $ w h^m w^{-1}=u h^n u^{-1}$ in $A$. Applying $\rho_i$ yields that $u_i^{-1}w_i \in E_{A_i}(h_i)=\gen{h_i}$, where $u_i:=\rho_i(u) \in A_i$, for $i=1,\dots,l$. It follows that
for every $i=1,\dots,l$, there exists $t_i \in \Z$ such that $w_i=u_i h_i^{t_i}$ in $A_i$. Thus, denoting $u_0:=\rho_0(u) \in A_0$, we achieve
$$w=w_1 \cdots w_l= u_1 h_1^{t_1} \cdots u_l h_l^{t_l}=  u_0^{-1} u h_1^{t_1} \cdots h_l^{t_l}=u_0^{-1} v,$$
where the element  $v:=u h_1^{t_1} \cdots h_l^{t_l}$ belongs to $H$ by construction. Since  $u_0 \in A_0$ is central in $A$, we see that $\phi(g)=wgw^{-1}=v g v^{-1}$ for all $g \in H$, thus $\phi$ is
indeed an inner automorphism of $H$.
\end{proof}

\begin{rem} The claim of Theorem \ref{thm:comm_end-raag} would be no longer true if one dropped the assumption that the ambient right angled Artin group is finitely generated. Indeed,
let $G$ be the direct product of infinitely (countably)  many copies of the free group of rank $2$. Then $G$ is a normal subgroup in the cartesian (i.e., unrestricted) product $P$ of
these free groups and any inner automorphism of $P$ induces a pointwise inner automorphism of $G$. It follows that $G$
has uncountably many pointwise inner (hence, commensurating) but non-inner automorphisms.
\end{rem}


\section{Criteria for residual finiteness of outer automorphism groups}\label{sec:crit}
Recall that, given a group $G$, the \emph{profinite topology} on $G$ is the topology whose basic open sets are cosets to normal subgroups of finite index in $G$. It is easy to see that group operations and group
homomorphisms are continuous with respect to this topology. In particular,  $G$, equipped with this topology, is a topological group.
One can also observe that the profinite topology is Hausdorff if and only if $\{1\}$ is a closed subset of $G$ if and only if $G$ is residually finite. It follows that any finite subset of a residually finite group is closed (in the profinite topology).

If $N \lhd G$, then $G/N$ is residually finite if and only if $N$ is closed in $G$. Thus if $G$ is residually finite and $|N|<\infty$ then $G/N$ is also residually finite.
Finally, residual finiteness is preserved under taking subgroups or overgroups of finite index.

\begin{rem}\label{rem:rf} Suppose that $G$ is a group and for every $g \in G \setminus \{1\}$ there is a homomorphism $\psi$ from $G$ to a residually finite group $K$ such that $\psi(g) \neq 1$. Then
$G$ is residually finite.
\end{rem}

In this section we discuss various conditions one can impose on $G$ to ensure residual finiteness of $Out(G)$. One set of conditions is given by Grossman's criterion \cite{Grossman}, mentioned in the Introduction.
In particular, since any pointwise inner automorphism is commensurating, we can combine this criterion with Theorem \ref{thm:comm_end-raag} to obtain

\begin{cor}\label{cor:cs+raag->out-rf}
Let $G$ be a finitely generated  conjugacy separable subgroup of a right angled Artin group. Then $Out(G)$  is residually finite.
\end{cor}

In \cite[Cor. 2.1]{Minasyan_hcs} the second author proved that groups from the class $\vr$ (i.e., virtual retracts of finitely generated right angled Artin groups)
are conjugacy separable. Since these groups are finitely generated (and even finitely presented), as virtual retracts of finitely
presented groups, we can apply Corollary \ref{cor:cs+raag->out-rf} to achieve
\begin{cor}\label{cor:vr}
If $G \in \vr$ then $Out(G)$ is residually finite.
\end{cor}

Another useful tool for establishing residual finiteness of $Out(G)$ is given by the following observation:

\begin{lem}[{\cite[Lemma 5.4]{G-L-i}}] \label{lem:outnormsub}
Suppose that $G$ is a finitely generated group, and $N$ is a centerless normal
subgroup of finite index in $G$.
Then some finite index subgroup
$Out_0(G)\leqslant Out(G )$ is isomorphic to a quotient of a subgroup of $Out(N )$ by
a finite normal subgroup. In particular, if $Out(N)$ is residually finite then so is $Out(G)$.
\end{lem}

For our purposes we will also need a criterion (see Proposition \ref{prop:new_crit} below) which applies when the center of $N$ is non-trivial.

Given a subgroup $H \leqslant G$,  define $Aut(G;H):=\{ \alpha \in Aut(G) \mid \alpha(H)=H\}\leqslant Aut(G)$,
and let $Out(G;H)$ be its image in $Out(G)$. Since a finitely generated group contains only finitely many subgroups of any given finite index, the following observation can be made:

\begin{rem} \label{rem:out(G;H)} If $G$ is a finitely generated group and $H \leqslant G$ has finite index then $|Aut(G):Aut(G;H)|<\infty$ and $|Out(G):Out(G;H)|< \infty$.
\end{rem}

If $Q$ is an \emph{abelian} group and $n \in \N$ then $Q^n:=\{z^n \mid z \in Q\}$ is called a \emph{congruence subgroup} of $Q$.
Clearly, every finite index subgroup of $Q$ contains $Q^n$ for some $n \in \N$.
If, additionally, $Q$ is finitely generated, then $|Q:Q^n|<\infty$  for all $n \in \N$, hence
the profinite topology of $Q$ is generated by the  congruence subgroups. It follows that for any fixed $m \in \N$, the profinite topology on $Q$ is also generated by the collection $\{Q^{mn} \mid n \in \N\}$.

Let us also specify some notation. If $x,y$ are elements of a group $G$, we will write $x^y$ for the conjugate $yxy^{-1}$ and $[x,y]$ for the commutator $xyx^{-1}y^{-1}$. If $E \subseteq G$
then $E^y$ and $[E,x]$ will denote the subsets $\{e^y \mid e \in E\} \subseteq G$ and $\{[e,x] \mid e \in E\} \subseteq G$ respectively.

\begin{prop}\label{prop:new_crit} Let $G$ be a finitely generated group, let $N\lhd G$ be a normal subgroup of finite index such that the center $Z=Z(N)$,  of $N$, is finitely generated. Suppose that
$Out(G/Z)$ is residually finite and there is $m \in \N$ such that $Out(G/Z^{mn})$ is residually finite for all $n \in \N$. Then $Out(G)$ is also residually finite.
\end{prop}

\begin{proof} In view of Remark \ref{rem:out(G;H)} and since $Out(G;N)\leqslant Out(G;Z)$
(because  $Z$ is a characteristic subgroup of $N$), it is enough to prove that $Out(G;Z)$ is residually finite.  So, consider any $\alpha \in Aut(G;Z) \setminus Inn(G)$ (note that $Inn(G) \leqslant Aut(G;Z)$ as $Z \lhd G$)
and let $\bar \alpha \in Out(G;Z)$ denote its image in $Out(G)$.

Note that $\alpha(Z^n)=Z^n$ for every $n \in \N$, hence $\alpha$ naturally induces
an automorphism of $G/Z^n$ (as it permutes the cosets of $Z^n$ in $G$). This gives rise to the following commutative diagram between automorphism groups:
$$
 \xymatrix{
Aut(G;Z) \ar[r]\ar[d] &  Aut(G/Z^n; Z/Z^n)\ar[r]\ar[d]  &   Aut(G/Z) \ar[d] \\
Out(G;Z) \ar[r]          &  Out(G/Z^n; Z/Z^n)\ar[r]          &   Out(G/Z)
}
$$

In view of the assumptions and Remark \ref{rem:rf}, to prove the proposition it is enough to show that there exists
$s \in \{1\} \cup m\N$ such that the image of $\bar \alpha$ in $Out(G/Z^{s})$ (coming from the commutative diagram above) is non-trivial.

If $\alpha$ induces a non-inner automorphism of $G/Z$, then  the image of $\bar \alpha$ will be non-trivial in $Out(G/Z)$. Thus, we can now suppose that $\alpha$ induces an inner
automorphism of $G/Z$.  This means that we can replace $\alpha$ by its composition with an inner automorphism of $G$ (this does not affect $\bar \alpha$) to further assume that $\alpha$ induces the identity on $G/Z$.
In other words, $\alpha(g)g^{-1} \in Z$ for all $g \in G$.

Choose a finite generating set $\{x_1,\dots,x_k\}$ of $G$. Then for every $i=1,\dots,k$, there is $z_i \in Z$ such that $\alpha(x_i)=z_ix_i$. Let $C$ be the full preimage in $G$ of the center $Z(G/Z)$, and set $C_1:=C \cap N$.

Let $P=G\times \dots \times G$ be the $k$-th direct power of $G$, let $Q=Z \times \dots \times Z \leqslant P$ be the $k$-th direct power of $Z$, and let $D:=\{(g,\dots,g)\mid g \in G\}\leqslant P$
be the corresponding diagonal subgroup of $P$.

Observe that for any given $n \in \N$, $\alpha$ induces an inner automorphism of $G/Z^n$  if and only if  there exists $a \in C$ such that $\alpha(x_i) \equiv a x_i a^{-1} \text{ (mod $Z^n$)}$ for every $i=1,\dots,k$.
The latter equality can be re-written as $z_i \equiv [a,x_i]  \text{ (mod $Z^n$)}$ in $G$.
Thus $\alpha$ induces an inner automorphism of $G/Z^n$  if and only if $(z_1,\dots,z_k) \in [E,(x_1,\dots,x_k)]$ $ \text{ (mod $Q^n$)}$, where $E:=(C \times \dots \times C) \cap D\leqslant P$.
Note that $[E,(x_1,\dots,x_k)]\subseteq Q$ as $[a,g] \in Z$ for all $a \in C$, $g \in G$, by the definition of $C$.

Observe that the subgroup $E_1:=(C_1 \times \dots \times C_1) \cap D\leqslant Q$ has finite index in $E$ (because $|C:C_1|<\infty$).
Moreover, if $c,c' \in C_1$ then $[c c',g]=[c,g][c',g]$ for any $g \in G$. This can be derived from the commutator identities,
because $[c,g], [c',g] \in Z$, and $Z$ is an abelian subgroup centralized by $C_1 \leqslant N$. It follows that $[E_1,(x_1,\dots,x_k)]$ is actually a subgroup of the finitely generated abelian group $Q$.
Therefore, $[E_1,(x_1,\dots,x_k)]$ is closed in the profinite topology of $Q$ (in fact any subgroup $H \leqslant Q$ is closed because the quotient $Q/H$ is again a finitely generated abelian group, and so it is residually finite as a
direct sum of cyclic groups).

By construction, there exist $e_1,\dots,e_l \in E$ such that $E=\bigcup_{j=1}^l e_j E_1$. Utilizing commutator identities once again, we get
$$ [E,(x_1,\dots,x_k)]=\bigcup_{j=1}^l [E_1,(x_1,\dots,x_k)]^{e_j} [e_j,(x_1,\dots,x_k)].$$ This shows that $[E,(x_1,\dots,x_k)]$ is also a closed subset of $Q$, as finite union of closed subsets.
Recall, that $\alpha \notin Inn(G)$, therefore $(z_1,\dots,z_k) \notin [E,(x_1,\dots,x_k)]$ in $Q$. It follows that we can find $n \in \N$ such that  $(z_1,\dots,z_k) \notin [E,(x_1,\dots,x_k)]$ $\text{ (mod $Q^{mn}$)}$.
The latter demonstrates that $\alpha$ induces a non-inner automorphism of $G/Z^{mn}$, which finishes the proof of  the proposition.
\end{proof}

\begin{rem} The proof of Proposition \ref{prop:new_crit} actually shows that if $G$ is a finitely generated group and $N\lhd G$ is a finite index normal subgroup such that the center $Z$, of $N$, is finitely generated then
for any $m \in \N$, $Out(G;Z)$ embeds into the cartesian product $Out(G/Z) \times \prod_{n\in \N} Out(G/Z^{mn})$.
\end{rem}


\section{Residual finiteness of outer automorphism groups of groups from $\mathcal{AVR}$}
In this section we will prove Theorem \ref{thm:avr}. In view of Corollary \ref{cor:vr} and Lemma \ref{lem:outnormsub}, essentially it remains to deal with the case when a finite index normal subgroup
$N \in \vr$ of  a group $G \in \avr$  has non-trivial center.

\begin{lem} \label{lem:descr_of_center} Let $A$ be the right angled Artin group corresponding to a finite graph $\ga$ and let $A=A_0 \times A_1 \times \dots \times A_l$ be its standard factorization, where $A_0$ is the abelian factor
and $A_1,\dots,A_l$ are the irreducible factors of $A$.
Suppose that $H \leqslant A$ is a subgroup such that $\pc_\ga(H)=A$ and $\rho_i(H)$ is not cyclic, for each $i=1,\dots,l$, where $\rho_i
\colon A \to A_i$ denotes the canonical retraction. Then the center of $H$ is equal to the
intersection of $H$  with $A_0$.
\end{lem}

\begin{proof}
Consider any $i \in \{1,\dots,l\}$ and observe that if $\rho_i(H)$ is contained in a proper parabolic subgroup $aB_ia^{-1}$ of $A_i$ (where $a \in A_i$ and $B_i$ is a full subgroup of $A_i$, and, hence, of $A$), then
$H$ is contained in the subgroup $a\left (A_0A_1 \dots A_{i-1} B_i A_{i+1}\dots A_l  \right)a^{-1}$, which is a proper parabolic subgroup of $A$, contradicting the assumption that $\pc_\ga(H)=A$. Therefore
the parabolic closure of $\rho_i(H)$ in $A_i$ is the whole of $A_i$, $i=1,\dots,l$.

Let $Z$ denote the center of $H$. Then $\rho_i(Z)$ is contained in the center of $\rho_i(H)$, which is trivial for $i=1,\dots,l$, by Lemma \ref{lem:no_center}. Thus $\rho_i(Z)=\{1\}$ for each $i \in \{1,\dots,l\}$,
which implies that $Z \leqslant A_0$. Evidently, $H \cap A_0 \leqslant Z$ because $A_0$ is central in $A$, hence $Z=H\cap A_0$, as claimed.
\end{proof}

It is not difficult to see that the class $\vr$ is closed under taking subgroups of finite index (see \cite[Remark 9.4]{Minasyan_hcs}). To prove the main result of this section we will also need the
fact that this class is closed under taking quotients by the center:

\begin{prop}\label{prop:quot_by_center} Let $C$ be a finitely generated right Angled Artin group, let $H \leqslant C$ be an arbitrary subgroup and let $Z$ be the center of $H$.
\begin{itemize}
\item[(a)] For any subgroup $Z_1 \leqslant Z$, $Z_1$ is finitely generated and the quotient $H/Z_1$ is residually finite.
\item[(b)] If $H$ is a virtual retract of $C$ then $H/Z \in \vr$.
\end{itemize}
\end{prop}

\begin{proof}
Since $C$ has finite rank, there exists a right angled Artin subgroup $A \leqslant C$ which contains $H$ and has minimal rank (among all such subgroups of $C$).
Let $\ga$ be the finite simplicial graph corresponding to $A$ and let
$A=A_0 \times A_1 \times \dots \times A_l$  be the standard factorization of $A$, where $A_0$ is the abelian factor and
$A_1,\dots,A_l$ are the irreducible factors of $A$. If $l=0$ then the groups $H$ and $A=A_0$ are free abelian of finite rank, hence
both statements are evidently true. Therefore we can assume that $l \ge 1$. Let $\rho_i\colon A \to A_i$ denote the canonical projection of $A$ onto $A_i$, $i=0,1,\dots,l$.

We remark that $\pc_\ga(H)=A$, by the choice of $A$.
If $\rho_i(H)$ is a cyclic subgroup $B$ of $A_i$ for some $i \in \{1,\dots,l\}$, then $H$ embeds into the subgroup $P \leqslant A$ where
$$P:=A_0  A_1  \dots  A_{i-1} B  A_{i+1}  \dots A_l \cong A_0 \times A_1\times  \dots \times A_{i-1}\times B \times  A_{i+1}\times  \dots \times A_l,$$
which is a right angled Artin group of strictly smaller rank than $A$, contradicting the choice of $A$.
Therefore we can conclude that $\rho_i(H)$ is non-cyclic for every $i\in \{1,\dots,l\}$.
Thus we are able to apply Lemma \ref{lem:descr_of_center}, claiming that $Z=H \cap A_0$.

Consider any subgroup $Z_1 \leqslant Z \leqslant A_0$. Since $A_0$ is a finitely generated abelian group we see that $Z_1$ is also finitely generated.
Moreover,  the quotient $H/Z_1$ naturally embeds into the quotient $A/Z_1 \cong A_0/Z_1 \times A_1 \times \dots \times A_l$. Therefore $A/Z_1$ (and hence $H/Z_1$)
is residually finite, as a direct product of residually finite groups: $A_0/Z_1$ is a finitely generated abelian group and $A_i$, $i=1,\dots,l$, are  right angled Artin groups,
whose residual finiteness is well-known (see \cite[Ch. 3, Thm 1.1]{Dromsphd} or \cite[Cor. 3.5]{Hsu-Wise}). Thus (a) is proved.

To prove (b) assume that $H$ is a virtual retract of $C$. This implies that for any subgroup $D \leqslant C$ such that $H \subseteq D$,
$H$ is a virtual retract of $D$. In particular, $H$ will also be a virtual retract of $A$.
Thus $A$ contains a finite index subgroup $K$ such that $H \subseteq K$ and there is a retraction $\theta: K \to H$.
Since $A_0$ is central in $A$, $K \cap A_0$ is central in $K$, and so $\theta(K \cap A_0) \subseteq Z$.

Consider the canonical projection $\xi\colon A \to A/A_0 \cong A_1 \times \dots \times A_l$, and observe that
$$\theta (K \cap \ker \xi)=\theta(K \cap A_0) \subseteq Z =H \cap A_0 \subseteq K \cap \ker \xi.$$
It follows (see \cite[Lemma 4.1]{Minasyan_hcs}) that  $\theta$ naturally induces a retraction $\bar \theta$ of $\xi(K)$
onto its subgroup $\xi(H)$. Thus $\xi(H)$ is a retract of $\xi(K)$, and the latter has finite index in the finitely generated right angled Artin group $A/A_0$, because $|A:K|<\infty$.
It remains to recall that $H \cap \ker \xi=H \cap A_0=Z$, hence $\xi(H) \cong H/Z$. Thus $H/Z \in \vr$, and the proposition is proved.
\end{proof}

\begin{rem} \begin{enumerate} \item[]
 \item Part (a) of Proposition \ref{prop:quot_by_center} can actually be derived from more general results. Indeed, it is known that the finitely generated right angled Artin group $C$
can be embedded into ${\rm GL}_k(\Z)$ for some $k \in \N$. Therefore every solvable subgroup $B \leqslant H \leqslant {\rm GL}_k(\Z)$ is polycyclic
(hence, finitely generated) and is closed in the profinite topology of ${\rm GL}_k(\Z)$  by a result of Segal \cite[4.C, Thm. 5]{Segal}.
Since the profinite topology of $H$ is finer than the topology induced by the profinite topology of ${\rm GL}_k(\Z)$, we can conclude that $B$ is closed in $H$.

\item Since right angled Artin groups are CAT($0$), it is easy to prove a weaker version of Proposition~\ref{prop:quot_by_center}.(b), that $H/Z \in \mathcal{AVR}$, using
the Flat Torus Theorem \cite[II.7.1.(5)]{BH}.
\end{enumerate}
\end{rem}

Combining Proposition \ref{prop:quot_by_center} with Lemma \ref{lem:norm-centr}.(ii) one immediately obtains

\begin{cor}
Suppose that $H \in \vr$ and $Z$ is the center of $H$. Then the group $H/Z$ is centerless and belongs to $\vr$.
\end{cor}

\begin{proof}[Proof of Theorem \ref{thm:avr}]
Let $G$ be any group from the class $\avr$. This means that $G$ contains a finite index subgroup $H \in \vr$; in particular, $G$ is finitely generated.
Note that $N:=\bigcap_{g \in G} H^g$ is a finite index normal subgroup of $G$, and $N \in \vr$ because the class $\vr$ is closed under taking finite index subgroups.

Let $Z$ denote the center of $N$. We are going to check that all the assumptions of the criterion from Proposition~\ref{prop:new_crit} are satisfied.
First, the fact that $Z$ is finitely generated follows from Proposition~\ref{prop:quot_by_center}.(a).
Second,  take any $n \in \N$ and note that $N/Z^n$ is residually finite,  also by Proposition~\ref{prop:quot_by_center}.(a).
Hence there is a finite index normal subgroup $M \lhd N/Z^n$ such that $M$ has trivial intersection with the finite subgroup $Z/Z^n$ in $N/Z^n$.
Again, we can replace $M$ with the intersection of all its conjugates in $G/Z^n$ to
further assume that $M \lhd G/Z^n$.

By construction, $M$ injects into the quotient $N/Z$ under the natural epimorphism $N/Z^n \to N/Z$. Let $\overline{M} \cong M$ denote the image of $M$ in $N/Z \leqslant G/Z$.
Since $N/Z \in \vr$ by Proposition~\ref{prop:quot_by_center}.(b) and $\overline{M}$ has finite index in $N/Z$, we see that $\overline{M} \in \vr$. Hence, according to Corollary \ref{cor:vr},
$Out(\overline{M}) \cong Out(M)$ is residually finite. Moreover, since the center of $N/Z \in \vr$ is trivial (Lemma~\ref{lem:norm-centr}.(ii)), the center of $Z(\overline{M})$ must be trivial as well
(because $Z(\overline{M})$ is an abelian normal subgroup of $N/Z$ and so it is central in $N/Z$ by Lemma~\ref{lem:norm-centr}.(i)). Therefore $M\cong \overline{M}$ is a centerless finite index normal subgroup
in $G/Z^n$ with a residually finite outer automorphism group. Consequently, Lemma \ref{lem:outnormsub} yields that $Out(G/Z^n)$ is residually finite.
Since this works for arbitrary $n \in \N$, we see that all the assumptions of  Proposition~\ref{prop:new_crit}  are satisfied.
It remains to apply this proposition to conclude that $Out(G)$ is residually finite, which
finishes the proof of the theorem.
\end{proof}


\section{Outer automorphisms of $3$-manifold groups}\label{sec:manifolds}
This last section of the paper is dedicated to proving Theorem \ref{thm:3-man}. We start with the following lemma, which allows to deal with the Seifert fibered case.

\begin{lem}\label{lem:Seifert}
Suppose that $G$ is a finitely generated group containing a finite index subgroup $H$ that fits into the short exact sequence
$$\{1\} \to K \to H \to L \to \{1\},$$ where
$K$ is a cyclic group and $L$ has a finite index subgroup which embeds into the fundamental group of a compact surface. Then $Out(G)$ is residually finite.
\end{lem}

\begin{proof} Since $K$ is cyclic, its automorphism group is finite.  Moreover, $H$ acts on $K$ by conjugation because $K \lhd H$, and the kernel of this action is the centralizer of $K$ in $H$.
It follows that  $|H:C_H(K)|<\infty$. Combining this with the other assumptions on $G$ and $H$,  we can find a finite index normal subgroup $N\lhd G$ such that  $N \leqslant H$, $Z\coloneq N \cap K$ is central in $N$ and
$N/Z$ is a subgroup of some compact surface group. Note that $N$ is finitely generated, as this is true for $G$ by the hypothesis, hence the quotient $N/Z$ is itself isomorphic to
the fundamental group of some compact surface $\Sigma$
(because any finitely generated subgroup of a surface group is itself a surface group).
Evidently we can also assume that $\Sigma$ is orientable.
It follows that $N/Z\cong \pi_1(\Sigma)$ is either abelian (isomorphic to $\{1\}$, $\Z$ or $\Z^2$)
or is non-elementary torsion-free hyperbolic. In the former case, $N$ is polycyclic, hence $G$ is virtually polycyclic and so $Out(G)$ is residually finite (according to a theorem of Wehrfritz \cite{Wehrfritz}, $Out(G)$ is linear over $\Z$).
Thus we can assume that $N/Z$ is a torsion-free hyperbolic group, which, in particular, implies that it is centerless and so the cyclic subgroup $Z$ is equal to the center of $N$.

Now, in order to apply Proposition \ref{prop:new_crit}, we check that $Out(G/Z^n)$ is residually finite for any $n \in \N$. Indeed, observe that $Z/Z^n$ is a finite central subgroup of
$N/Z^n$ such that the quotient $(N/Z^n)/(Z/Z^n)\cong N/Z$ is isomorphic to the surface group $\pi_1(\Sigma)$. It follows that $N/Z^n$ possesses a
finite index normal subgroup $M \lhd N/Z^n$ which intersects $Z/Z^n$ trivially (see \cite[Lemma~4.2]{Martino}). Thus the image $\overline{M}$, of $M$ in $N/Z$, is naturally isomorphic to $M$ and has finite index in $\pi_1(\Sigma)$.
Consequently, $\overline{M}$ is itself isomorphic to the fundamental group of a compact orientable surface, which finitely covers $\Sigma$.
By Grossman's theorem \cite{Grossman}, $Out(M) \cong Out(\overline{M})$ is residually finite;
moreover, $M \cong \overline{M}$ is centerless because it is a non-elementary torsion-free hyperbolic group (as it has finite index in $N/Z$).
Since $M$ has finite index in $G/Z^n$, Lemma~\ref{lem:outnormsub} implies that $Out(G/Z^n)$
is residually finite for any $n \in \N$. Therefore we can use Proposition~\ref{prop:new_crit} to conclude that $Out(G)$ is residually finite.
\end{proof}

One of the main ingredients of the proof of Theorem \ref{thm:3-man} is the following beautiful result of Hamilton, Wilton and Zalesskii,
which is based on the deep work of Wise \cite{Wise-QH} and Agol \cite{Agol} mentioned in the Introduction.

\begin{thm}[{\cite[Thm. 1.3]{HWZ}}]\label{thm:HWZ} If $\cM$ is a compact orientable $3$-manifold, then $\pi_1(\cM)$ is conjugacy
separable.
\end{thm}

The other ingredient comes from the following trichotomy, established by the second author and Osin:

\begin{thm}[{\cite[Thm. 5.6]{MO2}}]\label{thm:trichotomy} Let $\cM$ be a compact
$3$-manifold and let $H$ be a subgroup of $\pi_1(\cM)$. Then exactly one of the following three conditions holds.
\begin{enumerate}
\item[(I)] $H$ is acylindrically hyperbolic with $E_H(H)=\{1\}$;

\item[(II)] $H$ contains an infinite cyclic normal subgroup $K$ such that $H/K$ is virtually a subgroup of the fundamental group of a compact surface;

\item[(III)] $H$ is virtually polycyclic.
\end{enumerate}
\end{thm}

\begin{proof}[Proof of Theorem \ref{thm:3-man}]
By the assumptions $H\coloneq \pi_1(\cM)$ has finite index in $G$, thus both $H$ and $G$ are finitely generated  because $\cM$ is compact.
Moreover, after replacing $H$ with some finite index subgroup, we can assume that $H \lhd G$.

If $H$ is virtually polycyclic then so is $G$, hence $Out(G)$ is residually finite by Wehrfritz's theorem \cite{Wehrfritz}. If $H$ satisfies condition (II) of Theorem \ref{thm:trichotomy}
then  $Out(G)$ is residually finite by Lemma \ref{lem:Seifert}.

Thus, in view of Theorem \ref{thm:trichotomy}, we can assume that $H$ is acylindrically hyperbolic and $E_H(H)=\{1\}$. Therefore we can apply Corollary \ref{cor:conj_aut}, stating that
every pointwise inner automorphism of $H$ is inner. Recall that $H$ is finitely generated and conjugacy separable by Theorem \ref{thm:HWZ}, hence $Out(H)$ is residually finite
by Grossman's criterion \cite[Thm. 1]{Grossman}.

It remains to observe that the center $Z(H)$, of $H$, is finite because $H$ is acylindrically hyperbolic (see \cite[Cor. 4.34]{Osin-acyl}),
hence $Z(H) \leqslant E_H(H)=\{1\}$, i.e., $H$ is centerless. Consequently, Lemma \ref{lem:outnormsub} implies that $Out(G)$ is residually finite.
\end{proof}



\begin{thebibliography}{99}

\bibitem{Agol} I. Agol,
\newblock {\em The virtual Haken conjecture}. With an appendix by I. Agol, D. Groves and J. Manning.
\newblock  {Documenta Math}. \textbf{18} (2013) 1045--1087.

\bibitem{AKT-1} R.B.J.T. Allenby, G. Kim and C.Y. Tang, \emph{Outer automorphism groups of certain orientable Seifert 3-manifold groups}.  Combinatorial group theory, discrete groups, and number theory, 15--22,
Contemp. Math., \textbf{421}, Amer. Math. Soc., Providence, RI, 2006.

\bibitem{AKT-2}  R.B.J.T. Allenby, G. Kim and C.Y. Tang, \emph{Outer automorphism groups of Seifert 3-manifold groups over non-orientable surfaces.}
J. Algebra \textbf{322} (2009), no. 4, 957--968.

\bibitem{AntolinMinasyan}
Y. Antol\'{i}n and A. Minasyan,
\newblock {\em Tits alternatives for graph products}.
\newblock  J. Reine Angew. Math. \textbf{704} (2015), 55--83.

\bibitem{AMO}  G. Arzhantseva, A. Minasyan, and D. Osin,  \newblock \emph{The SQ-universality and residual properties of relatively hyperbolic groups}.
\newblock J. Algebra \textbf{315} (2007), no. 1, 165--177.

\bibitem{AFW} M. Aschenbrenner, S. Friedl and H. Wilton,
\newblock \emph{$3$-manifold groups.}
\newblock EMS Series of Lectures in Mathematics. European Mathematical Society (EMS), Z\"{u}rich, 2015. xiv+215 pp.


\bibitem{Baud}
A. Baudisch,
\newblock \textit{Subgroups of semifree groups.}
\newblock Acta Math. Acad. Sci. Hungar. \textbf{38} (1981), no. 1-4, 19--28.

\bibitem{Baumslag} G. Baumslag, \emph{Automorphism groups of residually finite groups.} J. London Math. Soc. \textbf{38} (1963),
117--118.


\bibitem{BestvinaFujiwara}
M. Bestvina and K. Fujiwara,
\newblock {\em Bounded cohomology of subgroups of
mapping class groups}.
\newblock Geom. Topol. {\bf 6} (2002) 69--89.

\bibitem{Bow} B. Bowditch,
\newblock {\em Tight geodesics in the curve complex.}
\newblock Invent. Math. \textbf{171} (2008), no. 2, 281--300.

\bibitem{BH}
M. Bridson and A. Haefliger,
\newblock{\em Metric spaces of non-positive curvature.}
\newblock  Grundlehren der
Mathematischen Wissenschaften [Fundamental Principles of Mathematical Sciences], 319.
Springer-Verlag, Berlin, 1999. xxii+643 pp.

\bibitem{Bum-Wise} I. Bumagin, D. Wise, \emph{Every group is an outer automorphism group of a finitely generated group.}
J. Pure App. Algebra \textbf{200} (2005), 137--147.





\bibitem{Burnside} W. Burnside, \emph{On the outer isomorphisms of a group.} Proc. London Math. Soc. (2)
\textbf{11} (1913), 40--42.

\bibitem{Cap-Fuj} P.-E. Caprace and K. Fujiwara,
\emph{Rank-one isometries of buildings and quasi-morphisms of Kac-Moody groups.}
Geom. Funct. Anal. \textbf{19} (2010), no. 5, 1296--1319.

\bibitem{Cap-Min} P.-E. Caprace and A. Minasyan,
\emph{On conjugacy separability of some Coxeter groups and parabolic-preserving automorphisms.}  Illinois J. Math. \textbf{57} (2013), no. 2, 499--523.

\bibitem{Char-Vogt} R. Charney, K.  Vogtmann, \emph{Subgroups and quotients of automorphism groups of RAAGs.} Low-dimensional and symplectic topology, 9--27,
Proc. Sympos. Pure Math., \textbf{82}, Amer. Math. Soc., Providence, RI, 2011.


\bibitem{DGO}
F. Dahmani, V. Guirardel and D. Osin,
\newblock {\em Hyperbolically embedded subgroups and rotating families in
groups acting on hyperbolic spaces}.
\newblock Memoirs of the AMS (to appear). \texttt{arXiv:1111.7048v5}

\bibitem{Deod} V.V. Deodhar,
\emph{On the root system of a Coxeter group.}
Comm. Algebra \textbf{10} (1982), no. 6, 611--630.

\bibitem{Dromsphd}
C. Droms,
\newblock {\em Graph groups.}
\newblock PhD thesis, Syracuse University, 1983.
Available from \\ \url{http://educ.jmu.edu/~dromscg/vita/thesis/}



\bibitem{Goryaga}
A.V. Goryaga,
\newblock {\em Example of a finite extension of an FAC-group that
is not an FAC-group}. (Russian).
\newblock Sibirsk. Mat. Zh. {\bf 27} (1986), no. 3, 203--205, 225.


\bibitem{Grossman}
E. Grossman,
\newblock {\em On the residual finiteness of certain mapping class groups}.
\newblock J. London Math. Soc. {\bf 9} (1975/75), No. 2, pp. 160--164.


\bibitem{G-L-i}
V. Guirardel and G. Levitt,
\newblock  {\em The outer space of a free product.}
\newblock Proc. London Math. Soc. {\bf 94} (2007), no. 3, 695--714.

\bibitem{Hag-Wise}
F. Haglund and D.T. Wise,
\newblock {\em Special cube complexes.}
\newblock Geom. Funct. Anal. \textbf{17} (2008), no. 5, 1551--1620.

\bibitem{Hag-Wise-cox}
F. Haglund and D.T. Wise,
\newblock {\em Coxeter groups are virtually special.}
\newblock Adv. Math. \textbf{224} (2010), no. 5, 1890--1903.

\bibitem{Ham}
U. Hamenst\" adt,
\newblock {\em Bounded cohomology and isometry groups of hyperbolic spaces.}
\newblock J. Eur. Math. Soc. \textbf{10} (2008), no. 2, 315--349.

\bibitem{HWZ} E. Hamilton, H. Wilton and  P. Zalesskii,
\newblock \emph{Separability of double cosets and conjugacy classes in 3-manifold groups.}
\newblock J. London Math. Soc. (2) \textbf{87} (2013), no. 1, 269--288.



\bibitem{Heil} W. Heil, \emph{On $P^2$-irreducible $3$-manifolds.} Bull. Amer. Math. Soc. \textbf{75} (1969), 772--775.

\bibitem{Hsu-Wise}
T. Hsu and D.T. Wise,
\newblock {\em On linear and residual properties of graph products.}
\newblock Michigan Math. J. \textbf{46} (2) (1999) 251--259.

\bibitem{Hull}
M. Hull,
\newblock {\em Small cancellation in acylindrically hyperbolic groups.}
\newblock Groups. Geom. Dyn. (to appear). \texttt{arXiv:1308.4345}


\bibitem{Kram} D. Krammer,
\emph{The conjugacy problem for Coxeter groups.}
Groups Geom. Dyn. \textbf{3} (2009), no. 1, 71--171.

\bibitem{Lub}
A. Lubotzky,
\newblock {\em Normal automorphisms of free groups.}
\newblock J. Algebra \textbf{63} (1980), no. 2, 494--498.

\bibitem{Malc1} A.I. Mal'cev, \emph{On Homomorphisms onto finite groups} (Russian). Uchen. Zap. Ivanovskogo Gos. Ped.
Inst. \textbf{18} (1958), 49--60.

\bibitem{Malc2} A.I. Mal'cev, \emph{On the faithful representation of infinite groups by matrices.} (English. Russian original)
Am. Math. Soc., Transl., II. Ser. \textbf{45} (1965), 1--18; translation from Mat. Sb., N. Ser. \textbf{8} (50)(1940), no. 3, 405--422.

\bibitem{Martino}
A. Martino,
\newblock \emph{A proof that all Seifert 3-manifold groups and all virtual surface groups are
conjugacy separable.}
\newblock J. Algebra \textbf{313} (2007), no. 2, 773--781.

\bibitem{McCullough} D. McCullough,
\emph{Topological and algebraic automorphisms of 3-manifolds.} Groups of self-equivalences and related topics (Montreal, PQ, 1988), 102--113,
Lecture Notes in Math., 1425, Springer, Berlin, 1990.

\bibitem{Minasyan_classes}
A. Minasyan,
\newblock {\em Groups with finitely many conjugacy classes and their automorphisms}.
\newblock Comm. Math. Helv. {\bf 84} (2009), no. 2, pp. 259--296.

\bibitem{Minasyan_hcs}
A. Minasyan,
\newblock {\em Hereditary conjugacy separability of right angled Artin groups and its applications}.
\newblock Groups. Geom. Dyn. {\bf 6} (2012), no. 2, pp. 335--388.

\bibitem{MM} A. Martino, A. Minasyan, \emph{Conjugacy in normal subgroups of hyperbolic groups.} Forum Math. \textbf{24} (2012), no. 5, 889--909.

\bibitem{MO2}
A. Minasyan and D. Osin,
\newblock {\em Acylindrical hyperbolicity of groups acting on trees.}
\newblock Math. Ann. \textbf{362} (2015), no. 3-4, 1055--1105.

\bibitem{MO}
A. Minasyan and D. Osin,
\newblock {\em Normal automorphism of relatively hyperbolic groups}.
\newblock Trans. Amer. Math. Soc. {\bf 362}  (2010), pp. 6079--6103.

\bibitem{Obr}
 V.N. Obraztsov,
\newblock {\em Embedding into groups with well-described lattices of subgroups.}
\newblock Bull. Austral. Math. Soc.  \textbf{54} (1996), no. 2, 221--240.

\bibitem{Osin-acyl}
D. Osin,
\newblock \emph{Acylindrically hyperbolic groups.}
\newblock  Trans. Amer. Math. Soc. \textbf{368} (2016), no. 2, 851--888.


\bibitem{O06}
D. Osin,
\newblock {\em Elementary subgroups of relatively hyperbolic groups and bounded generation.}
\newblock Internat. J. Algebra Comput. {\bf 16} (2006), no. 1, 99--118.

\bibitem{Osin-periph_fil}
D. Osin,
\newblock {\em Peripheral fillings of relatively hyperbolic groups.}
\newblock Invent. Math. \textbf{167 }(2007), no. 2,
295--326.


\bibitem{Ols92}
A.Yu. Olshanskii,
\newblock {\em Periodic quotients of hyperbolic groups}.
\newblock  Mat. Sbornik {\bf 182} (1991), 4, 543--567 (in Russian),
\newblock English translation in Math. USSR
Sbornik {\bf 72} (1992), 2, 519--541.


\bibitem{Segal}
D. Segal,
\newblock {\em Polycyclic groups.}
\newblock Cambridge Tracts in Mathematics, 82.
Cambridge University Press, Cambridge, 1983. xiv+289 pp.

\bibitem{Sisto-contr} A. Sisto,
\newblock {\em Contracting elements and random walks.}
\newblock J. Reine Angew. Math. (to appear).  \texttt{arXiv:1112.2666}

\bibitem{Sisto}
A. Sisto,
\newblock {\em On metric relative hyperbolicity}.
\newblock Preprint (2012). \texttt{arXiv:1210.8081}

\bibitem{Wald}
F. Waldhausen,
\newblock {\em On irreducible 3-manifolds which are sufficiently large.}
\newblock Ann. of Math. (2) \textbf{87} (1968), 56--88.

\bibitem{Wehrfritz}
B.A.F. Wehrfritz,
\newblock \emph{Two remarks on polycyclic groups.}
\newblock Bull. London Math. Soc. \textbf{26} (1994), no. 6, 543--548.


\bibitem{Wise-QH}
D.T. Wise,
\newblock {\em The structure of groups with a quasiconvex hierarchy.}
\newblock Preprint (2011). Available from \\
\url{http://www.math.mcgill.ca/wise/papers.html}
\end{thebibliography}
\end{document}